\documentclass[11pt]{amsart}

\usepackage{amsmath,amsthm,indentfirst}
\usepackage{amssymb}
\usepackage{amsfonts}
\usepackage{amscd}
\usepackage[latin1]{inputenc}
\usepackage{hyperref}
\usepackage{ifthen, amsfonts, amssymb, graphicx, srcltx, mathrsfs, xfrac,
amsmath}
\usepackage{epsfig}
\usepackage{tikz}
\usepackage{pinlabel}
\usepackage{enumerate}
\input xy 
\xyoption{all}
\usepackage{color, soul}

\theoremstyle{plain}
\newtheorem*{maintheorem*}{Main Theorem}

\newtheorem*{thmd*}{Theorem 1.5}
\newtheorem*{thme*}{Theorem 1.6}
\newtheorem*{remark*}{Remark}
\newtheorem*{conjecture*}{Conjecture}
\newtheorem*{prop*}{Proposition}
\newtheorem{thm}{Theorem}[section]
\newtheorem{cor}[thm]{Corollary}
\newtheorem{lem}[thm]{Lemma}
\newtheorem{prop}[thm]{Proposition}

\theoremstyle{definition}

\newtheorem*{proofc*}{Proof of Theorem C}

\newtheorem{question}[thm]{Question}

\newtheorem{definition}[thm]{Definition}

\newtheorem{remark}[thm]{Remark}
\newtheorem{notation}[thm]{Notation}
\newtheorem{claim}[thm]{Claim}

\DeclareMathOperator{\Teich}{Teich}

\DeclareMathOperator{\base}{base}
\DeclareMathOperator{\tw}{tw}

\DeclareMathOperator{\diam}{diam}

\DeclareMathOperator{\hyp}{Hyp}
\DeclareMathOperator{\Area}{Area}
\DeclareMathOperator{\cyl}{cyl}
\DeclareMathOperator{\Ext}{Ext}

\newcommand{\ML}{\mathcal{ML}}
\newcommand{\PML}{\mathcal{PML}}
\newcommand{\EL}{\mathcal{EL}}

\numberwithin{equation}{section}

\begin{document}


\title[Limit sets of Teichm\"{u}ller geodesics]{Limit sets of Teichm\"{u}ller geodesics with minimal nonuniquely ergodic vertical foliation, II}

\date{\today}


\author[Jeff Brock]{Jeffrey Brock}
\address{Department of Mathematics, Brown University, Providence, RI, }
\email{brock@math.brown.edu}
\author[Chris Leininger]{Christopher Leininger}
\address{ Department of Mathematics, University of Illinois at Urbana-Champaign, 1409 W Green ST, Urbana, IL}
\email{clein@illinois.edu}
\author{Babak Modami}
\address{ Department of Mathematics, University of Illinois at Urbana-Champaign, 1409 W Green ST, Urbana, IL}
\email{bmodami@illinois.edu}
\author{Kasra Rafi}
\address{Department of Mathematics, University of Toronto, Toronto, ON}
\email{rafi@math.utoronto.edu}
\thanks{The first author was partially supported by NSF grant DMS-1207572,  the second author by NSF grant DMS-1510034, the third author by NSF grant DMS-1065872 and the fourth author by NCERC grant \# 435885.}





\begin{abstract}
Given a sequence of curves on a surface, we provide conditions which ensure that (1) the sequence is an infinite quasi-geodesic in the curve complex, (2) the limit in the Gromov boundary is represented by a nonuniquely ergodic ending lamination, and (3) the sequence divides into a finite set of subsequences, each of which projectively converges to one of the ergodic measures on the ending lamination.  The conditions are sufficiently robust, allowing us to construct sequences on a closed surface of genus $g$ for which the space of measures has the maximal dimension $3g-3$, for example.


We also study the limit sets in the Thurston boundary of Teichm\"{u}ller geodesic rays defined by quadratic differentials whose vertical foliations are obtained from the constructions mentioned above.  We prove that such examples exist for which the limit is a cycle in the $1$--skeleton of the simplex of projective classes of measures visiting every vertex.
\end{abstract}

\maketitle

\setcounter{tocdepth}{1}
\tableofcontents

\section{Introduction}
This paper builds on the work of the second and fourth author with Anna Lenzhen, \cite{nonuniqueerg}, in which  the authors construct a sequence of curves in the five-punctured sphere $S$ with the following properties (see \S\ref{sec : background} for definitions).  First, the sequence is a quasi-geodesic ray in the curve complex of $S$, and hence converges to some ending lamination $\nu$.  Second, $\nu$ is nonuniquely ergodic, and the sequence naturally splits into two subsequences, each of which converges to one of the ergodic measures on $\nu$ in the space of projective measured laminations.  Third, for any choice of measure $\bar \nu$ on $\nu$ and base point $X$ in Teichm\"uller space, the Teichm\"uller ray based at $X$ and defined by the quadratic differential with vertical foliation $\bar \nu$, accumulates on the entire simplex of measures on $\nu$ in the Thurston compactification.  The construction in \cite{nonuniqueerg} was actually a family of sequences depending on certain parameters.

In this paper we extract the key features of the sequences produced in
the above construction as a set of {\em local} properties for any
sequence of curves $\{\gamma_k\}_{k=0}^\infty$ on any surface, which
we denote $\mathcal P$; see \S\ref{sec : sequence} and
Definition~\ref{def : P} as well as \S\ref{sec : constructions} for
examples.  Here, ``local'' is more precisely {\em $m$--local} for some
$2 \leq m \leq \xi(S)$ (where $\xi(S) =\dim_{\mathbb C}(\Teich(S))$),
and means that the conditions in $\mathcal P$ involve relations
between curves contained subsets of the form
$\{\gamma_k,\ldots,\gamma_{k+2m}\}$ for $k \geq 0$.  We refer to the
number $m$ as the {\em subsequence counter}.  Most of the conditions in
$\mathcal P$ are stated in terms of intersection numbers, though they
also include information about twisting which is recorded in an auxiliary sequence $\{e_k\}_{k=0}^\infty \subset \mathbb N$. 

\begin{thm} \label{thm : subsurfcoeff}
For appropriate choices of parameters in $\mathcal P$, any sequence $\{\gamma_k\}_{k=0}^\infty \subset \mathcal C(S)$ satisfying $\mathcal P$ will be the vertices of a quasi-geodesic in $\mathcal C(S)$ and hence will limit to an ending lamination $\nu$ in $\partial \mathcal C(S) \cong \mathcal{EL}(S)$.

If $\mu = \gamma_0 \cup \ldots \cup \gamma_{m-1}$, then for any $k
\geq m$, the subsequence counter, we have
$$d_{\gamma_{k}}(\mu,\nu) \stackrel + \asymp e_{k}.$$
On the other hand, there is a constant $R>0$ with the property that for any proper subsurface $W\neq \gamma_{k}$ for any $k\in\mathbb{N}$ we have
$$d_{W}(\mu,\nu) < R.$$ 
\end{thm}
See Propositions~\ref{prop : P implies loc-to-global} and \ref{prop :
  anncoeff + coeffbd} for precise statements.  Here $d_W$ is the
projection coefficient for $W$ and $d_\gamma$ the projection
coefficient for (the annular neighborhood of) $\gamma$; see
\S\ref{subsec : sub-coeff}. 

Although the conditions in $\mathcal P$ only provide local information about intersection numbers, we can deduce estimates on intersection numbers between any two curves in the sequence from this; see Theorem~\ref{thm : intgkgi}.  From these estimates, we are able to promote the convergence $\gamma_k \to \nu$ in $\overline{\mathcal C(S)}$ into precise information about convergence in $\mathcal{PML}(S)$.  To state this, we note that the local condition depends on the subsequence counter $m$.  There are $m$ subsequences $\{\gamma_i^h\}_{i=0}^\infty$, for $h=0,\ldots,m-1$, defined by $\gamma_i^h = \gamma_{im+h}$.
\begin{thm} For appropriate choices of parameters in $\mathcal P$, and any sequence $\{\gamma_k\}_{k=0}^\infty \subset \mathcal C(S)$ satisfying $\mathcal P$, the ending lamination $\nu \in \EL(S)$ from Theorem \ref{thm : subsurfcoeff} is nonuniquely ergodic.  Moreover, if $m$ is the subsequence counter, then the dimension of the space of measures on $\nu$ is precisely $m$, and the $m$ subsequences $\{\gamma_i^h\}_{i=0}^\infty$ converge to $m$ ergodic measures $\bar \nu^h$ on $\nu$, for $h = 0,\ldots,m-1$, spanning the space of measures.
\end{thm}
\noindent For precise statements, see Theorems~\ref{thm : MLlimitgi}, \ref{thm : mutualsing}, and \ref{thm : all ergodic measures}.  

We note that for any nonuniquely ergodic lamination $\nu$, the space of measures is always the cone on the {\em simplex of measures on $\nu$}, denoted  $\Delta(\nu)$, which is projectively well-defined.  The vertices of $\Delta(\nu)$ are the {\em ergodic measures}, and the dimension of the space of measures is at most $\xi(S)$:   This follows from the fact that the Thurston symplectic form on the $2\xi(S)$--dimensional space $\ML(S)$ must restrict to zero on  the cone on $\Delta(\nu)$ since it is bounded above by the geometric intersection number, \cite[Ch. 3.2]{phtraintr}, and consequently must be at most half-dimensional (see also \cite[\S 1]{masurIET} and the reference to \cite{veechIET,katokINV}).  We note that the subsequence counter $m$ can also be at most $\xi(S)$, and the explicit constructions in \S\ref{sec : constructions} are quite flexible and provide examples with this maximal dimension, as well as examples with smaller dimensions.

As an application of these theorems, together with the main result of the first and third author in \cite{wprecurnue} and Theorem~\ref{thm : subsurfcoeff}, we have.
 \begin{cor}
Suppose $\nu$ is as in Theorem~\ref{thm : subsurfcoeff}.  Any
Weil-Petersson geodesic ray with forward ending lamination $\nu$ is
recurrent to a compact subset of the moduli space. 
 \end{cor}
 Here, the {\em ending lamination} of a Weil-Petersson geodesic ray is
 given as in
 \cite {bmm1,bmm2}.  The Corollary, which follows directly from
 \cite[Theorem~4.1]{wprecurnue} after observing that $\nu$ satisfies the
 condition of {\em nonannular bounded combinatorics} (see Proposition~\ref{prop : anncoeff + coeffbd}), provides greater
 insight into the class of Weil-Petersson ending laminations that
 violate {\em Masur's criterion}. In particular, these nonuniquely ergodic
 laminations determine {\em recurrent} Weil-Petersson geodesic
 rays, by contrast to the setting of Teichm\"uller geodesics where
 Masur's criterion \cite{masurcriterion} guarantees a Teichm\"uller geodesic with such a
 vertical foliation diverges.
 
For any lamination $\nu$ coming from a sequence $\{\gamma_k\}_{k=0}^\infty$ satisfying $\mathcal P$, as well as some additional conditions (see (\ref{eq : G1}) in $\S$\ref{sec : actinterval} and condition $\mathcal P(iv)$ in $\S$\ref{sec : limset}), we analyze the limit set of a Teichm\"uller geodesic ray defined by a quadratic differential with vertical foliation $\bar \nu$ supported on $\nu$.  To describe our result about the limiting behavior of this geodesic ray, we denote the simplex of the projective classes of measures supported on the lamination by $\Delta(\nu)$ in the space of projective measured foliations, viewed as the Thurston boundary of Teichm\"uller space.

\begin{thm} \label{thm:main limit theorem}
Suppose that $\nu$ is the limiting lamination of a sequence $\{\gamma_k\}_{k=0}^\infty$ satisfying the conditions $\mathcal P$, $\mathcal P(iv)$, and (\ref{eq : G1}). Let $\bar{\nu}=\sum_{h=0}^{m-1}x_{h}\bar{\nu}^{h}$ where $x_{h}>0$ for $h=0,...,m-1$, and $r:[0,\infty)\to\Teich(S)$ be a Teichm\"{u}ller geodesic ray with vertical measured lamination $\bar{\nu}$. Then the limit set of $r$ in the Thurston boundary is the simple closed curve in the simplex $\Delta(\nu)$ of measures on $\nu$ that is the concatenation of edges 
$$\bigl[[\bar{\nu}^{0}],[\bar{\nu}^{1}] \bigr] \cup \bigl[ [\bar{\nu}^1,\bar{\nu}^2] \bigr] \cup \ldots \cup \bigl[ [\bar{\nu}^{m-1}],[\bar{\nu}^{0}] \bigr].$$
\end{thm}

When $m \geq 3$, the theorem shows that there are Teichm\"{u}ller geodesics whose limit set does not contain any point in the interior of $\Delta(\nu)$.  In addition, it answers the following question raised by Jonathan Chaika: 
\begin{question}
Is the limit set of each Teichm\"{u}ller geodesic ray simply connected?  
\end{question}
For $m \geq 3$, the theorem shows that answer to this question is no. Namely, Teichm\"{u}ller geodesic rays with vertical measured lamination as above provide examples of geodesics with limit set being a topological circle, and hence not simply connected. \\

The results of this paper (as well as those of \cite{nonuniqueerg}) were inspired by work of Masur in \cite{2bdriesteich}, Lenzhen \cite{lenzhen-nolimit}, and Gabai \cite{gabaiendlamspace}.  In \cite{2bdriesteich} Masur showed that if $\nu$ is a uniquely ergodic foliation, then any Teichm\"uller ray defined by a quadratic differential with vertical foliation supported on $\nu$ limits to $[\nu]$ in the Thurston compactification.   Lenzhen \cite{lenzhen-nolimit} gave the first examples of Teichm\"uller rays which do not converge in the Thurston compactification.  Lenzhen's rays were defined by quadratic differentials with non-minimal vertical foliations, and in both \cite{nonuniqueerg} and \cite{nuechaikaetl}, nonconvergent rays defined by quadratic differentials with minimal vertical foliations were constructed.  The methods in these two papers are quite different, and as mentioned above, the approach taken in this paper is more closely related to that of \cite{nonuniqueerg}.  We also note the results of this paper, as well as \cite{lenzhen-nolimit,nonuniqueerg,nuechaikaetl}, are in sharp contrast to the work of Hackobyan and Saric in \cite{haksar} where it is shown that Teichm\"uller rays in the {\em universal Teichm\"uller space} always converge in the corresponding Thurston compactification.

Our example of nonuniquely ergodic laminations obtained from a sequence of curves are similar to those produced by Gabai's in \cite{gabaiendlamspace}.  On the other hand, our construction provides additional information, especially important are the estimates on intersection numbers and subsurface projections, that allow us to study the limiting  behavior of the associated Teichm\"uller rays.   For more on the history and results about the existence and constructions of nonuniquely ergodic laminations and the study of limit sets of Teichm\"{u}ller geodesics with such vertical laminations we refer the reader to the introduction of \cite{nonuniqueerg}.

\medskip
\noindent\textbf{Acknowledgement:} We would like to thank Howard Masur for illuminating 
conversations and communications as well as the anonymous referee for helpful suggestions. We also would like to thank Anna Lenzhen; 
her collaboration in the first paper was crucial  for the development of the current 
paper. Finally we would like to thank MSRI at Berkeley for hosting the program Dynamics on moduli spaces 
 in April 2015; where the authors had the chance to form some of the techniques of this paper.  

\section{Background} \label{sec : background}

We use the following notation throughout this paper.
\begin{notation}
Suppose $K\geq 1$ and $C\geq 0$ and $f,g:X\to\mathbb{R}$ are two functions. We write $f\stackrel{+}{\asymp}_{C}g$ if $f(x)-C\leq g(x)\leq f(x)+C$, $f\stackrel{*}{\asymp}_{K}g$ if $\frac{1}{K}f(x)\leq g(x)\leq Kf(x)$, and $f\asymp_{K,C}g$ if $\frac{1}{K}(f(x)-C)\leq g(x)\leq Kf(x)+C$. We also write $f\stackrel{*}{\prec}_{K}g$ if $f(x)\leq Kg(x)$, $f\stackrel{+}{\prec}_{C} g$ if $f(x)\leq g(x)+C$, and $f\prec_{K,C}g$ if $f(x)\leq Kg(x)+C$.
When the constants are known from the text we drop them from the notations.  Finally, we also write $f = O(g)$ if $f \stackrel * \prec g$.
\end{notation}

Let $S=S_{g,b}$ be an orientable surface of finite type with genus $g$ and $b$ holes (a hole can be either a puncture or a boundary component). Define the complexity of $S$ by $\xi(S)=3g-3+b$.  The main surface we will consider will have $\xi > 1$ and all holes will be punctures.  However, we will also be interested in subsurfaces and covers of the main surface, which can also have $\xi \leq 1$.  For surfaces $S$ with $\xi(S) \geq 1$, we will equip it with a reference metric, which is any complete, hyperbolic metric of finite area with geodesic boundary (if any).

\subsection{Curve complexes}\label{subsec : ccplx} 
For any surface $Y$, $\xi(Y) \geq 1$, the curve complex of $Y$, denoted by $\mathcal{C}(Y)$, is a flag complex whose vertices are the isotopy classes of simple closed curves on $Y$ that are essential, meaning non-null homotopic and non-peripheral.  For $\xi(Y) > 1$, a set of $k+1$ distinct isotopy classes of curves defines a $k$--simplex if any pair can be represented by disjoint curves. For $\xi(Y)=1$ ($Y$ is $S_{0,4}$ or $S_{1,1}$), the definition is modified as follows: a set of $k+1$ distinct isotopy classes defines a $k$--simplex if the curves can be represented intersecting twice (for $Y= S_{0,2}$) or once (for $Y = S_{1,1}$).  

The only surface $Y$ with $\xi(Y) < 1$ of interest for us is a compact annulus with two boundary components.  These arise as follows.  For any essential simple closed curve $\alpha$ on our main surface $S$, let $Y_\alpha$ denote the annular cover of $S$ to which $\alpha$ lifts.  The reference hyperbolic metric on $S$ lifts and provides a compactification of this cover by a compact annulus with boundary (which is independent of the metric).   The curve complex of $\alpha$, denoted $\mathcal C(Y_\alpha)$, or simply $\mathcal C(\alpha)$, has vertex set being the properly embedded, essential arcs in $Y_\alpha$, up to isotopy fixing the boundary pointwise.   A set of isotopy classes of arcs spans a simplex if any pair can be realized with disjoint interiors.  

Distances between vertices in $\mathcal C(Y)$ (for any $Y$) will be measured in the $1$--skeleton, so the higher dimensional simplices are mostly irrelevant.  Masur and Minsky \cite{mm1} proved that for any $Y$, there is a $\delta > 0$ so that $\mathcal C(Y)$ is $\delta$--hyperbolic.

For surfaces $Y$ with $\xi(Y) \geq 1$, we also consider the arc and curve complex $\mathcal{AC}(Y)$, defined in a similar way to $\mathcal C(Y)$.  Here vertices are isotopy classes of essential simple closed curves and essential, properly embedded arcs (isotopies need not fix the boundary pointwise), with simplices defined again in terms of disjoint representatives.   Arc and curve complexes are quasi-isometric to curve complexes, and so are also $\delta$--hyperbolic.  

Multicurves (respectively, multiarcs) are disjoint unions of pairwise non-isotopic essential simple closed curves (respectively, simple closed curves and properly embedded arcs).  Up to isotopy a multicurve (respectively, multiarc) determines, and is determined by, a simplex in $\mathcal C(S)$ (respectively, $\mathcal{AC}(S)$).  A marking $\mu$ is a pants decomposition $\base(\mu)$, called the base of $\mu$, together with a transversal curve $\beta_\alpha$, for each $\alpha \in \base(\mu)$, which is a curve minimally intersecting $\alpha$ and disjoint from $\base(\mu) - \alpha$.  A partial marking $\mu$ is similarly defined, but not every curve in the pants decomposition $\base(\mu)$ is required to have a transversal curve.

For more details on curve complexes, arc and curve complexes, and markings, we refer the reader to \cite{mm1}.

\begin{remark}
When the number $\xi(S)$ is at least $1$, it is equal to the number of curves in a pants decomposition.  When all the holes of $S$ are punctures, $\xi(S)$ is also the complex dimension of Teichm\"{u}ller space of $S$.
\end{remark}

\subsection{Laminations and foliations.} 
A lamination will mean a geodesic lamination (with respect to the reference metric if no other metric is specified), and a measured lamination is a geodesic lamination $\nu$, called the support, with an invariant transverse measure $\bar \nu$.  We will often refer to a measured lamination just by the measure $\bar \nu$ (as this determines the support $\nu$).   The space of all measured laminations will be denoted $\mathcal{ML}(S)$, and for any two metrics, the resulting spaces of measured laminations are canonically identified.  By taking geodesic representatives, simple closed curves and multicurves determine geodesic laminations.  Weighted simple closed curves and multicurves determine measured laminations are dense in $\mathcal{ML}(S)$, and the geometric intersection number extends to a continuous, bi-homogeneous function
\[ i \colon \mathcal{ML}(S) \times \mathcal{ML}(S) \to \mathbb R.\]
By a measured foliation on $S$ we will mean a singular measured foliation with prong singularities of negative index (and at punctures, filling in the puncture produces a $k$--prong singularity with $k \geq 1$).  When convenient, a measured foliation may be considered only defined up to measure equivalence, and the space of measure equivalence classes of measured foliations is denoted $\mathcal{MF}(S)$.  The spaces $\mathcal{MF}(S)$ and $\mathcal{ML}(S)$ are canonically identified, and we will frequently not distinguish between measured laminations and measured foliations.  A foliation or lamination is uniquely ergodic if it supports a unique (up to scaling) transverse measure, or equivalently, if the first return map to (the double of) any transversal is uniquely ergodic.  Otherwise it is nonuniquely ergodic.  We write $\mathcal{PML}(S)$ and $\mathcal{PMF}(S)$ for the quotient spaces, identifying measured laminations or foliations that differ by scaling the measure.
See \cite{phtraintr,notesonthurston,FLP,thurston:GT,fol=lam} for complete definitions, detailed discussion, and equivalence of $\mathcal{MF}(S)$ and $\mathcal{ML}(S)$.

\subsection{Gromov boundary of the curve complex}
A lamination $\nu$ on $S$ is called an ending lamination if it is minimal (every leaf is dense) and filling (every simple closed geodesic on the surface nontrivially, transversely intersect $\nu$). Every ending lamination admits a transverse measure, and we let $\mathcal{EL}(S)$ denote the space of all ending laminations.  This is the quotient space of the subspace of $\mathcal{ML}(S)$ consisting of measured laminations supported on ending laminations, by the map which forgets the measures.  The following theorem of Klarreich \cite{bdrycc} identifies the Gromov boundary of $\mathcal C(S)$ with $\mathcal{EL}(S)$.

\begin{thm}\textnormal{(Boundary of the curve complex)}\label{thm : bdrycc}
  There is a homeomorphism $\Phi$ from the Gromov boundary of
  $\mathcal{C}(S)$ equipped with its standard topology to
  $\mathcal{EL}(S)$. 
  
  Let $\{\gamma_{k}\}_{k=0}^{\infty}$ be a sequence of
  curves in $\mathcal{C}_{0}(S)$ that converges to a point $x$ in the
  Gromov boundary of $\mathcal{C}(S)$. Regarding each $\gamma_{k}$ as
  a projective measured lamination, any accumulation point of the sequence
  $\{\gamma_{k}\}_{k=0}^{\infty}$ in $\mathcal{PML}(S)$ is supported on $\Phi(x)$.
\end{thm}

We will use this theorem throughout to identify points in $\partial \mathcal C(S)$ with ending laminations in $\mathcal{EL}(S)$.

\subsection{Subsurface coefficients} \label{subsec : sub-coeff} An {\em essential subsurface} $Y$ of a surface $Z$ with $\xi(Y) \geq 1$ is a closed, connected, embedded subsurface whose boundary components are either essential curves in $Z$ or boundary component of $Z$, and whose punctures are punctures of $Z$.  All such subsurfaces are considered up to isotopy, and we often choose representatives that are components of complements of small neighborhoods of simple closed geodesics, which therefore have minimal, transverse intersection with any lamination.
The only essential subsurfaces $Y$ of $Z$ with $\xi(Y) < 1$ are not actually subsurfaces at all, but rather such a $Y$ is the compactified annular covers $Y_\alpha$ of $Z$ associated to a simple closed curve $\alpha$ in $Z$.  We sometimes confuse an annular neighborhood of $\alpha$ with the cover $Y_\alpha$ (hence the reference to it as a subsurface) when convenient.  We will always write $Y \subseteq Z$ to denote an essential subsurface, even when it is not, strictly speaking, a subset of $Z$.

Let $Y\subseteq Z$ be an essential non-annular subsurface and $\lambda$ a lamination (possibly a multicurve) and we define the subsurface projection of $\lambda$ to $Y$. Represent $Y$ as a component of the complement of a very small neighborhood of geodesic multicurve. If $\lambda\cap Y=\emptyset$, then define $\pi_{Y}(\lambda)=\emptyset$. Otherwise, $\pi_{Y}(\lambda)$ is the union of all curves which are (i) simple closed curve components of $Y \cap \lambda$ or (ii) essential components of $\partial N(a \cup \partial Y)$, where $a \subset \lambda \cap Y$ is any arc, and $N(a \cup \partial Y)$ is a regular neighborhood of the union.  If $Y_\alpha$ is an essential annular subsurface, then $\pi_{Y_\alpha}(\lambda)$, or simply $\pi_\alpha(\lambda)$, is defined as follows.  For any component of the preimage of $\lambda$ in the annular cover corresponding to $\alpha$, the closure is an arc in $Y_\alpha$, and we take the union of all such arcs that are essential (that is, the arcs that connect the two boundary components).

For a marking $\mu$ (or partial marking), if $Y = Y_\alpha$ is an annulus with core curve $\alpha\in\base(\mu)$, then $\pi_{Y}(\mu) = \pi_\alpha(\beta_\alpha)$, where $\beta_\alpha$ is the transverse curve for $\alpha$ in $\mu$.  Otherwise, $\pi_{Y}(\mu)=\pi_{Y}(\base(\mu))$.  For any lamination or partial marking $\lambda$ and any essential subsurface $Y$, $\pi_Y(\lambda)$ is a subset of diameter at most $2$.


Let $\mu,\mu'$ be laminations, multiarcs, or partial markings on $Z$ and $Y \subset Z$ an essential subsurface. The {\it $Y$--subsurface coefficient} of $\mu$ and $\mu'$ is defined by
\begin{equation}d_{Y}(\mu,\mu'):=\diam_{\mathcal{C}(Y)}(\pi_{Y}(\mu)\cup\pi_{Y}(\mu'))\end{equation}
\begin{remark}
The subsurface coefficient is sometimes alternatively defined as the (minimal) distance between $\pi_{Y}(\mu)$ and $\pi_{Y}(\mu')$.  Since the diameter of the projection of any marking or lamination is bounded by $2$, these definitions differ by at most $4$.  The definition we have chosen satisfies a triangle inequality (when the projections involved are nonempty), which is particular useful for our purposes.
\end{remark}

The following lemma provides an upper bound for a subsurface coefficient in terms of intersection numbers. 
\begin{lem}\label{lem : id}\cite[$\S 2$]{mm2}
Given curves $\alpha,\beta\in \mathcal{C}(S)$, for any essential subsurface $Y\subseteq S$ we have
$$d_{Y}(\alpha,\beta)\leq 2i(\alpha,\beta)+1.$$
When $Y$ is an annular subsurface the above bound holds with multiplicative factor $1$.
\end{lem}
\begin{remark}
The bound in the above lemma can be improved to $\prec \log i(\alpha,\beta)$ for $\xi(Y) \geq 1$, but the bound given is sufficient for our purposes.
\end{remark}



The following result equivalent to \cite[Corollary D]{compteichlip} provides for a comparison between the logarithm of intersection number and sum of subsurfaces coefficients.  For a pair of markings $\mu,\mu'$, the intersection number $i(\mu,\mu')$ is defined to be the sum of the intersection numbers of the curves in $\mu$ with those in $\mu'$.

\begin{thm}\label{thm : i=dY} 
Given $A>0$ sufficiently large, there are constants so that for any two multi-curves, multi-arcs or markings $\mu$ and $\mu'$ we have
$$\log i(\mu,\mu') \asymp \sum_{\substack{W\subseteq Y,\\ \text{non-annular}}} \{d_{W}(\mu,\mu')\}_{A}+\sum_{\substack{W\subseteq Y,\\ \text{annular}}} \log\{ d_{W}(\mu,\mu')\}_{A}.$$
Where $W$ is so that $\mu,\mu'\pitchfork W$.
\end{thm}

In this theorem, $\{ \cdot \}_{A}$ is a cut-off function defined by $\{x\}_A = x$ if $x\geq A$, and $\{x\}_{A}=0$ if $x<A$.

\begin{notation}  \label{notation : cutting}
Given a lamination or a partial marking $\mu$ and subsurface $Y$ we say that $\mu$ and $Y$ overlap, writing $\mu \pitchfork Y$ if $\pi_Y(\mu) \neq \emptyset$.  For any marking $\mu$ and any subsurface $Y$, we have $\mu \pitchfork Y$.  Given two subsurfaces $Y$ and $Z$, if $\partial{Y}\pitchfork Z$ and $\partial{Z}\pitchfork Y$ then we say that $Y$ and $Z$ overlap, and write $Y\pitchfork Z$.
\end{notation}

The inequality first proved by J.~Behrstock \cite{beh} relates subsurface coefficients for overlapping subsurfaces.
\begin{thm}[Behrstock inequality]\label{thm : behineq}
There is a constant $B_{0}>0$ so that given a partial marking or lamination $\mu$ and subsurfaces $Y$ and $Z$ satisfying $Y\pitchfork Z$ we have
$$\min\{d_{Y}(\partial{Z},\mu),d_{Z}(\partial{Y},\mu)\}\leq B_{0}.$$
whenever $\mu \pitchfork Y$ and $\mu \pitchfork Z$.
\end{thm}
\begin{remark} \label{rem : Behrstock sharp}  As shown in \cite{mangahas}, the constant $B_0$ can be taken to be $10$.  In fact, if one projection is at least $10$, then the other is $\leq 4$.
\end{remark}

The following theorem is a straightforward consequence of the Bounded Geodesic Image Theorem \cite[Theorem 3.1]{mm2}.

\begin{thm}\textnormal{(Bounded geodesic image)}\label{thm : bddgeod}
Given $k\geq 1$ and $c\geq 0$, there is a constant $G>0$ with the following property. Let $Y\subsetneq S$ be a subsurface. Let $\{\gamma_{k}\}_{i=0}^{\infty}$ be a $1-$Lipschitz $(k,c)-$quasi-geodesic in $\mathcal{C}(S)$ so that $\gamma_{k}\pitchfork Y$ for all $i$. Then $\diam_{Y}(\{\gamma_{k}\}_{i=0}^{\infty})\leq G$.
\end{thm}
\medskip

\subsection{Teichm\"{u}ller theory}\label{teichth}
We assume that any holes of $S$ are punctures.  The Teichm\"uller space of $S$, $\Teich(S)$, is the space of equivalence classes of marked complex structures $[f \colon S \to X]$, where $f$ is an orientation preserving  homeomorphism to a finite type Riemann surface $X$, where $(f \colon S \to X) \sim (g \colon S \to Y)$ if $f \circ g^{-1}$ is isotopic to a conformal map.  We often abuse notation, and simply refer to $X$ as a point in Teichm\"uller space, with the equivalence class of marking implicit.  We equip $\Teich(S)$ with the Teichm\"uller metric, whose geodesics are defined in terms of quadratic differentials.

Let $X$ be a finite type Riemann surface and let $T^{(1,0)*}X$ be the holomorphic cotangent bundle of $X$.  A quadratic differential $q$ is a nonzero, integrable, holomorphic section of the bundle $T^{(1,0)*}X\otimes T^{(1,0)*}X$.  In local coordinates $q$ has the form $q(z)dz^{2}$ where $q(z)$ is holomorphic function.  Changing to a different coordinate $w$, $q$ changes by the square of the derivative, and is thus given by $q(z(w)) (\frac{\partial w}{\partial z})^2 dw^2$.  The integrability condition is only relevant when $X$ has punctures, in which case it guarantees that $q$ has at worst simple poles at the punctures.

In local coordinates away from zeros of $q$ the quadratic differential $q$ determines the $1-$form $\sqrt{q(z)dz^{2}}$. Integrating this $1$--form determines {\it a natural coordinate} $\zeta=\xi+i\eta$. Then the trajectories of $d\xi\equiv 0$ and $d\eta\equiv 0$, respectively, determine the horizontal and vertical foliations of $q$ on $X$. Integrating $|d\xi|$ and $|d\eta|$ determines transverse measures on vertical and horizontal foliations, respectively.  These extend to measured foliations on the entire surface $S$ with singularities at the zeros.  Using the identification $\mathcal{MF}(S) \cong \mathcal{ML}(S)$, we often refer to the vertical and horizontal measured laminations of $q$.

Now given a quadratic differential $q$ on $X$, the associated Teichm\"{u}ller geodesic is determined by the family of Riemann surfaces $X_t$ defined by local coordinates $\zeta_{t}=e^{t}\xi +e^{-t}\eta$ where $\zeta=\xi+i\eta$ is a natural coordinate of $q$ at $X$ and $t\in\mathbb{R}$.  Every Teichm\"uller geodesic ray based at $X$ is determined by a quadratic differential $q$ on $X$.  See \cite{gardiner} for details on Teichm\"uller space and quadratic differentials.

\subsection{The Thurston compactification}
Given a point $[f \colon S \to X]$ in $\Teich(S)$ and a curve $\alpha$, the hyperbolic length of $\alpha$ at $[f \colon S \to X]$ is defined to be hyperbolic length of the geodesic homotopic to $f(\alpha)$ in $X$.  Again abusing notation and denoting the point in $\Teich(S)$ by $X$, we write the hyperbolic length simply as $\hyp_X(\alpha)$.  The hyperbolic length function extends to a continuous function
\[ \hyp_{(\cdot)}(\cdot) \colon \Teich(S) \times \mathcal{ML}(S) \to \mathbb R.\]

The Thurston compactification, $\overline{\Teich(S)} = \Teich(S) \cup \mathcal{PML}(S)$ is constructed so that a sequence $\{X_n\} \subseteq \Teich(S)$ converges to $[\bar \nu] \in \mathcal{PML}(S)$ if and only if
\[ \lim_{n \to \infty} \frac{\hyp_{X_n}(\alpha)}{\hyp_{X_n}(\beta)} = \frac{i(\bar \nu,\alpha)}{i(\bar \nu,\beta)}\]
for all simple closed curves $\alpha,\beta$ with $i(\bar \nu,\beta) \neq 0$.
See \cite{bonahon-teich,FLP} for more details on the Thurston compactification.

\subsection{Some hyperbolic geometry}

Here we list a few important hyperbolic geometry estimates.  For a hyperbolic metric $X \in \Teich(S)$ and a simple closed curve $\alpha$, in addition to the length $\hyp_X(\alpha)$, we also have the quantity $w_X(\alpha)$, the {\em width of $\alpha$ in $X$}.  This is the width of a maximal embedded tubular neighborhood of $\alpha$ in the hyperbolic metric $X$--that is, $w_X(\alpha)$ is the maximal $w$ so that the open $w/2$--neighborhood of $\alpha$ is an annular neighborhood of $\alpha$.
The Collar Lemma (see e.g.~\cite[\S 4.1]{buser}) provides a lower bound on the width
\begin{lem} \label{lem : collar lemma}
For any simple closed curve $\alpha$, we have
\[ w_X(\alpha) \geq 2\sinh^{-1}\Big( \frac{1}{\sinh(\hyp_X(\alpha)/2)} \Big).\]
Consequently, 
\begin{equation}
\label{eq : w} w_X(\alpha) \stackrel+ \asymp 2 \log\Big(\frac{1}{\hyp_X(\alpha)}\Big)
\end{equation}
\end{lem}
The second statement comes from the first, together with an easy area argument.  The implicit additive error depends only on the topology of $S$.  

We also let $\epsilon_0 > 0$ be the Margulis constant, which has the property that any two hyperbolic geodesics of length at most $\epsilon_0$ must be embedded and disjoint.

\subsection{Short markings}
For $L>0$ sufficiently large, an $L$--bounded length marking at $X\in\Teich(S)$ (or $L$--short marking) is a marking with the property that any curve in $\base(\mu)$ has hyperbolic length less than $L$, and so that for each $\alpha\in\base(\mu)$, the transversal curve to $\alpha$ has smallest possible length in $X$. Choosing $\epsilon$ sufficiently large (larger than the Bers constant of the surface) the distance between any two points in Teichm\"uller space can be estimated up to additive and multiplicative error in terms of the subsurface coefficients of the short markings at those points, together with the lengths of their base curves; see \cite{rcombteich}.

\section{Sequences of curves}\label{sec : sequence}

Over the course of the next three sections we will provide general conditions on a sequence of curves which guarantee that any accumulation point in $\PML(S)$ of this sequence is a nonuniquely ergodic ending lamination.  In \cite[\S 9]{gabaiendlamspace}, Gabai describes a construction of minimal filling nonuniquely ergodic geodesic laminations. The construction is topological in nature. Our construction in this paper and that of \cite{nonuniqueerg} can be considered as quantifications of Gabai's construction where the estimates for intersection numbers are computed explicitly.  These estimates allow us to provide more detailed information about the limits in $\PML(S)$ as well as limiting behavior of associated Teichm\"uller geodesics.

In this section we state conditions a sequence of curves can satisfy, starting with an example, and describe a useful way of mentally organizing them.  The conditions are motivated by the examples in \cite{nonuniqueerg}, and so we recall that construction to provide the reader concrete examples to keep in mind.  A more robust construction that illustrates more general phenomena is detailed in \S\ref{sec : constructions}.

Throughout the rest of this paper $\{e_{k}\}_{k=0}^{\infty}$ is an increasing sequence of integers satisfying  
\begin{equation}\label{eq : ek}e_{k+1}\geq ae_{k} \;\text{for any}\; k\geq 0.\end{equation}
where $a>1$.  Consequently, for all $l < k$, we have $e_k \geq a^{k-l}e_l$.  

\subsection{Motivating example.}

The motivating examples are sequences of curves in $S_{0,5}$, the five-punctured sphere.  We view this surface as the double of a pentagon minus its vertices over its boundary.  This description provides an obvious order five rotational symmetry $\rho$ obtained by rotating the pentagon counter-clockwise by an angle $4\pi/5$.  Let $\gamma_0$ be a curve which is the boundary of a small neighborhood of one of the sides of the pentagon and let $\gamma = \rho^2(\gamma_0)$ (see Figure~\ref{Fig:fivecurves}).   Write  $\mathcal D = \mathcal D_\gamma$ for the positive Dehn twist about $\gamma$.

Now define $\gamma_k$ to be the image of $\gamma_0$ under a composition of powers of $\mathcal D$ and $\rho$ by the formula:
\[ \gamma_k = \mathcal D^{e_2} \rho \mathcal D^{e_3} \rho \cdots \mathcal D^{e_k} \rho \mathcal D^{e_{k+1}} \rho(\gamma_0).\]
The first five curves, $\gamma_0,\ldots,\gamma_4$ in the sequence are shown in Figure~\ref{Fig:fivecurves}.

\begin{figure}[htb]
\labellist
\small\hair 2pt
 \pinlabel {$\gamma_0$} [ ] at 55  32
 \pinlabel {$\gamma_1$} [ ] at 185 60
 \pinlabel {$\gamma = \gamma_2$} [ ] at 308 49
 \pinlabel {$\gamma_3$} [ ] at 440 32
 \pinlabel {$\gamma_4$} [ ] at 560 32 
 \pinlabel {$2e_2$} [ ] at 520 90
\endlabellist
\begin{center}
\includegraphics[width=5in,height=.9in]{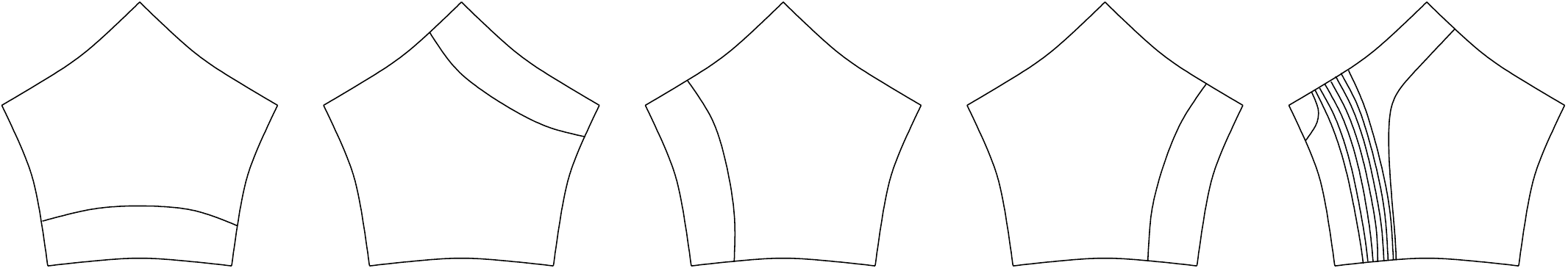}
\caption{The curves $\gamma_0,\ldots,\gamma_4$ in $S_{0,4}$.  Any five consecutive curves $\gamma_{k-2},\ldots,\gamma_{k+2}$ differ from 
those shown here by a homeomorphism, and replacing $e_2$ by $e_k$.}
\label{Fig:fivecurves}
\end{center}
\end{figure}
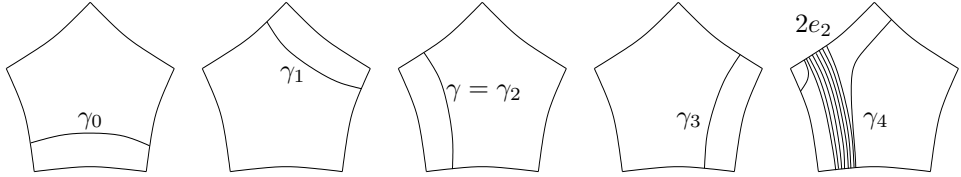

Observe that for any $k \geq 3$, the four consecutive curves $\gamma_{k-2},\ldots,\gamma_{k+1}$ are just the image of $\gamma_0,\ldots,\gamma_3$ under the homeomorphism
\[ \Phi_{k-1} = \mathcal D^{e_2} \rho \cdots \mathcal D^{e_{k-1}} \rho. \]
Furthermore, the next curve in the sequence, $\gamma_{k+2}$, is the image of $\mathcal D^{e_k} \rho(\gamma_3)$.  In particular, up to homeomorphism, any five consecutive curves $\gamma_{k-2},\ldots,\gamma_{k+2}$ in the sequence appear as in Figure~\ref{Fig:fivecurves} with $e_2$ replaced by $e_k$.

\subsection{Intersection conditions}
We now describe the general conditions, and verify that the above sequence of curves satisfies them.  To begin, we fix positive integers $b_1 \leq b \leq b_2$.  We will also assume that $e_0  > E + G$ (and hence $e_k > a^k (E+G)$ for all $k$), where $G$ is the constant from Theorem~\ref{thm : bddgeod} and $E$ is the constant in Theorem~\ref{thm : locglob} below.  For the examples in $S_{0,5}$ described above, we will have $b = b_1 = b_2 = 2$.

In the next definition, $\mathcal D_\gamma$ is the Dehn twist in a curve $\gamma$.

\begin{definition}\label{def : P}
Suppose $m \leq \xi(S)$, and assume $b,b_1,b_2,a,$ and $\{e_k\}_{k=0}^\infty$ are as above.  We say that a sequence of curves $\{\gamma_{k}\}_{k=0}^{\infty}$ on $S$ satisfies $\mathcal P$ if the following properties hold for all $k \geq 0$:
\begin{enumerate}[(i)]
\item $\gamma_{k},...,\gamma_{k+m-1}$ are pair-wise disjoint and distinct,
\item\label{int : fill} $\gamma_{k},...,\gamma_{k+2m-1}$ fill the surface $S$, and 
\item\label{int : gk+m} $\gamma_{k+m} = \mathcal D^{e_k}_{\gamma_k}(\gamma_{k+m}')$, where $\gamma'_{k+m}$ is a curve such that
\[ i(\gamma'_{k+m},\gamma_j) \left\{ \begin{array}{ll} \in [b_1,b_2] & \mbox{for } j\in\{k - m,...,k-1\} \\
= b & \mbox{for } j = k\\
= 0 & \mbox{for } j\in\{k+1,...,k+m-1\}  \end{array} \right. \]
(here we ignore any situation with $j < 0$).
\end{enumerate}
We will wish to impose some additional constraints on the constant $a$ (specifically, we will require it to be chosen so that (\ref{eq : restriction on a}) holds), and so in the notation we sometimes express the dependence on $a$ writing $\mathcal P = \mathcal P(a)$.  Of course, $\mathcal P$ depends on the choice of constants $b_1 \leq b \leq b_2$ and the sequence $\{e_k\}$, but we will impose no further constraints on the $b$ constants, and the conditions on $\{e_k\}$ depend on $a$.
\end{definition}

Here we verify that the sequence of curves on $S_{0,5}$ described above satisfies these conditions with $m = 2$.  Note that the conditions are all ``local'', meaning that they involve a consecutive sequence of at most $2m+1$ curves---for our example, that's a sequence of at most $5$ consecutive curves.  As noted above, any five consecutive curves $\gamma_{k-2},\ldots,\gamma_{k+2}$ differ from those in Figure~\ref{Fig:fivecurves} by applying the homeomorphism $\Phi_{k-1} = \mathcal D^{e_2} \rho \cdots \mathcal D^{e_{k-1}} \rho$, and changing $e_2$ to $e_k$.  From this, it is straight forward to verify that this sequence satisfies these conditions.

Since $m = 2$, condition (i) says that $2$ consecutive curves are disjoint, while (ii) says that four consecutive curves fill $S_{0,5}$. Note that (i) is clearly true for $\gamma_0,\gamma_1$ and (ii) for $\gamma_0,\ldots,\gamma_3$.  Since any two or four consecutive curves differ from these by a homeomorphism, conditions (i) and (ii) hold for all $k$.

Finally, note that $\gamma_4 = \mathcal D_\gamma^{e_2} \rho(\gamma_3)$, and so setting $\gamma_4' = \rho(\gamma_3)$ and observing that $\gamma = \gamma_2$, (iii) clearly holds for $k=2$ by inspection of Figure~\ref{Fig:fivecurves}.   The case for general $k$ follows from this figure as well, after applying $\Phi_{k-1}$.  Specifically, $\gamma_{k+2}$ is obtained from $\rho(\gamma_3)$ by applying $\Phi_{k-1} \mathcal D_{\gamma_2}^{e_k}$, or equivalently, setting $\gamma_{k+2}' = \Phi_{k-1}(\rho(\gamma_3))$
\[ \gamma_{k+1} =  \Phi_{k-1} \mathcal D_{\gamma_2}^{e_k} \Phi_{k-1}^{-1} (\Phi_{k-1}(\rho(\gamma_3))) = \mathcal D_{\Phi_{k-1}(\gamma_2)}^{e_k}(\gamma_{k+2}') = \mathcal D_{\gamma_k}^{e_k}(\gamma_{k+2}'). \]
Because $\gamma_{k-2},\gamma_{k-1},\gamma_k,\gamma_{k+1},\gamma_{k+2}',\gamma_{k+2}$ are the images of $\gamma_0,\gamma_1,\gamma_2,\gamma_3,\gamma_4',\gamma_4$, respectively, under $\Phi_{k-1}$, (iii) follows for general $k$ by inspection of Figure~\ref{Fig:fivecurves}.\\

Returning to the general case, we elaborate a bit on the properties in $\mathcal P$.  First we make a simple observation.
\begin{lem} \label{lem : length 2m intervals}
For every $j,k \geq 0$ with $j \in \{k-m+1,\ldots,k\}$, we have $i(\gamma_{k+m},\gamma_j)\in[b_{1},b_{2}]$ with $i(\gamma_{k+m},\gamma_k) = b$.  \end{lem}
\begin{proof}  Since $\gamma_{k+m} = \mathcal D^{e_k}_{\gamma_k}(\gamma_{k+m}')$ and $\mathcal D_{\gamma_k}(\gamma_j) = \gamma_j$  (because $i(\gamma_j,\gamma_k) = 0$), it follows that
\[  i(\gamma_{k+m}, \gamma_j) = i(\mathcal D^{-e_k}_{\gamma_k}(\gamma_{k+m}),\mathcal D^{-e_k}_{\gamma_k}(\gamma_j)) =  i(\gamma_{k+m}',\gamma_j)  \in [b_1,b_2]  \]
proving the first statement.  For the special case $j=k$, $i(\gamma_{k+m}',\gamma_k) = b$, and the second statement follows.
\end{proof}

\subsection{Visualizing the conditions of $\mathcal P$.}
The conditions imposed in $\mathcal P$ involve intervals of length $m$ and $2m$, as well as mod $m$ congruence conditions.  It is useful to view the tail of the sequence starting at any curve $\gamma_i$ (for example, when $i = 0$ this is the entire sequence), in the following form
\begin{equation} \label{eq:how we view it} \xymatrix{ & \gamma_i \ar[r] & \gamma_{i+1} \ar[r] & \cdots \ar[r] & \gamma_{i+m-1} \ar `r[d] `[l] `^d[llll] `^r[d] [dlll]   \\
& \gamma_{i+m} \ar[r] & \gamma_{i+m+1} \ar[r] & \cdots \ar[r] & \gamma_{i+2m-1} \ar `r[d] `[l] `^d[llll] `^r[d] [dlll]  \\
& \gamma_{i+2m} \ar[r] & \gamma_{i+2m+1} \ar[r] & \cdots & }
\end{equation}
From the first condition of $\mathcal P$, all curves in any row are pairwise disjoint. Lemma~\ref{lem : length 2m intervals} tells us that $\gamma_i$ intersects the curve directly below it $b$ times and it intersects everything in the row directly below it between $b_1$ and $b_2$ times.  The second condition in $\mathcal P$ tells us that any two consecutive rows fill $S$.  The third condition (part of which is used in the proof of Lemma~\ref{lem : length 2m intervals}), can be thought of as saying that going straight down two rows from $\gamma_i$ to $\gamma_{i+2m}$ gives a curve that ``almost'' differs by the power of the Dehn twist $\mathcal D_{\gamma_{i+m}}^{e_{i+m}}$.  To understand this interpretation, note that $\gamma_{i+2m}'$ and $\gamma_{i+2m}$ differ precisely by this power of a twist, while on the other hand, each of $\gamma_{i+2m}'$ and $\gamma_i$ have intersection number at most $b_2$ with the filling set $\gamma_i,\ldots,\gamma_{i+2m-1}$ (which we view as saying that $\gamma_i$ and $\gamma_{i+2m}'$ are ``similar'').

\section{Curve complex quasi-geodesics}\label{sec : locglob}

The purpose of this section is to provide general conditions (Theorem~\ref{thm : locglob}) on a sequence of subsurfaces in terms of subsurface coefficients of consecutive elements which guarantee that their boundaries define a quasi-geodesic in the curve complex of the surface.  Appealing to Theorem~\ref{thm : bdrycc}, we deduce that such sequences determine an ending lamination.   We end by proving that   a sequence of curves satisfying $\mathcal P$ are core curves of annuli satisfying the conditions of Theorem~\ref{thm : locglob}, and hence are vertices of a quasi-geodesic in $\mathcal C(S)$ defining an ending lamination $\nu \in \mathcal{EL}(S)$.

Variations of this result appeared in \cite{mangahas}, \cite{CLM}, \cite{wpbehavior}, \cite{nonuniqueerg}, and \cite{wprecurnue} for example. Here our conditions only involve the intersection pattern and projection coefficients of fixed number of consecutive subsurfaces along the sequence. In this sense these are local conditions. 

\begin{thm}\label{thm : locglob}\textnormal{(Local to Global)}
Given a surface $S$ and $2 \leq m \leq \xi(S)$, there are constants $E> C>0$ with the following properties. Let $\{Y_{k}\}_{k= 0}^{\infty}$ be a sequence of subsurfaces of $S$. Suppose that for each integer $k\geq 0$,
\begin{enumerate}
\item\label{lg : locqg} The multi-curves $\partial{Y}_{k},...,\partial{Y}_{k+m-1}$ are pairwise disjoint,
\item \label{lg : locint}$Y_{k}\pitchfork Y_{j}$ for all $j\in\{k+m,...,k+2m-1\}$, and
\item \label{lg : loclargesubsurf}$d_{Y_{k}}(\partial{Y}_{j},\partial{Y}_{j'}) > E$  for any $j\in\{k+m,...,k+2m-1\}$ and any $j'\in\{k-2m+1,...,k-m\}$.

\hspace{-1.8cm} Then for every $j,j',k$ with $j\geq k+m$ and $j' \leq k-m$ we have
 \begin{equation}\label{eq : int} Y_{k}\pitchfork Y_{j} \;\text{and}\; Y_{k}\pitchfork Y_{j'} \end{equation}
and
\begin{equation}\label{eq : locglob}d_{Y_{k}}(\partial{Y}_{j},\partial{Y}_{j'})\geq d_{Y_{k}}(\partial{Y}_{k-m},\partial{Y}_{k+m})-C. \end{equation}

\hspace{-1.8cm} Furthermore suppose that for some $n\geq 1$ and all $k \geq 0$
\item \label{lg : locfill} the multi-curves $\partial{Y}_{k}, ....,\partial{Y}_{k+2n-1}$ fill $S$. 
\end{enumerate}
Then for any two indices $k,j\geq 0$ with $|k-j|\geq 2n$ we have
\begin{equation}\label{eq : lbdS} d_S(\partial Y_j,\partial Y_k) \geq \frac{|k-j|}{4n} - \left( \frac{m}{2n} + 1 \right).\end{equation}
\end{thm}
In the hypotheses (as well as the conclusions) of this theorem, we ignore any condition in which there is a negative index.

\begin{proof}  Set the constants
$$C=2B_{0}+4+G \quad  \mbox{ and } \quad E=C+B_{0}+G+4.$$ 
Here $B_{0}$ is the constant from Theorem \ref{thm : behineq} (Behrstock inequality) and $G$ is the constant from Theorem~\ref{thm : bddgeod} (Bounded geodesic image theorem) for a geodesic (i.e.~$k=1$, $c =0$).
We prove (\ref{eq : int}) and (\ref{eq : locglob}) simultaneously by a double induction on $(j-k,k-j')$. 

For the base of induction, suppose that $m\leq k-j'\leq 2m-1$ and $m\leq j-k\leq 2m-1$.  The statement (\ref{eq : int}) follows from (\ref{lg : locint}). To prove (\ref{eq : locglob}) note that by (\ref{lg : locqg}) $\partial{Y}_{k+m},...,\partial{Y}_{j}$ are pairwise disjoint and have non-empty projections to $Y_k$.  Consequently, the distance in $Y_k$ between any two of these boundaries is at most $2$, and so $\diam_{Y_{k}}(\{\partial{Y}_{l}\}_{l=k+m}^{j})\leq 2$.  Similarly, $\diam_{Y_{k}}(\{\partial{Y}_{l}\}_{l=j'}^{k-m})\leq 2$.
By the triangle inequality we have
\begin{eqnarray*}
d_{Y_{k}}(\partial{Y}_{j},\partial{Y_{j'}})&\geq& d_{Y_{k}}(\partial{Y}_{k-m},\partial{Y_{k+m}})-d_{Y_{k}}(\partial{Y}_{j},\partial{Y_{k+m}})-d_{Y_{i}}(\partial{Y}_{k-m},\partial{Y_{j'}})\\
&\geq&d_{Y_{k}}(\partial{Y}_{k-m},\partial{Y_{k+m}})-4 \geq d_{Y_{k}}(\partial{Y}_{k-m},\partial{Y_{k+m}})-C
\end{eqnarray*}
which is the bound (\ref{eq : locglob}).

 Suppose that  (\ref{eq : int}) and (\ref{eq : locglob}) hold for all $m\leq k-j'\leq 2m-1$ and $m\leq j-k\leq N$, for some $N \geq 2m-1$. We suppose $j-k=N+1$ and we must prove both \eqref{eq : int} and \eqref{eq : locglob} for $(j-k,k-j')$.  
 
 From the base of induction we already have $Y_k \pitchfork Y_{j'}$.   To complete the proof of \eqref{eq : int}, we prove $Y_k \pitchfork Y_j$.  Since $m = (k+m)-k \leq 2m-1$ and $m \leq j-(k+m) = N+1-m \leq N$, from the inductive hypothesis we have
 \[ Y_k \pitchfork Y_{k+m} \mbox{ and } Y_j \pitchfork Y_{k+m} \]
 and
 \[ d_{Y_{k+m}}(\partial Y_k, \partial Y_j) \geq d_{Y_{k+m}}(\partial Y_k, \partial Y_{k+2m}) - C \geq E-C \geq 4.\]
Consequently, $i(\partial Y_k,\partial Y_j) \neq 0$ and $Y_k \pitchfork Y_j$ as required.

We now turn to the proof of \eqref{eq : locglob}.   Since $Y_k \pitchfork Y_j$ and $Y_k \pitchfork Y_{j'}$, by \eqref{lg : locint} we may write the following triangle inequality
\begin{eqnarray} \label{eq : dYiYjYj'}
d_{Y_{k}}(\partial{Y}_{j'},\partial{Y}_{j}) & \geq & d_{Y_{k}}(\partial{Y}_{k-m},\partial{Y}_{k+m})\\ \notag
& & \, -d_{Y_{k}}(\partial{Y}_{k-m},\partial{Y}_{j'})-d_{Y_{k}}(\partial{Y}_{j},\partial{Y}_{k+m}).
\end{eqnarray}
Since $m \leq j-(k+m) = N+1-m \leq N$, from the inductive hypothesis we have
\[ d_{Y_{k+m}}(\partial Y_k,\partial Y_j) \geq d_{Y_{k+m}}(\partial Y_k,\partial Y_{k+2m}) - C \geq E - C \geq B_0.\]
By Theorem~\ref{thm : behineq}, $d_{Y_k}(\partial Y_{k+m},Y_j) \leq B_0$. On the other hand, as in the proof of the base case of induction, since $m \leq k-j' \leq 2m-1$  we have
\[ d_{Y_k}(\partial Y_{k-m},\partial Y_{j'}) \leq 2.\]
Combining these two inequalities with \eqref{eq : dYiYjYj'}, we obtain
\begin{eqnarray*}
d_{Y_k}(\partial Y_{j'},\partial Y_j) & \geq & d_{Y_k}(\partial Y_{k-m},\partial Y_{k+m}) - B_0 -2 \\
& \geq & d_{Y_k}(\partial Y_{k-m},\partial Y_{k+m}) - C.
\end{eqnarray*}
This completes the first half of the double induction.

We now know that \eqref{eq : int} and \eqref{eq : locglob} hold for all $j,j',k$ with $m \leq k-j' \leq 2m-1$ and all $j-k \geq m$.  We assume that they hold for $m \leq k-j' \leq N$ and $j-k \geq m$ for some $N \geq 2m-1$, and prove that they hold for $k-j' = N+1$.  The proof of \eqref{eq : int} is completely analogous to the proof in the first part of the induction, and we omit it.  The proof of \eqref{eq : locglob} is also similar, but requires one additional step so we give the proof.

We may again write the triangle inequality \eqref{eq : dYiYjYj'}.  Since $m \leq (k-m)-j' = N+1-m \leq N$ by the inductive hypothesis we have
\[
d_{Y_{k-m}}(\partial{Y}_{k},\partial{Y}_{j'}) \geq E-C\geq B_{0},
\]
and so Theorem~\ref{thm : behineq} again implies $d_{Y_k}(\partial Y_{k-m},\partial Y_{j'}) \leq B_0$.  If $j-k \leq 2m-1$, then as above $d_{Y_k}(\partial Y_{k+m},\partial Y_j) \leq 2$.  Otherwise, by induction we have
\[ d_{Y_{k+m}}(\partial Y_k,\partial Y_j) \geq E-C \geq B_0 \]
and Theorem~\ref{thm : behineq} once again implies $d_{Y_k}(\partial Y_{k+m},\partial Y_j) \leq B_0$.  Combining these inequalities with \eqref{eq : dYiYjYj'} we have
\begin{eqnarray*}
d_{Y_k}(\partial Y_{j'},\partial Y_j) & \geq & d_{Y_k}(\partial Y_{k-m},\partial Y_{k+m}) - B_0 - \max \{2,B_0\} \\
& \geq & d_{Y_k}(\partial Y_{k-m},\partial Y_{k+m}) - C.
\end{eqnarray*}
This completes the proof of (\ref{eq : locglob}), and hence the double induction is finished.

\bigskip
 
Now further assuming \eqref{lg : locfill} we prove \eqref{eq : lbdS}.  Note that we must have $n \geq m$.  Without loss of generality we assume $j < k$, so that $k - j \geq 2n \geq 2m$.  For the rest of the proof, for any $s,r \in \mathbb Z$, $s \leq r$, we write $[s,r] = \{ t \in \mathbb Z \mid s \leq t \leq r \}$.

Suppose $\delta$ is any multi-curve.  Let $\mathcal I(\delta) = \{ s \in [j,k] \mid i(\delta,\partial Y_s) \neq 0 \}$. 

\begin{claim} \label{claim : new ends of I} Suppose $s',r' \in [j,k] \setminus \mathcal I(\delta)$.  Then $|r'-s'| \leq 4n-2$.
\end{claim}
Observe that by the claim, $[j,k] \setminus \mathcal I(\delta)$ contains fewer than $4n$ integers.
\begin{proof}
Without loss of generality, we assume $s' < r'$, and suppose for a contradiction $r' - s' \geq 4n-1$.   By \eqref{lg : locfill} $\partial Y_{s'+n},\ldots,\partial Y_{s'+3n-1}$ fills $S$, and so there exists $s'+n \leq t \leq s'+3n-1$ with $t \in \mathcal I(\delta)$.

Now observe that $s' + m \leq s' + n \leq t$ and $t \leq s'+3n-1 \leq r' - n \leq r' - m$, by the first part of the theorem we know that
\[ d_{Y_t}(\partial Y_{s'},\partial Y_{r'}) \geq E - C > 4. \]
On the other hand, since $i(\delta,\partial Y_{s'}) = 0 = i(\delta,\partial Y_{r'})$, and since $t \in \mathcal I(\delta)$ implies $\pi_{Y_t}(\delta) \neq \emptyset$, the triangle inequality implies
\[ d_{Y_t}(\partial Y_{s'},\partial Y_{r'}) \leq d_{Y_t}(\partial Y_{s'},\delta) + d_{Y_t}(\delta,\partial Y_{r'}) \leq 2+ 2 = 4\]
a contradiction.
\end{proof}

Let $\eta$ be a geodesic in $\mathcal{C}(S)$ connecting $\partial{Y}_{j}$ to $\partial{Y}_{k}$. For any $l\in\{j+m,..., k-m\}$, by (\ref{eq : locglob}) we have that
$$d_{Y_{l}}(\partial{Y}_{j},\partial{Y}_{k})\geq E-C> G.$$
Thus Theorem \ref{thm : bddgeod} guarantees that there is a curve $\delta_{l}\in\eta$ disjoint from $Y_{l}$.  Choose one such $\delta_l \in \eta$ for each $l \in [j+m,k-m]$.  By the previous claim there are at most $4n$ integers $l' \in [j+m,k-m]$ such that $i(\delta_l,\partial Y_{l'}) = 0$, and hence $l \mapsto \delta_l$ is at most $4n$-to-$1$.

Therefore, $\eta$ contains at least $\tfrac{k-j-2m + 1}{4n} > \tfrac{k-j}{4n} - \tfrac{m}{2n}$ curves.  It follows that
\[ d_S(\partial Y_j,\partial Y_k) \geq \frac{k-j}{4n} - \left( \frac{m}{2n} + 1 \right) \]
proving \eqref{eq : lbdS}.  This completes the proof of the theorem.
\end{proof}

\begin{thm}\label{thm : minfill}
Let $\{Y_{k}\}_{k=0}^{\infty}$ be an infinite sequence of subsurfaces satisfying conditions (1)-(4) in Theorem \ref{thm : locglob}. Then there exists a unique $\nu \in \EL(S)$ so that any accumulation point of $\{\partial Y_{k}\}_{k=0}^{\infty}$ in $\PML(S)$ is supported on $\nu$.
\end{thm}
\begin{proof}
By Theorem \ref{thm : locglob}, Inequality (\ref{eq : lbdS}), the sequence $\{\partial{Y}_{k}\}_{k=0}^{\infty}$ is (multi-curve) quasi-geodesic in $\mathcal{C}(S)$. Furthermore $\mathcal{C}(S)$ is $\delta-$hyperbolic. Thus the sequence converges to a point in the Gromov boundary of $\mathcal{C}(S)$. Theorem~\ref{thm : bdrycc} completes the proof.
\end{proof}

We complete this section by showing that $\mathcal P$ is sufficient to imply the hypotheses of Theorem~\ref{thm : locglob}.  Given a curve $\alpha$ and an annular subsurface $Y_\beta$ with core curve $\beta$, we note that $\alpha \pitchfork Y_\beta$ if and only if $i(\alpha,\beta) \neq 0$.  Consequently, to remind the reader of the relation to Theorem~\ref{thm : locglob}, we write $\alpha \pitchfork \beta$ to mean $i(\alpha,\beta) \neq 0$.

\begin{prop} \label{prop : P implies loc-to-global}
Any sequence $\{\gamma_k\}_{k=0}^\infty$ satisfying $\mathcal P(a)$ with $a > 2$ and $e_0 \geq E$ are the core curves of annuli satisfying conditions (1)--(4) of Theorem~\ref{thm : locglob} with $n =m$. Consequently, $\{\gamma_k\}_{k=0}^\infty$ is a $1$--Lipschitz, $(4m,\tfrac32)$--quasi-geodesic in $\mathcal C(S)$ and there exists $\nu \in \EL(S)$ so that any accumulation point of $\{\gamma_k\}_{k=0}^\infty$ in $\PML(S)$ is supported on $\nu$.
\end{prop}
\begin{proof}
Condition (i) of $\mathcal P$ is the same as condition (1) of Theorem~\ref{thm : locglob}, while (ii) is just condition (4) with $n = m$.  Condition (2) follows from Lemma~\ref{lem : length 2m intervals}.  Finally, to see that condition (3) is satisfied, we note that $d_{\gamma_k}(\gamma_{k-m},\gamma_{k+m}) \geq e_k > a^k E > 2E > E$, for all $k \geq m$.  Furthermore, for $k-2m+1 \leq j \leq k-m$, $\gamma_j \pitchfork \gamma_k$ by Lemma~\ref{lem : length 2m intervals}, and similarly $\gamma_{j'} \pitchfork \gamma_k$, for $k+m \leq j' \leq k+2m-1$.  For $j$ and $j'$ in these intervals, $i(\gamma_j,\gamma_{k-m}) = 0$ and $i(\gamma_{j'},\gamma_{k+m}) = 0$.  Therefore, by the triangle inequality, $d_{\gamma_k}(\gamma_j,\gamma_{j'}) \geq a^k E - 2  > E$, as required by (3).
\end{proof}

\subsection{Subsurface coefficient bounds}\label{subsec : subsurfbd}

We will need estimates on all subsurface coefficients for a sequence satisfying $\mathcal P$.  This follows from what we have done so far, together with similar arguments.
\begin{prop}  \label{prop : anncoeff + coeffbd}
Given a sequence $\{\gamma_k\}_{k=0}^\infty$ satisfying $\mathcal P(a)$ with $a > 2$ and $e_0 \geq E$, then there exists $R > 0$ with the following property.
\begin{enumerate}
\item  If $i,j,k$ satisfy $j \leq i-m$ and $i+m \leq k$ then $\gamma_i \pitchfork \gamma_k$, $\gamma_i \pitchfork \gamma_j$, and
\begin{equation} \label{eqn : gamma_i, big projection} d_{\gamma_i}(\gamma_j,\gamma_k) \stackrel{+}{\asymp}_R e_i \quad \mbox{ and } \quad d_{\gamma_i}(\gamma_j,\nu) \stackrel{+}{\asymp}_R e_i
\end{equation}
\item  If $W \subsetneq S$ is a proper subsurface, $W \neq \gamma_i$ for any $i$, then for any $j,k$ with $\gamma_j \pitchfork W$ and $\gamma_k \pitchfork W$
\begin{equation} \label{eqn : not gamma_i, small projection} d_W(\gamma_j,\gamma_k) < R \quad \mbox{ and } \quad d_W(\gamma_j,\nu) < R.
\end{equation}
\end{enumerate}
Let $\mu$ be a marking on $S$, then there is a constant $R(\mu)$ so that

For any $k$ sufficiently large and $i\leq k-m$
\begin{equation}\label{eq : mu gi big proj}
d_{\gamma_i}(\mu,\gamma_{k})\stackrel{+}{\asymp}_{R(\mu)} e_{i} \;\;\mbox{ and }\;\; d_{\gamma_i}(\mu,\nu)\stackrel{+}{\asymp}_{R(\mu)} e_{i}
\end{equation}
For any proper subsurface $W\neq \gamma_{i}$ for any $i$ we have
\begin{equation}\label{eq : not gi small proj}
d_{W}(\mu,\gamma_{k})<R(\mu)  \;\;\mbox{ and }\;\; d_{W}(\mu,\nu)<R(\mu) 
\end{equation}
\end{prop}
\begin{proof}
We begin with the proofs of (\ref{eqn : gamma_i, big projection}) and (\ref{eqn : not gamma_i, small projection}).
First note that since any accumulation point of $\{\gamma_k\}$ in $\PML(S)$ is supported on $\nu$, any Hausdorff accumulation point of $\{\gamma_k\}$ contains $\nu$.  Thus, for any fixed, proper subsurface $W \subsetneq S$ and all sufficiently large $k$ we have $\pi_W(\nu) \subseteq \pi_W(\gamma_k)$.  Furthermore, since $\nu$ is an ending lamination, $\pi_W(\nu) \neq \emptyset$, and hence $d_W(\gamma_k,\nu) \leq 1$, for $k$ sufficiently large.  Therefore, for each of (\ref{eqn : gamma_i, big projection}) and (\ref{eqn : not gamma_i, small projection}), the statement on the left implies the one on the right after increasing the constant by at most $1$.  Thus it suffices to prove the two statements on the left.

We begin with (\ref{eqn : gamma_i, big projection}).  From the conditions in $\mathcal P$, we have $d_{\gamma_i}(\gamma_{i-m},\gamma_{i+m}) = e_i$.   By Theorem~\ref{thm : locglob} (which is applicable according to Proposition~\ref{prop : P implies loc-to-global}), $\{\gamma_l\}_{l = i+m}^k$ is a $1$--Lipschitz $(4m,3/2)$--quasi-geodesic such that every curve has nonempty projection to $\gamma_i$.  Therefore, by Theorem~\ref{thm : bddgeod} and the triangle inequality we have
\[ |d_{\gamma_i}(\gamma_{i-m},\gamma_k) - d_{\gamma_i}(\gamma_{i-m},\gamma_{i+m})| \leq d_{\gamma_i}(\gamma_{i+m},\gamma_k) \leq G.\]
Note that $G$ depends only on $m$.
Similar reasoning implies
\[ |d_{\gamma_i}(\gamma_j,\gamma_k) - d_{\gamma_i}(\gamma_{i-m},\gamma_k)| \leq d_{\gamma_i}(\gamma_j,\gamma_{i-m}) \leq G. \]
Combining these we have
\begin{eqnarray*} 
|d_{\gamma_i}(\gamma_j,\gamma_k)  - d_{\gamma_i}(\gamma_{i-m},\gamma_{i+m})| & = & |d_{\gamma_i}(\gamma_j,\gamma_k) - d_{\gamma_i}(\gamma_{i-m},\gamma_k)\\
& &  + d_{\gamma_i}(\gamma_{i-m},\gamma_k) - d_{\gamma_i}(\gamma_{i-m},\gamma_{i+m})|\\
& \leq & 2G
\end{eqnarray*}
It follows that $d_{\gamma_i}(\gamma_j,\gamma_k) \stackrel + \asymp_{2G} e_i$.  For $R \geq 2G$, (\ref{eqn : gamma_i, big projection}) holds.

We now move on to (\ref{eqn : not gamma_i, small projection}), and without loss of generality assume $j \leq k$.  If $k \leq j+2m-1$, then the conditions in $\mathcal P$ together with Lemma~\ref{lem : length 2m intervals} imply $i(\gamma_j,\gamma_k) \leq b_2$, so by Lemma~\ref{lem : id}, $d_W(\gamma_j,\gamma_k) \leq 2 b_2 + 1$.  

Next, suppose that $k = j + 2m$.  Let $\gamma_k'$ be the element guaranteed by $\mathcal P$, so that $\gamma_k = \mathcal D_{\gamma_{k-m}}^{e_{k-m}}(\gamma_k')$.  There are two cases to consider depending on whether $\gamma_k' \not\pitchfork W$ or $\gamma_k' \pitchfork W$.  If $\gamma_k' \not\pitchfork W$, then since $\gamma_k  = \mathcal D_{\gamma_{k-m}}^{e_{k-m}}(\gamma_k') \pitchfork W$, we must have $\gamma_{k-m} \pitchfork W$.  Now observe that $j \leq k-m = j + m \leq j + 2m-1$ and $k-m \leq k \leq k-m + 2m - 1$.   It follows from the previous paragraph that $d_W(\gamma_j,\gamma_{k-m}) \leq 2b_2+1$ and $d_W(\gamma_{k-m},\gamma_k) \leq 2b_2 + 1$, hence
\[ d_W(\gamma_j,\gamma_k) \leq 4b_2 + 2. \]

Now suppose $\gamma_k' \pitchfork W$.  If $\gamma_{k-m} \pitchfork W$ then just as in the first case we have $d_W(\gamma_j,\gamma_k) \leq 4b_2 + 2$.   Suppose then that $\gamma_{k-m} \not\pitchfork W$.  If $W$ is not an annulus, then $\pi_W(\gamma_k) = \pi_W(\gamma_k')$ since $\mathcal D_{\gamma_{k-m}}$ is supported outside $W$.  Therefore
\[ d_W(\gamma_j,\gamma_k) = d_W(\gamma_j,\gamma_k') \leq 2b_2 + 1 \]
since $i(\gamma_j,\gamma_k') \leq b_2$.   If $W$ is an annulus, because $W \neq \gamma_{k-m}$ and $\gamma_{k-m} \not \pitchfork W$, it easily follows that 
\[ d_W(\gamma_j,\gamma_k) \leq d_W(\gamma_j,\gamma_k') + d_W(\gamma_k',\gamma_k) \leq (2b_2 +1) + 1\]
(see e.g.~\cite{rk1mcg}).
Therefore, we have shown that if $k \leq j + 2m$, we have
\begin{equation} \label{eqn : bounding projection, small} d_W(\gamma_j,\gamma_k) \leq 4b_2 + 2.
\end{equation}

Now we suppose $k > j + 2m$.  Setting $\delta=\partial W$, as in the proof of Theorem~\ref{thm : locglob} we let $\mathcal I(\delta) = \{ s \in [j,k] \mid i(\delta,\gamma_s) \neq 0 \}$.   Similarly, we let $\mathcal I(W) = \{ s \in [j,k] \mid \gamma_s \pitchfork W\}$, and observe that $\mathcal I(\delta) \subseteq \mathcal I(W)$.

Note that $j,k \in \mathcal I(W)$, and we let $s \leq r$ be such that $[j,s],[r,k] \subseteq \mathcal I(W)$ are maximal subintervals of $\mathcal I(W)$ containing $j$ and $k$, respectively (if $\mathcal  I(W) = [j,k]$, we can arbitrarily choose $j \leq s < k$ and $r = s+1$ for the argument below).  By our choice of $r$ and $s$, it follows that $s+1,r-1 \not \in \mathcal I(W)$, and so Claim \ref{claim : new ends of I} implies $r-1 - (s+1) \leq 4m-2$ and hence $r-s \leq 4m$.

Note that since any $2m$ consecutive curves fill $S$, either $r-s \leq 2m$, or else there exists $s',r' \in \mathcal I(W)$ such that $s < s' \leq r' < r$ and $r- r', r'-s',s'-s \leq 2m$.  For example, consider the extremal case that $r-s = 4m$.  Then
\[ s' = \max \mathcal I(W) \cap [s,s+2m] \quad \mbox{ and } \quad r' = \min \mathcal I(W)\cap [s+2m,r]\]
have the desired properties.  Indeed, $s'-s,r-r'$ are clearly less than $2m$.  If $r'-s' > 2m$, then since any $2m$ consecutive curves fill $S$, there must be some $s' < u < r'$ in $\mathcal I(W)$, contradicting the choice of either $s'$ or $r'$. The general case is similar.  

By the triangle inequality and (\ref{eqn : bounding projection, small}) we have
\begin{equation} \label{eqn : bounding projection, medium} d_W(\gamma_s,\gamma_r) \leq d_W(\gamma_s,\gamma_{s'}) + d_W(\gamma_{s'},\gamma_{r'}) + d_W(\gamma_{r'},\gamma_r) \leq 12b_2 + 6.
\end{equation}
Since $\{\gamma_l\}_{l = j}^s$ and $\{\gamma_l\}_{l =r}^k$ are $1$--Lipschitz $(4m,3/2)$--quasi-geodesics with $\gamma_l \pitchfork W$ for all $l \in [j,s] \cup [r,k]$, we can apply Theorem~\ref{thm : bddgeod}, and so the triangle inequality and (\ref{eqn : bounding projection, medium}) give us
\[ d_W(\gamma_j,\gamma_k) \leq d_W(\gamma_j,\gamma_s) + d_W(\gamma_s,\gamma_r) + d_W(\gamma_r,\gamma_k) \leq 2 G + 12b_2 + 6.\]
So the inequality on the left of (\ref{eqn : not gamma_i, small projection}) holds for any $R \geq 2G + 12b_2+6$. This completes the proof of the first four estimates.

Given a marking $\mu$, note that the intersection number of any curve in $\mu$ and any of the curves in the set of filling curves $\gamma_{0},...,\gamma_{2m-1}$ is bounded. Then the estimates in (\ref{eq : mu gi big proj}) follow from the ones in (\ref{eqn : gamma_i, big projection}) and Lemma \ref{lem : id} respectively. Similarly the estimates in (\ref{eqn : not gamma_i, small projection}) follow from the ones in (\ref{eq : not gi small proj}).
\end{proof}

\section{Measures supported on laminations}\label{sec : measnue}

In this section we begin by proving intersection number estimates for a sequence of curves satisfying $\mathcal P$.  Using these estimates, we decompose the sequence into $m$ subsequences and prove that these converge in $\mathcal{PML}(S)$.  In the next section, we will show that these $m$ limits are precisely the vertices of the simplex of measures on the single topological lamination $\nu$ from Proposition~\ref{prop : P implies loc-to-global}.



\subsection{Intersection number estimates}  \label{S : intersection number estimates} Here we estimate the intersection numbers of curves in the sequence of curves $\{\gamma_{k}\}_{k=0}^{\infty}$ satisfying $\mathcal P$.  The estimates will be in terms of the constant $b$ and sequence $\{e_k\}$ fixed above.
Specifically, given $i,k\in\mathbb{N}$ with $k\geq i$, define
\begin{equation}\label{eq : Aik} A(i,k):= \prod_{\substack{i+m \leq j<k \; and\\ j\equiv k \mod m}}be_{j}.\end{equation}
When the set of indices of the product is the empty set we define the product to be $1$.
It is useful to observe that for $k \geq i+2m$, 
\[ A(i,k) = b e_{k-m} A(i,k-m).\]
It is also useful to arrange the indices as in \eqref{eq:how we view it} in the following form
\begin{equation} \label{eq:how we view indices} \begin{array}{ccccccc}
i & \quad \quad & i+1 & \quad \quad & \cdots & \quad \quad & i+m-1\\
i+m & & i+m+1 & & \cdots & & i+2m-1\\
i+2m & & i+2m+1 & & \cdots \end{array} \end{equation}
Then $A(i,k)$ is $1$ exactly when $k$ is in the first or second row.  If $k$ is below these rows, then the product defining $A(i,k)$ is over all indices $j$ directly above $k$, up to and including the entry in the second row. 

We now state the main estimate on intersection numbers.
\begin{thm}\label{thm : intgkgi}
Suppose $\{\gamma_k \}_{k=0}^\infty$ is a sequence on a surface $S$ satisfying $\mathcal P(a)$.
For $a$ is sufficiently large, there is a constant $\kappa = \kappa(a)>1$, so that for each $i,k$ with $k\geq i+m$ we have
\begin{equation}\label{eq : intgkgi}i(\gamma_i,\gamma_k)\stackrel{*}{\asymp}_{\kappa}A(i,k).\end{equation}
\end{thm}
Recall that for $i \leq k < i+m$, $i(\gamma_i,\gamma_k) = 0$.  Combining this with the theorem gives estimates on all intersection numbers $i(\gamma_i,\gamma_k)$, up to a uniform multiplicative error.

\bigskip

Throughout all that follows, we will assume that a sequence of curves $\{\gamma_k\}_{k=0}^\infty$ satisfies $\mathcal P = \mathcal P(a)$ for $a > 1$.

\bigskip

\noindent
{\bf Outline of the proof:}  The proof is rather complicated involving multiple induction arguments, so we sketch the approach before diving into the details.   The upper bound on $i(\gamma_i,\gamma_k)$ is proved first, and is valid for any $a > 1$.  We start by recursively defining a function $K(i,k)$ for all nonnegative integers $i \leq k$.  By induction, we will prove that
\[ i(\gamma_i,\gamma_k) \leq K(i,k) A(i,k). \]
By a second induction, we will bound $K(i,k) \leq K_1 = K_1(a)$, with the bound $K_1(a)$ a decreasing function of $a$.   Next, we will recursively define a function $K'(i,k) = K'(i,k,a)$.  By another induction, we prove that
\[ i(\gamma_i,\gamma_k) \geq K'(i,k) A(i,k). \]
For $a$ sufficiently large, we prove that $K'(i,k,a) \geq K_2 = K_2(a) > 0$.  Setting $\kappa = \max\{K_1,\frac1{K_2}\}$ will prove the theorem. 

\bigskip

\noindent
{\bf Upper bound.}  Recall from $\mathcal P$ (Definition \ref{def : P}) that for any $k \geq 2m$, the set of curves $\{\gamma_l\}_{l=k-2m}^{k-1}$ fill the surface, and the curve $\gamma_k'$ intersects each of these at most $b_2$ times.  Consequently, all complementary components of $S \setminus (\gamma_{k-2m} \cup \ldots \gamma_{k-1})$ are either disks or once-punctured disks containing at most $2mb_2$ pairwise disjoint arcs of $\gamma_k'$.  In examples we may have many fewer than $2mb_2$ such arcs, and it is useful to keep track of this constant on its own.  Consequently, we set
\begin{equation} \label{eq : definition of B}
B \leq 2mb_2
\end{equation}
to be the maximum number of arcs in any complementary component (over all configurations in minimal position).

We are now ready for a recursive definition which will be used in the bounds on intersection numbers (it is useful again to picture the indices as in \eqref{eq:how we view indices}):
\[ K(i,k) = \left\{ \begin{array}{ll} 0 & \mbox{for } i \leq k < i+m \\
b_2 & \mbox{for } i+m \leq k < i+2m\\
K(i,k-m) + 2B \displaystyle{\sum_{l = k-2m}^{k-1} \frac{A(i,l)}{A(i,k)}} K(i,l) & \mbox{for } i+2m \leq k \end{array} \right. \]

\begin{lem} \label{lem : K(i,k)A(i,k) bound}
For all $i\leq k$, we have $i(\gamma_i,\gamma_k) \leq K(i,k)A(i,k)$.
\end{lem}

The proof takes advantage of the following well-known estimate on the intersection of two curves after applying a power of a Dehn twist on one proved in the Appendix A of \cite[Expos\'{e} 4]{FLP}, see also Lemma 4.2 in \cite[\S 4]{ivanovsubgpmcg}.  
\begin{prop}{\sc (Intersection number after Dehn twist)}\label{prop : intdtflp}

Let $\delta$, $\delta'$, and $\beta$ be curves in $\mathcal{C}(S)$.  Then for any integer $e$
\begin{equation}\label{eq : intdtflp}\Big|i(\mathcal{D}_{\beta}^{e}\delta,\delta')-|e|i(\beta,\delta)i(\beta,\delta')\Big|\leq i(\delta,\delta').\end{equation}
\end{prop}

As above, $\mathcal D_\beta$ is a Dehn twist in $\beta$.  This proposition has the following general application to intersection numbers of curves with the curves in our sequence.
\begin{prop} \label{prop : application of flp}
For any curve $\delta$ and any $k \geq 2m$, we have
\[ |i(\delta,\gamma_k)-be_{k-m}i(\delta,\gamma_{k-m})|\leq 2B \sum_{l=k-2m}^{k-1} i(\delta,\gamma_l). \]
\end{prop}
\begin{proof}
Since $\gamma_k = \mathcal D_{\gamma_{k-m}}^{e_{k-m}}(\gamma_k')$, Proposition~\ref{prop : intdtflp} implies
\[  |i(\delta,\gamma_k)-be_{k-m}i(\delta,\gamma_{k-m})|\leq i(\delta,\gamma_k'). \]
Assume all curves intersect minimally transversely and that there are no triple points of intersection.  From the definition of $B$, all complementary components of $S \setminus (\gamma_{k-2m} \cup \ldots \gamma_{k-1})$ contain at most $B$ pairwise disjoint arcs of $\gamma_k'$.  Therefore, between any two consecutive intersection points of $\delta$ with $\gamma_{k-2m} \cup \ldots \cup \gamma_{k-1}$, there are at most $2B$ intersections points with $\gamma_k'$ (any two arcs in a disk component can intersect at most once, and in a once-punctured disk component can intersect in at most two points).  Therefore, 
\[ i(\delta,\gamma_k') \leq 2B \sum_{l = k-2m}^{k-1} i(\delta,\gamma_l).\]
Combining this with the above inequality proves the proposition.
\end{proof}
We are now ready for the
\begin{proof}[Proof of Lemma~\ref{lem : K(i,k)A(i,k) bound}.]
Fix $i$.  The proof is by induction on $k$.  For $i \leq k < i+m$, $i(\gamma_i,\gamma_k) = 0$, $K(i,k) = 0$ and $A(i,k) = 1$, so the lemma follows.  Similarly, for $i+m \leq k < i+ 2m$, $i(\gamma_i,\gamma_k) \leq b_2$, $K(i,k) = b_2$, and $A(i,k) = 1$, so again the lemma follows.  Now suppose that $k \geq i+2m$, and assuming that $i(\gamma_i,\gamma_l) \leq K(i,l)A(i,l)$ for all $i \leq l < k$, we must prove $i(\gamma_i,\gamma_k) \leq K(i,k) A(i,k)$.

Applying Proposition~\ref{prop : application of flp} to the case $\delta = \gamma_i$, we have

\[ |i(\gamma_i,\gamma_k)-be_{k-m}i(\gamma_i,\gamma_{k-m})|\leq 2B \sum_{l=k-2m}^{k-1} i(\gamma_i,\gamma_l). \]
Therefore, we have
\[ i(\gamma_i,\gamma_k) \leq be_{k-m}i(\gamma_i,\gamma_{k-m}) + 2B \sum_{l=k-2m}^{k-1} i(\gamma_i,\gamma_l). \]
Applying the inductive hypothesis and the definitions of $A$ and $K$ to this inequality we obtain
\begin{eqnarray*} i(\gamma_i,\gamma_k) & \leq & be_{k-m}i(\gamma_i,\gamma_{k-m}) + 2B \sum_{l=k-2m}^{k-1} i(\gamma_i,\gamma_l)\\
& \leq & be_{k-m}K(i,k-m)A(i,k-m) + 2B \sum_{l=k-2m}^{k-1} K(i,l)A(i,l)\\
& = & A(i,k) K(i,k-m) + A(i,k) 2B \sum_{l = k-2m}^{k-1} \frac{A(i,l)}{A(i,k)} K(i,l)\\
& = & A(i,k)\left(K(i,k-m) + 2B \sum_{l = k-2m}^{k-1} \frac{A(i,l)}{A(i,k)} K(i,l) \right)\\
& = & A(i,k)K(i,k),\\
\end{eqnarray*}
as required.
\end{proof}

Next we prove that $K(i,k)$ is uniformly bounded, and in particular:
\begin{prop} \label{prop : uniform upper bound int}
There exists $K_1 = K_1(a)> 0$ so that for all $i \leq k$, $K(i,k) \leq K_1$ and in particular, $i(\gamma_i,\gamma_k) \leq K_1A(i,k)$.  As a function of $a$, $K_1(a)$ is decreasing.
\end{prop}

For the proof of this proposition, we will need the following bound.  

\begin{lem} \label{lem : A ratio bound}
For all $i \leq l < k$, we have
\[ \frac{A(i,l)}{A(i,k)} \leq a^{1-\lfloor \tfrac{k-i}m \rfloor}.\]
\end{lem}
\begin{proof}
If $k < i+2m$, then $A(i,l),A(i,k) =1$ and $a^{1-\lfloor \tfrac{k-i}m \rfloor } \geq 1$, so the inequality follows.

Now assume $k \geq i+2m$.
By definition, we have
\[ \frac{A(i,l)}{A(i,k)} = \frac{\displaystyle{\prod_{\tiny\substack{i+m \leq j'< l \; and\\ j' \equiv l \mod m}}be_{j'}}}{\displaystyle{\prod_{\tiny \substack{i+m \leq j<k \; and\\ j\equiv k \mod m}}be_{j}}} \]
(where $A(i,l)$ is $1$ if $l < i + 2m$).
Observe that the denominator has $r = \lfloor \frac{k-(i+m)}m \rfloor = \lfloor \frac{k-i}m \rfloor -1 > 0$ terms in the product, indexed by $j \in \{k-m,k-2m,\ldots,k-rm \}$, while the numerator has $s = \max\{0,\lfloor \frac{l - i}m \rfloor - 1\} \geq 0$ terms, indexed by $j' \in \{ l-m,l-2m,\ldots,l-sm\}$ (possibly the empty set).  Since $l < k$, $s\leq r$. Moreover, we have $k-pm > l-pm$, and thus $e_{k-pm} > a e_{l-pm}$ by \eqref{eq : ek}, for all $p = 1,\ldots,s$.  Since \eqref{eq : ek} also implies $e_j > a$ for all $j \geq 1$, combining these bounds with the equation above gives
\[ \frac{A(i,l)}{A(i,k)} = \prod_{p=1}^s \frac{e_{l-pm}}{e_{k-pm}} \prod_{p=s+1}^r \frac1{e_{k-pm}} < \prod_{p=1}^s a^{-1} \prod_{p=s+1}^r a^{-1} = a^{-r} = a^{1-\lfloor \tfrac{k-i}m \rfloor },\]
as required.
\end{proof}

As an application, of Lemma~\ref{lem : A ratio bound}, we prove
\begin{lem} \label{lem : final K(i,k) bound}
For all $i \leq k$ we have
\[ K(i,k) \leq b_2 \prod_{i+m \leq j < k} (1 + 4mB a^{1-\lfloor \tfrac{j-i+1}m \rfloor}). \]
\end{lem}
As above, the empty product is declared to be $1$.
\begin{proof}
The proof is by induction on $k$.  Since $K(i,k) \leq b_2$ for $i \leq k < i+2m$, the lemma clearly holds for all such $k$.  Now assume that $k \geq i+2m$, and assume that the lemma holds for all integers less than $k$ and at least $i$.  Let $l_0$ be such that $k-2m \leq l_0 \leq k-1$ and
\[ K(i,l_0) = \max \{ K(i,l) \mid k-2m \leq l \leq k-1 \}.\]
From this, the definition of $K(i,k)$, and from Lemma~\ref{lem : A ratio bound} we have
\begin{eqnarray*} K(i,k) & = & K(i,k-m) + 2B \sum_{l = k-2m}^{k-1} \frac{A(i,l)}{A(i,k)} K(i,l) \\
& \leq & K(i,l_0)(1 + 2B \sum_{l = k-2m}^{k-1} a^{1-\lfloor \frac{k-i}m \rfloor} )\\
& = & K(i,l_0)(1+4mB a^{1-\lfloor \frac{k-i}m \rfloor})
\end{eqnarray*}
Since $l_0 < k$, the proposed bound on $K(i,l_0)$ holds by the inductive assumption.  Next, observe that the proposed upper bound is an increasing function of $k$.  Indeed, the required bound for $K(i,k)$ is obtained from the one for $K(i,k-1)$ by multiplying by a number greater than or equal to $1$.  By this monotonicity, the above bound implies
\begin{eqnarray*}
 K(i,k) & \leq & K(i,l_0)(1+4mBa^{1-\lfloor \frac{k-i}m \rfloor})\\
 & \leq & \left( b_2 \prod_{i+m \leq j < k-1} (1 + 4mB a^{1-\lfloor \tfrac{j - i + 1}m \rfloor}) \right)(1+4mBa^{1-\lfloor \frac{k-i}m \rfloor})\\
 & = & b_2 \prod_{i+m \leq j < k} (1 + 4mB a^{1-\lfloor \tfrac{j-i + 1}m \rfloor}).
\end{eqnarray*}
This completes the proof.
\end{proof}

We are now ready for the
\begin{proof}[Proof of Proposition~\ref{prop : uniform upper bound int}.]  The upper bound on $K(i,k)$ in Lemma~\ref{lem : final K(i,k) bound} is itself bounded above by the infinite product
\[ K_1(a) = b_2 \prod_{j = i+m}^\infty (1 + 4mB a^{1-\lfloor \tfrac{j-i+1}m \rfloor}) = b_2 \prod_{l = 0}^\infty (1 + 4mB a^{-\lfloor \tfrac{l + 1}m \rfloor}), \]
where we have substituted $l = j-i-m$.
We will be done if we prove that this product is convergent, for all $a > 1$, since the product then clearly defines a decreasing function of $a$.

The infinite product converges if and only if the infinite series obtained by taking logarithms does.  Since $\log(1+x) \leq x$
we have
\begin{eqnarray*} \log\left(b_2 \prod_{l = 0}^\infty (1 + 4mB a^{-\lfloor \tfrac{l + 1}m \rfloor})\right) & = & \log(b_2) + \sum_{l = 0}^{\infty}\log(1+4mB a^{-\lfloor \tfrac{l + 1}m \rfloor})\\
& \leq & \log(b_2) + 4mB \sum_{l = 0}^\infty a^{-\lfloor \tfrac{l + 1}m \rfloor} \\
\end{eqnarray*}
The last expression is essentially a geometric series, and hence converges for all $a > 1$, completing the proof.
\end{proof}

\noindent {\bf Lower bound.}  Let $b_{1}$ be the constant in $\mathcal P$ (Definition \ref{def : P}). We assume $a > 1$ is sufficiently large so that 
\begin{equation} \label{eq : restriction on a} C = 8mBK_1 \sum_{j=1}^\infty a^{-j}  < b_1.
\end{equation}
(which is possible since $K_1 = K_1(a)$ is decreasing by Proposition~\ref{prop : uniform upper bound int}).
For all $k \geq i+m$, define the function $K'(i,k)$ by the recursive formula for all $k \geq i + m$
\[ K'(i,k) = \left\{ \begin{array}{ll}
C & \mbox{for } i+m \leq k < i+2m\\
\displaystyle{K'(i,k-m)- 2B \!\!\!\! \sum_{l=k-2m}^{k-1} \!\!\! \tfrac{A(i,l)}{A(i,k)}K(i,l)} & \mbox{for } i+2m \leq k \end{array} \right. \]

\begin{lem} \label{lem : K'(i,k)A(i,k) bound}
For all $k \geq i + m$, we have $i(\gamma_i,\gamma_k) \geq K'(i,k)A(i,k)$.
\end{lem}
\begin{proof}
Fix integer $i\geq 0$.  The proof is by induction on $k$.  For the base case, we let $i +m \leq k < i+2m$. Then $A(i,k) = 1$ and $K'(i,k) = C < b_1$, while $i(\gamma_i,\gamma_k) \geq b_1$, and hence $i(\gamma_i,\gamma_k) \geq K'(i,k)A(i,k)$.  We assume therefore that $k \geq i+2m$ and that the lemma is true for all $i+m \leq l < k$.  

Applying Proposition~\ref{prop : application of flp} to the curve $\delta = \gamma_i$, together with Lemma~\ref{lem : K(i,k)A(i,k) bound} and the inductive hypothesis we have
\begin{eqnarray*}
i(\gamma_i,\gamma_k)& \geq & e_{k-m}bi(\gamma_i,\gamma_{k-m})- 2B \sum_{l=k-2m}^{k-1} i(\gamma_{i},\gamma_{l})\\
&\geq&e_{k-m}bK'(i,k-m)A(i,k-m)-2B \sum_{l=k-2m}^{k-1}K(i,l)A(i,l)\\
&=&A(i,k)\left( K'(i,k-m)-2B \sum_{l=k-2m}^{k-1}\tfrac{A(i,l)}{A(i,k)}K(i,l)\right)\\
&=&A(i,k)K'(i,k)\\
\end{eqnarray*}
as required.
\end{proof}

\begin{lem}
Setting $K_2 = C/2 >0$, then whenever $k \geq i+m$, $K'(i,k)\geq K_2$.
\end{lem} 

\begin{proof}
If $i+m \leq k < i+2m$, then $K'(i,k) = C > C/2 = K_2 > 0$.  Suppose now that $k \geq i+2m$, and let $k = p + sm$, where $s$ and $p$ are positive integers with $i+m \leq p < i+2m$ and $p \equiv k$ mod $m$.  Note that
\[ \lfloor \frac{k-i}m \rfloor = \lfloor \frac{p+sm-i}m \rfloor = s + \lfloor \frac{p-i}m \rfloor = s + 1.\]
By Lemma~\ref{lem : A ratio bound}, it follows that for all $l < k$, we have $\frac{A(i,l)}{A(i,k)} \leq a^{-s}$.
Then from the definition of $K'$ and Proposition~\ref{prop : uniform upper bound int} we have
\begin{eqnarray*} K'(i,k) & = & K'(i,k-m) - 2B \sum_{l=k-2m}^{k-1} \tfrac{A(i,l)}{A(i,k)} K(i,l) \\
 & \geq & K'(i,k-m) - 2B \sum_{l=k-2m}^{k-1} a^{-s} K_1 \\
 & \geq & K'(i,k-m) - 2B (2m)a^{-s}K_1 = K'(i,k-m) - 4mB K_1 a^{-s}\\
\end{eqnarray*}
Iterating this inequality $s$ times implies
\[ K'(i,k) \geq K'(i,p) - 4mB K_1 \sum_{q=1}^s a^{-q}.\]
Since $i+m \leq p < i+2m$, $K'(i,p) = C = 8mBK_1 \sum_{j=1}^\infty a^{-j}$ and hence
\[ K'(i,k) \geq 4mBK_1 (2\sum_{j=1}^\infty a^{-j} - \sum_{q=1}^s a^{-q}) \geq 4mBK_1 \sum_{j=1}^\infty a^{-j} = C/2 = K_2.\]
This completes the proof.
\end{proof}

\begin{proof}[Proof of Theorem~\ref{thm : intgkgi}.] For $a > 1$ satisfying \eqref{eq : restriction on a}, we have proved that for all $k \geq i+m$,
\[ K_2 A(i,k) \leq i(\gamma_i,\gamma_k) \leq K_1 A(i,k).\]
Since $K_1,K_2 > 0$, setting $\kappa = \max\{K_1,\frac1{K_2} \}$ finishes the proof.
\end{proof}

\noindent {\bf Convention.}  From this point forward, we will assume that $\mathcal P = \mathcal P(a)$ always has $a > 1$ sufficiently large so that (\ref{eq : restriction on a}) is satisfied, and consequently the intersection numbers of curves in any sequence $\{\gamma_k\}_{k=0}^\infty$ satisfies (\ref{eq : intgkgi}) in Theorem~\ref{thm : intgkgi}.  For concreteness, we note that from \eqref{eq : restriction on a}, $a \geq 16  > 2$ (though in fact, it is much larger).

\subsection{Convergence in $\mathcal{ML}(S)$} 

Consider again a sequence of curves $\{\gamma_{k}\}_{k=0}^{\infty}$ which satisfies the conditions of Theorem~\ref{thm : intgkgi}. Let $\nu\in \mathcal{EL}(S)$ be the lamination from Proposition~\ref{prop : P implies loc-to-global}.  In this section we will prove this sequence naturally splits into $m$ convergent subsequences in $\PML(S)$.

For each $h=0,...,m-1$ and $i\in\mathbb{N}$ let
 \begin{equation}\label{eq : cih} c^{h}_{i}=A(0,im+h) = \prod_{j=1}^{i-1} b e_{jm + h} .\end{equation}
where $A$ is defined in (\ref{eq : Aik}).

For each $h=0,1,...,m-1$, define the subsequence $\gamma^{h}_{i}$ of the sequence $\{\gamma_k\}_{k=0}^{\infty}$ by 
\begin{equation} \label{eq : gammaih} \gamma^{h}_{i}=\gamma_{im+h}.\end{equation}
  The main result of this section is the following theorem.
\begin{thm}\label{thm : MLlimitgi} 
Suppose $\{\gamma_k\}_{k=0}^\infty$ satisfies $\mathcal P$.  Then for each $h = 0,1,\ldots,m-1$, there exists a transverse measure $\bar \nu^h$ on $\nu$ so that 
\[ \lim_{i \to \infty} \frac{\gamma_i^h}{c_i^h} = \bar \nu^h \]
in $\ML(S)$, where $\gamma_i^h$ and $c_i^h$ are as above.
\end{thm}

We will need the following generalization of Theorem~\ref{thm : intgkgi}.  

\begin{lem} \label{lem : int estimate general}
For any curve $\delta$, there exists $\kappa(\delta) > 0$ and $N(\delta)>0$ so that for all $k \geq N(\delta)$
\begin{equation}\label{eq : int estimate general} i(\delta,\gamma_k)\stackrel{*}{\asymp}_{\kappa(\delta)} A(0,k).\end{equation}
\end{lem}
\begin{remark}Note that in Theorem~\ref{thm : intgkgi}, we estimate $i(\gamma_i,\gamma_k)$ with a uniform multiplicative constant $\kappa$ that works for any two curves $\gamma_i$ and $\gamma_k$, but the comparison is with $A(i,k)$ rather than $A(0,k)$.  On the other hand, the ratio of $A(0,k)$ and $A(i,k)$ is bounded by a constant depending on $i$, and not $k$, so the lemma for $\delta = \gamma_i$ is an immediate consequence of that theorem.
\end{remark}
\begin{proof}
First we note that by Theorem~\ref{thm : intgkgi}, we have
\[ i(\gamma_i,\gamma_k) \stackrel*{\asymp}_\kappa A(i,k).\]
From the definition of $A$, and the fact that $\{e_j\}_{j=0}^\infty$ is an increasing sequence, it follows that for each $i=0,\ldots,2m-1$, and all $k \geq i$, we have the bound
\[ 1 \leq \frac{A(0,k)}{A(i,k)} \leq b^2e_{2m}e_{3m}.\]
Setting $\kappa_0 = \kappa b^2e_{2m}e_{3m}$, for each $i=0,\ldots,2m-1$, we have
\begin{equation} \label{eq : temp i(i,k) bound} i(\gamma_i,\gamma_k) \stackrel*{\asymp}_{\kappa_0} A(0,k).
\end{equation}

Next, let $d = 2m\kappa_0$.  Note that since $\gamma_0,\ldots,\gamma_{2m-1}$ fills $S$, the set of measured laminations
\[ \Delta = \{ \bar{\lambda} \mid \sum_{j=0}^{2m-1} i(\gamma_j,\bar{\lambda}) \stackrel{*}{\asymp}_d 1 \} \subset \ML(S) \]
is compact. From \eqref{eq : temp i(i,k) bound} we have $\{\tfrac{\gamma_k}{A(0,k)}\}_{k=3m}^\infty \subset \Delta$. 

Let $\nu \in \EL(S)$ be the lamination from Proposition~\ref{prop : P implies loc-to-global}.  Since $\nu$ is an ending lamination, the set of measures $\bar \nu \in \Delta$ supported on $\nu$ is a compact subset.  By continuity of the intersection number $i$, there exists $c(\delta) > 0$ so that $i(\delta,\bar \nu) \stackrel{*}\asymp_{c(\delta)} 1$ for all such $\bar \nu$.  

Let $K(\delta) \subset \mathcal{ML}(S)$ be a compact neighborhood which contains the set of measures $\bar{\nu}$ which are supported on $\nu$ and are in $\Delta$. By continuity of $i$ again, we can take $K(\delta)$ sufficiently small so that there exists $\kappa(\delta) > 0$ such that $i(\delta,\bar{\lambda}) \stackrel*\asymp_{\kappa(\delta)} 1$ for all $\bar{\lambda} \in K(\delta)$.  Since every accumulation point of $\{\tfrac{\gamma_k}{A(0,k)}\}_{k=3m}^\infty$ is a measure $\bar \nu \in \Delta$ supported on $\nu$, it follows that there exists $N(\delta)$ so that for all $k \geq N(\delta)$, $\tfrac{\gamma_k}{A(0,k)} \in K(\delta)$.  Consequently, for all $k \geq N(\delta)$, we have $i(\delta,\gamma_k) \stackrel*\asymp_{\kappa(\delta)} A(0,k)$, which completes the proof.
\end{proof}


Using the estimates from Lemma~\ref{lem : int estimate general}, we prove the next lemma.  Theorem~\ref{thm : MLlimitgi} will then follow easily.
\begin{lem}\label{lem : convint}
For any curve $\delta$ and any $h=0,...,m-1$, the sequence $\{i(\delta,\frac{\gamma^{h}_{i}}{c^{h}_{i}})\}_{i=0}^\infty$ converges.
 \end{lem}
 \begin{proof}
By Proposition~\ref{prop : application of flp} we have that
$$\Big|i(\delta,\gamma_{im+h})-e_{(i-1)m+h}bi(\delta,\gamma_{(i-1)m+h})\Big|\leq 2B \sum_{l=(i-2)m+h}^{im+h-1}i(\delta,\gamma_l).$$
Dividing both sides by $c^h_i=A(0,im+h) = be_{(i-1)m+h}A(0,(i-1)m+h)$, and letting $\kappa(\delta)$ be the constant from Lemma~\ref{lem : int estimate general}, it follows that for all $h=0,\ldots,m-1$, and $i$ sufficiently large
\begin{eqnarray*}
\Big|i(\delta,\frac{\gamma_{im+h}}{c^{h}_i})-i(\delta,\frac{\gamma_{(i-1)m+h}}{c^{h}_{i-1}})\Big|&\leq&\frac{2B}{A(0,im+h)}\Big(\sum_{l=(i-2)m+h}^{im+h-1}i(\delta,\gamma_{l})\Big)\\
&\leq&\frac{2B}{A(0,im+h)}\Big(\sum_{l=(i-2)m+h}^{im+h-1}\kappa(\delta) A(0,l)\Big)\\
& = & \sum_{l=(i-2)m+h}^{im+h-1}2B\kappa(\delta) \frac{A(0,l)}{A(0,im+h)}
\end{eqnarray*}
Lemma~\ref{lem : A ratio bound} implies that the expressions in the final sum admit the following bounds:
\[ \frac{A(0,l)}{A(0,im+h)} \leq a^{1-\lfloor \tfrac{im+h-0}m \rfloor} = a^{1-i}\]
Since $\gamma_i^h = \gamma_{im+h}$, we have
\[ \Big|i(\delta,\frac{\gamma_i^h}{c^{h}_i})-i(\delta,\frac{\gamma_{i-1}^h}{c^{h}_{i-1}})\Big| \leq 4mB\kappa(\delta)a^{1-i}. \]
Consequently, for all $i> j$ sufficiently large, applying this inequality and the triangle inequality we have
\[ \Big|i(\delta,\frac{\gamma_i^h}{c^h_i})-i(\delta,\frac{\gamma_j^h}{c^h_j}) \Big| \leq 4mB\kappa(\delta) \sum_{l = j+1 }^i a^{1-l}.\]
By taking $i$ and $j$ sufficiently large, the (partial) sum of the geometric series on the right can be made arbitrarily small.  In particular, $\{i(\delta,\tfrac{\gamma_i^h}{c_i^h}) \}$ is a Cauchy sequence, hence converges.
\end{proof}

\begin{proof}[Proof of Theorem~\ref{thm : MLlimitgi}.]  Fix $h\in\{ 0,\ldots,m-1\}$.  Since the intersection numbers $\{i(\delta,\tfrac{\gamma_i^h}{c_i^h})\}_{i=0}^\infty$ converge for all simple closed curves $\delta$, it follows that $\{\tfrac{\gamma_i^h}{c_i^h} \}_{i=0}^\infty$ converges to some lamination $\bar \nu^h$ in $\ML(S)$ (since $\ML(S)$ is a closed subset of $\mathbb R^{\mathcal C(S)}$).   By Proposition~\ref{prop : P implies loc-to-global}, $\bar \nu^h$ is supported on $\nu$.
\end{proof}

\section{Ergodic measures}\label{sec : ergmeas}

We continue to assume throughout the rest of this section that $\{\gamma_k\}_{k=0}^\infty$ satisfies $\mathcal P$ and that $\{\frac{\gamma^h_i}{c^h_i}\}_{i=0}^\infty$ for $h = 0,\ldots,m-1$ are the subsequences defined in the previous section limiting to $\bar \nu^h$ supported on $\nu$ by Theorem~\ref{thm : MLlimitgi} for each $h = 0,\ldots,m-1$.  We say that $\bar \nu^h$ and $\bar \nu^{h'}$ are {\em not absolutely continuous} if neither is absolutely continuous with respect to the other one.  Note that this is weaker than requiring that the measures be mutually singular.

Recall from the introduction that the space of measures supported on $\nu$ is the cone on the simplex of measure $\Delta(\nu)$.   We denote (choices of) the ergodic measures representing the vertices by $\bar \mu^0,\ldots,\bar \mu^{d-1}$, where $0 \leq d \leq \xi(S)$ is the dimension of the space of measure on $\nu$.  The ergodic measures are mutually singular since the generic points are disjoint.  It follows that if we write $\bar \nu^h$ and $\bar \nu^{h'}$ as nonnegative linear combinations of $\bar \mu^0,\ldots,\bar \mu^{d-1}$, then $\bar \nu^h$ and $\bar \nu^{h'}$ are not absolutely continuous if and only if there exists $\bar \mu^j,\bar \mu^{j'}$ so that $\bar \mu^j$ has positive coefficient for $\bar \nu^h$ and zero coefficient for $\bar \nu^{h'}$, while $\bar \mu^{j'}$ has positive coefficient for $\bar \nu^{h'}$ and zero coefficient for $\bar \nu^h$.

The aim of this section is to show that $d = m$, and in particular, $\nu$ is nonuniquely ergodic.  In fact, we will prove that up to scaling and reindexing we have $\bar \mu^h = \bar \nu^h$.

Using the estimates on the intersection numbers from Theorem \ref{thm : intgkgi} we first show that the measures $\bar \nu^h$ for $h=0,...,m-1,$ are pairwise not absolutely continuous.
 \begin{thm}\label{thm : mutualsing}
Let $h,h'\in\{0,...,m-1\}$ and $h\neq h'$. Then
\[ \lim_{i \to \infty} \frac{i(\gamma^h_i,\bar \nu^h)}{i(\gamma^h_i,\bar \nu^{h'})} = \infty \quad \mbox{ and } \quad \lim_{i \to \infty} \frac{i(\gamma^{h'}_i,\bar \nu^{h'})}{i(\gamma^{h'}_i,\bar \nu^h)} = \infty.\]
In particular, the measures $\bar \nu^{h}$ and $\bar \nu^{h'}$ are not absolutely continuous with respect to each other.
 \end{thm}
 The last statement is a consequence of the two limits, for if $\bar \nu^h$ and $\bar \nu^{h'}$ were positive linear combinations of the same set of ergodic measures, then these ratios would have to be bounded. 
\begin{proof}
For $h \neq h'$, we will calculate that 
\begin{equation}\label{eq : ighighi+1}
 i(\gamma^{h}_{0},\gamma^{h}_{i+1})i(\gamma^{h}_{i},\bar \nu^{h})\stackrel{*}{\asymp} 1 \quad \mbox{ and } \quad \lim_{i\to\infty}i(\gamma^{h}_{0},\gamma^{h}_{i+1})i(\gamma^{h}_{i},\bar \nu^{h'})= 0. 
 \end{equation}
Dividing the first equation by the second and taking limit (and doing the same with the roles of $h$ and $h'$ reversed) gives the desired limiting behavior.

To treat the two estimates simultaneously, we suppose for the time being that $h,h' \in \{0,\ldots,m-1\}$, but we do not assume $h \neq h'$.  From Theorem~\ref{thm : MLlimitgi} together with (\ref{eq : cih}) and (\ref{eq : gammaih}) we have
\[ \bar \nu^h = \lim_{k \to \infty} \frac{\gamma^h_k}{c^h_k} = \lim_{k \to \infty} \frac{\gamma_{km+h}}{A(0,km + h)}.\]
Combining this with (\ref{eq : Aik}), (\ref{eq : gammaih}), and the estimate in Theorem~\ref{thm : intgkgi}, we see that for any $i$ we may take $k$ sufficiently large so that
\begin{eqnarray} i(\gamma^{h}_{0},\gamma^{h}_{i+1})i(\gamma^{h}_{i},\bar \nu^{h'}) & \stackrel{*}{\asymp} & i(\gamma_h,\gamma_{(i+1)m+h})i(\gamma_{im+h},\frac{\gamma_{km+h'}}{A(0,km+h')}) \notag \\
 & \stackrel{*}{\asymp} & \frac{A(h,(i+1)m+h)A(im+h,km+h')}{A(0,km+h')} \label{eq : lots of As}
\end{eqnarray}

We will simplify the expression on the right, but the precise formula depends on whether $h' \geq h$ or $h' < h$.  From the definition (\ref{eq : Aik}), the right hand side of (\ref{eq : lots of As}) can be written as
$$\frac{\prod_{r=1}^i b e_{rm+h} \prod_{r = j_0}^{k-1} b e_{rm + h'}}{\prod_{r = 1}^{k-1} be_{r m + h' }}= \frac{\prod_{r=1}^i b e_{rm+h}}{ \prod_{r =1}^{j_0-1} be_{rm + h'}}$$
 where $j_0 = i+1$ if $h' \geq h$ and $j_0 = i+2$ if $h' < h$.  Therefore, from (\ref{eq : lots of As}) we can write
\begin{equation} \label{eq : estimates for products} i(\gamma^{h}_{0},\gamma^{h}_{i+1})i(\gamma^{h}_{i},\bar \nu^{h'}) \stackrel{*}{\asymp} \left\{ \begin{array}{cc} \displaystyle{ \prod_{r=1}^i  \tfrac{e_{rm+h}}{e_{rm + h'}} }& h' \geq h\\ \displaystyle{\tfrac{1}{be_{m+h'}}\prod_{r=1}^i  \tfrac{e_{rm+h}}{e_{(r+1)m + h'}} } & h' < h \end{array} \right.
\end{equation}
Now observe that when $h' = h$, this becomes 
\[ i(\gamma^{h}_{0},\gamma^{h}_{i+1})i(\gamma^{h}_{i},\bar \nu^h) \stackrel{*}{\asymp} 1,\]
proving the first of the two required equations.  So, suppose $h \neq h'$.  Then each of the $i$ terms in the product is bounded above by $a^{-1}$ since the index for the denominator is greater than that of the numerator, and $e_l \geq a e_{l - 1}$ for all $l\geq 1$.  Thus we have
\[ i(\gamma^{h}_{0},\gamma^{h}_{i+1})i(\gamma^{h}_{i},\bar \nu^{h'}) \stackrel{*}{\prec}
 a^{-i} \] 
where when $h' < h$, we have absorbed the constant $b e_{m + h'}$  into the multiplicative error since $m+h' < 2m$.  Letting $i$ tend to infinity, we arrive at the second of our required estimates, and have thus completed the proof.
\end{proof}

We immediately obtain the following
\begin{cor}\label{cor : nue}
The lamination $\nu$ is nonuniquely ergodic.
\end{cor}

In fact, Theorem~\ref{thm : mutualsing} implies the main desired result of this section in a special case.  To prove this we first prove a lemma which will be useful in the general case as well.
\begin{lem} \label{lem : dimension suffices}  If $m \geq d$, then $m = d$, the measures $\bar \nu^0,\cdots, \bar \nu^{m-1}$ are distinct and ergodic, and these can be taken as the vertices of $\Delta(\nu)$.
\end{lem}
\begin{proof}  Recall that $\bar \mu^0,\ldots, \bar \mu^{d-1}$ are ergodic measures spanning the ($d$--dimensional) space of measures on $\nu$.  For each $0 \leq h < m$, write
\[ \bar \nu^h = \sum_{j=0}^{d-1} c_j^h \bar \mu^j, \]
where $c_j^h \geq 0$ for all $j,h$.  Then for each $i$, $h$, and $h'$, we have
\[ i(\gamma_i^h, \bar \nu^{h'}) = \sum_{j=0}^{d-1} c_j^{h'} i(\gamma_i^h,\bar \mu^j).\]

Next, fix $h$ and let $j_h \in \{0,\ldots,m-1\}$ be such that $c_{j_h}^h \neq 0$ and so that there exists a subsequence of $\gamma_i^h$, so that if $0 \leq j < m-1$ and $c_j^h \neq 0$, then
\begin{equation} \label{eqn : h-special subsequence} i(\gamma_i^h,\bar \mu^{j_h}) \geq i(\gamma_i^h,\bar \mu^j). \end{equation}

Now suppose that for some $h' \neq h$, $c_{j_h}^{h'} \neq 0$.  On the subsequence of $\{\gamma_i^h\}$ above where (\ref{eqn : h-special subsequence}) holds, Theorem~\ref{thm : mutualsing} implies
\begin{eqnarray*} \infty & = & \lim_{i \to \infty} \frac{\displaystyle{\sum_j c_j^h i(\gamma_i^h,\bar \mu^j)}}{\displaystyle{\sum_j c_j^{h'} i(\gamma_i^h,\bar \mu^j)}} \leq \limsup_{i \to \infty} \frac{\displaystyle{\sum_j c_j^h i(\gamma_i^h,\bar \mu^j)}}{c_{j_h}^{h'} i(\gamma_i^h,\bar \mu^{j_h})}\\
& = & \limsup_{i \to \infty} \sum_j \frac{c_j^h}{c_{j_h}^{h'}} \frac{i(\gamma_i^h,\bar \mu^j)}{i(\gamma_i^h,\bar \mu^{j_h})} \leq \sum_j \frac{c_j^h}{c_{j_h}^{h'}} < \infty.\end{eqnarray*}
This contradiction shows that $c_{j_h}^{h'} = 0$ for all $h' \neq h$.  Since $c_{j_h}^h \neq 0$, it follows that $h \mapsto j_h$ defines an injective function $\{0,\ldots,m-1\} \to \{0,\ldots,d-1\}$.  Since $m \geq d$, this function is a bijection, $m = d$, and $\bar \nu^h = c_{j_h}^h \bar \mu^{j_h}$.  Since $\bar \mu^0,\ldots,\bar \mu^{d-1}$ are distinct ergodic measures spanning the simplex of measures on $\nu$, the lemma follows.
\end{proof}

\begin{cor} \label{cor : special case m = xi} If $m = \xi(S)$, then the measures $\bar \nu^0,\cdots, \bar \nu^{m-1}$ are distinct and ergodic and can be taken as the vertices of $\Delta(\nu)$.
\end{cor}
\begin{proof} Since the dimension of the space of ergodic measures $d$ is at most $\xi(S)$, it follows that $m \geq d$, and hence Lemma~\ref{lem : dimension suffices} implies the result.
\end{proof}

\subsection{The general case} \label{sec : ergodic limits general case}

In \cite{lenmasdivergence} Lenzhen and Masur prove that for any nonuniquely ergodic lamination $\nu$ the ergodic measures are ``reflected'' in the geometric limit of a Teichm\"uller geodesic whose vertical foliation is topologically equivalent to $\nu$.  We will use this to prove the following generalization of Corollary~\ref{cor : special case m = xi} we need.

\begin{thm}  \label{thm : all ergodic measures} Suppose that $\{\gamma_l\}_{l=0}^\infty$ satisfies $\mathcal P$ and that $\{\gamma_k^h\}_{k=0}^\infty\;,h=0,...,m-1,$ is the partition into $m$ subsequences with $\displaystyle{\lim_{k \to \infty} \gamma_k^h = \bar \nu^h}$, all supported on $\nu$.  Then the measures $\bar \nu^0,\cdots, \bar \nu^{m-1}$ are distinct and ergodic and can be taken as the vertices of $\Delta(\nu)$.
\end{thm}

Let $\bar \mu^0,\ldots,\bar \mu^{d-1}$ be the ergodic measures on $\nu$ and set
\[ \bar \mu = \sum_{j=0}^{d-1} \bar \mu^j \quad \mbox{ and } \quad \bar \gamma = \sum_{j=0}^{m-1} \gamma_j = \sum_{h=0}^{m-1} \gamma_0^h.\]
Here we are viewing the curves in the sum on the right as measured laminations with transverse counting measure on each curve.  We choose a normalization for the measures $\bar \mu^j$ so that $i(\bar \gamma,\bar \mu) = 1$.
According to \cite{gardinermasur}, there is a unique complex structure on $S$ from a marked Riemann surface $S \to X$ and unit area holomorphic quadratic differential $q$ on $X$ with at most simple poles at the punctures, so that the vertical foliation $|dx|$ is $\bar \mu$ and the horizontal foliation $|dy|$ is $\bar \gamma$.
Area in the $q$--metric is computed by integrating $d\bar \mu|dy|$.
We will also be interested in the measure obtained by integrating $d \bar \mu_j |dy|$ for each $j = 0,\ldots, d-1$, which we denote by $\Area_j$.  Of course, $\Area = \sum_j \Area_j$.

Next let $g$ denote the Teichm\"uller geodesic defined by $q$.  We will write $g(t) = [f_t \colon X \to X(t)]$ where $X(t)$ is the terminal Riemann surface, or $g(t) = [f_t \colon (X,q) \to (X(t),q(t))]$, where $q(t)$ is the terminal quadratic differential. Note that since $\nu$ is a nonuniquely ergodic lamination by Masur's criterion \cite{masurcriterion} the geodesic $g$ is divergent in the moduli space.
The vertical and horizontal measure of a curve $\gamma$ is denoted $v_{q(t)}(\gamma)$ and $h_{q(t)}(\gamma)$, which are precisely the intersection numbers with the horizontal and vertical foliations of $q(t)$, respectively.  These are given by
\[ v_{q(t)}(\gamma) = e^{-t} i(\gamma, |dy|) = e^{-t} i(\gamma,\bar \gamma) \quad \mbox{and} \quad h_{q(t)}(\gamma) = e^ti(\gamma,|dx|) = e^t i(\gamma,\bar \mu).\]
From this it follows that the natural area measure from $q(t)$ is the push forward of the area measure from $q$.   Likewise, this area naturally decomposes as the push forward of the measures $\Area_j$, for $j = 0,\ldots, d-1$.   Consequently, we will often confuse a subset of $X$ and its image in $X(t)$ and will simply write $\Area$ and $\Area_j$ in either $X$ or $X(t)$.

Given $\epsilon > \epsilon' > 0$, an {\em $(\epsilon',\epsilon)$--thick subsurface} of $(X(t),q(t))$ is a compact surface $Y$ and a continuous map $Y \to X(t)$, injective on the interior of $Y$ with the following properties.
\begin{enumerate}
\item The boundary of $Y$ is sent to a union of $q(t)$--geodesics, each with extremal length less than $\epsilon'$ in $X(t)$.
\item If $Y$ is not an annulus, then every non-peripheral curve in $Y$ has $q(t)$--length at least $\epsilon$ and $Y$ has no peripheral Euclidean cylinders.
\item If $Y$ is an annulus, then it is a maximal Euclidean cylinder.  
\end{enumerate}
\begin{remark}
We will be interested in the case that $\epsilon' \ll \epsilon$.  In this case, $\partial Y$ has a large collar neighborhood in $Y$, which does not contain a Euclidean cylinder (i.e.~a large modulus expanding annulus; see \cite{rshteich}).  Consequently, $\partial Y$ will have short hyperbolic and extremal length.
\end{remark}
As an abuse of notation, we will write $Y \subset X$, although $Y$ is only embedded on its interior.  An {\em $(\epsilon',\epsilon)$--decomposition of $(X(t),q(t))$} is a union of $(\epsilon',\epsilon)$--thick subsurfaces $Y_1(t),\ldots,Y_r(t) \subset X(t)$ with pairwise disjoint interiors.  We note that $X(t)$ need not be the union of these subsurfaces.  For example, suppose $(X(t),q(t))$ is obtained from two flat tori by cutting both open along a very short segment, and gluing them together along the exposed boundary component.   If the area of one torus is very close to $1$ and the other very close to $0$, then an $(\epsilon',\epsilon)$--decomposition would consist of the larger slit torus, $Y(t)$, while $X(t) - Y(t)$ would be the (interior of the) smaller slit torus.

The key results from \cite{lenmasdivergence} we will need are summarized in the following theorem.

\begin{thm} [Lenzhen-Masur]  \label{thm : Lenzhen Masur}
With the assumptions on the Teichm\"uller geodesic $g$ above, there exists $\epsilon > 0$ and $B > 0$ with the following property.  Given any sequence of times $t_k \to \infty$, there exists a subsequence (still denoted $\{t_k\}$), a sequence of subsurfaces $Y_0(t_k),\ldots,Y_{d-1}(t_k)$ in $X(t_k)$, and a sequence $\epsilon_k \to 0$, so that for all $k\geq 1$
\begin{enumerate}
\item $Y_0(t_k),\ldots,Y_{d-1}(t_k)$ is an $(\epsilon_k,\epsilon)$--thick decomposition,
\item $\Area_j(Y_j^0(t_k)) > B$ for all $0 \leq j \leq d-1$ and for any component $Y_j^0(t_k) \subset Y_j(t_k)$,
\item $\Area_j(Y_i(t_k)) < \epsilon_k$ for all $0 \leq i,j \leq d-1$ with $i \neq j$, and
\item $\Area(X(t_k) - (Y_0(t_k) \cup \ldots \cup Y_{d-1}(t_k)) < \epsilon_k$.
\end{enumerate}
\end{thm}

The bulk of this theorem comes from Proposition 1 of \cite{lenmasdivergence}.  More precisely, in the proof of Proposition 1 given in \cite{lenmasdivergence}, the authors produce a sequence of subsurface $\{Y(t_k)\}$ whose components give an $(\epsilon_k,\epsilon)$--thick decomposition so that each component has area uniformly bounded away from zero, so that the areas of the complements tend to zero.  For each ergodic measure $\bar \mu^j$ the authors then find subsurfaces $Y_i(t_k)$ so that $\Area_j(Y_i(t_k)) \to 0$ as $k \to \infty$ if $i \neq j$ (see inequality (16) from \cite{lenmasdivergence} and its proof).  This proves (1), (3), and (4).  Since $\Area = \sum_j \Area_j$, (2) follows as well.

To apply this construction, we will need the following lemma.   First, for a curve $\gamma$ and $t\geq 0$, let $\cyl_t(\gamma) \subset X(t)$ denote the (possibly degenerate) maximal Euclidean cylinder foliated by $q(t)$--geodesic representatives of $\gamma$.  We note that $\cyl_t(\gamma) = f_t(\cyl_0(\gamma))$.

\begin{lem} \label{lem : always contain our cylinders} Given any sequence $t_k \to \infty$, let $Y_0(t_k),\ldots,Y_{d-1}(t_k) \subset X(t_k)$ denote the $(\epsilon_k,\epsilon)$--thick decomposition from Theorem~\ref{thm : Lenzhen Masur} (obtained after passing to a subsequence).  Then for all $k$ sufficiently large, each $Y_j(t_k)$ contains a curve from the sequence $\{\gamma_l\}$ as a non-peripheral curve, or else contains a component which is a cylinder with core curve in the sequence $\{\gamma_l\}$.
\end{lem}

We postpone the proof of this lemma temporarily and use it to easily prove the main result of this section.
\begin{proof}[Proof of Theorem~\ref{thm : all ergodic measures}.]
Let $t_k \to \infty$ be any sequence and $Y_0(t_k),\ldots,Y_{d-1}(t_k)$ the $(\epsilon_k,\epsilon)$--thick decomposition obtained from Theorem~\ref{thm : Lenzhen Masur} after passing to a subsequence.  Let $k$ be large enough so that the conclusion of Lemma~\ref{lem : always contain our cylinders} holds.  For each $j \in \{0,\ldots,d-1\}$ let $\gamma_{l_j}$ be one of the curves in our sequence so that $\gamma_{l_j}$ is either a nonperipheral curve in $Y_j(t_k)$, or else $Y_j(t_k)$ contains a cylinder component with core curve $\gamma_{l_j}$.   Since $Y_0(t_k),\ldots,Y_{d-1}(t_k)$ have disjoint interiors, it follows that $\gamma_{l_0},\ldots,\gamma_{l_{d-1}}$ are pairwise disjoint, pairwise nonisotopic curves.  By Theorem~\ref{thm : intgkgi}, for example, the difference in indices of disjoint curves in our sequence is at most $m$, and consequently $\{\gamma_{l_0},\ldots,\gamma_{l_{d-1}}\}$ consists of at most $m$ curves.  That is, $m \geq d$.   By Lemma~\ref{lem : dimension suffices}, $d=m$, and $\bar \nu^0,\ldots,\bar \nu^{m-1}$ are ergodic measures spanning the space of all measures on $\nu$, proving the theorem.
\end{proof}

\subsection{Areas and extremal lengths.}\label{subsec : area extlength}
 
The proof of the Lemma~\ref{lem : always contain our cylinders} basically follows from the results of \cite{rshteich}, together with the estimates on intersection numbers described at the beginning of this section and subsurface coefficient bounds in $\S$\ref{subsec : subsurfbd}.
Let $g(t) = [f_t \colon (X,q) \to (X(t),q(t)))]$ be the Teichm\"{u}ller geodesic described above with vertical foliation $\bar \mu = \sum \bar \mu_i$, the sum of the ergodic measures on $\nu$, and horizontal foliation $|dy| = \bar \gamma$. 

Suppose $Y \to X(t)$ is a map of a connected surface into $X(t)$ which is an embedding on the interior, sends the boundary to $q(t)$--geodesics, and has no peripheral Euclidean cylinders unless $Y$ is itself a Euclidean cylinder (in which case we assume it is maximal).   As in the case of thick subsurfaces, we write $Y \subset X(t)$, though we are not assuming that $Y$ is thick.  Suppose $Y \subset X(t)$ is a subsurface so that the leaves of the vertical and horizontal foliations intersect $Y$ in arcs.  This is the case for $Y = \cyl_t(\gamma_k)$ for all $k$ sufficiently large, as well as any $Y$ for which $\Ext_{X(t)}(\partial Y)$ is small when $t$ is large, and these will be the main cases of interest for us.

As in \cite{rshteich}, the surface $Y$ decomposes into a union of {\em horizontal strips} $Y = H_1(Y) \cup \ldots \cup H_r(Y)$ and {\em vertical strips} $Y = V_1(Y) \cup \ldots \cup V_{r'}(Y)$.  Each horizontal strip $H_i(Y)$ is the image of map $f^H_i \colon [0,1] \times [0,1] \to Y$ which is injective on the interior, sends $[0,1] \times \{s\}$ to an arc of a horizontal leaf with endpoints on $\partial Y$.  Furthermore, the images of the interiors of $f^H_1,\ldots,f^H_r$ are required to be pairwise disjoint.  Let $\ell^H_i = f^H_i([0,1] \times \{ \tfrac12 \})$ be a ``core arc'' of the strip.
Vertical strips are defined similarly (and satisfy the analogous properties for the vertical foliation) as are the core arcs $\ell^V_1,\ldots,\ell^V_{r'}$.
\begin{remark}  This is a slight variation on the strip decompositions in\cite{rshteich}.
\end{remark}

The {\em width} of a horizontal strip $H_i(Y)$, denoted $w(H_i(Y))$ is the vertical variation of any (or equivalently, every) arc $H_i(\{s\} \times [0,1])$.  The width of a vertical strip, $w(V_i(Y))$, is similarly defined in terms of the horizontal variation.  An elementary, but important property of these strips is the following.
\begin{prop}\label{prop : stripslengthbd} Let $Y \subset X(t)$ be as above. If
\[ Y = H_1(Y) \cup \ldots \cup H_r(Y) = V_1(Y) \cup \ldots V_{r'}(Y) \]
is a decomposition into maximal horizontal and vertical strips, then
\[ v_{q(t)}(\partial Y)=2\sum_{i=1}^r w(H_i(Y)) \quad \mbox{ and } \quad h_{q(t)}(\partial Y) = 2 \sum_{i=1}^{r'} w(V_i(Y)).\]
\end{prop}

The area of $Y$ can be estimated from this by the inequalities
\begin{equation} \label{eq : initial area bound strips}
\begin{array}{l}
\displaystyle{\sum_{ij} w(H_i(Y))w(V_j(Y)) (i(\ell^H_i,\ell^V_j) - 2) \leq \Area(Y)}\\
\hspace{3cm} \displaystyle{\leq \sum_{i,j} w(H_i(Y))w(V_j(Y))(i(\ell^H_i,\ell^V_j)+2)} \end{array}
\end{equation}
To see this, we note that the area of $Y$ is the sum of the areas of the horizontal (or vertical) strips.  Every time $V_j(Y)$ crosses $H_i(Y)$, it does so in a rectangle, which contains a unique point of intersection $\ell^H_i \cap \ell^V_j$, {\bf except}, near the ends of $H_i(Y)$ where we might not see an entire rectangle (and consequently we may or may not see a point of $\ell^H_i \cap \ell^V_j$).  We may also have an intersection point in $\ell^H_i \cap \ell^V_j$ that does not come in a complete rectangle (but only part of a rectangle).  Adding and subtracting $2$ to the intersection number accounts for the ends of $H_i(Y)$, and summing gives the bounds.

If $Y$ is non-annular, then note that
\[ \sum i(\ell^H_i,\ell^V_j) + 2 \prec i(\pi_Y(\bar \gamma),\pi_Y(\nu)).\]
To see this, we note that the horizontal foliation (for example) is $\bar \gamma$ and $\pi_Y(\bar \gamma)$ is basically obtained from the arcs $\ell^H_i$ by surgering with arcs from the boundary (see also Lemma 3.8 from \cite{rshteich}).
Combining this inequality with the upper bound in (\ref{eq : initial area bound strips}) and Proposition~\ref{prop : stripslengthbd} we obtain
\begin{equation} \label{eq : area bound nonannular strips}
\Area(Y) \prec h_{q(t)}(\partial Y)v_{q(t)}(\partial Y) i(\pi_Y(\bar \gamma),\pi_Y(\nu)).
\end{equation}

Now suppose that $Y = \cyl_t(\gamma)$ is a maximal Euclidean cylinder with core curve $\gamma$.  Then there is a decomposition into strips with just one horizontal strip $H(Y)$ and one vertical strip $V(Y)$ and core arcs $\ell^H$ and $\ell^V$, respectively.  In this case, the intersection number $i(\ell^H,\ell^V)$ is just $d_Y(\bar \gamma,\nu)$ up to an additive constant (of at most $4$---again, see Lemma 3.8 of \cite{rshteich}).
Therefore, the bounds in (\ref{eq : initial area bound strips}) together with Proposition~\ref{prop : stripslengthbd} implies
\begin{equation} \label{eq : area bound annular strip}
\frac{4\Area(\cyl_0(\gamma))}{h_{q(t)}(\gamma)v_{q(t)}(\gamma)} = \frac{4\Area(\cyl_0(\gamma))}{i(\gamma,\bar \gamma)i(\gamma,\bar \mu)} \stackrel+\asymp d_\gamma(\bar \gamma,\nu).
\end{equation}
In particular, if $d_{\gamma}(\bar \gamma,\nu)$ is large, then
\[ \Area(\cyl_0(\gamma_k^h)) \stackrel*\asymp h_{q(t)}(\gamma)v_{q(t)}(\gamma) d_\gamma(\bar \gamma, \nu) =i(\gamma,\bar \gamma)i(\gamma,\bar \mu) \gamma(\bar \gamma, \nu). \]

The balance time of $\gamma$ along the Teichm\"uller geodesic $g$ is the unique $t \in \mathbb R$ 
so that 
\[ v_{q(t)}(\gamma)=h_{q(t)}(\gamma). \]
Consider $Y = \cyl_{t(\gamma)}(\gamma)$ at the balance time of $\gamma$, together with the horizontal and vertical strips $H(Y)$ and $V(Y)$, respectively.  In this situation, the rectangles of intersections between $H(Y)$ and $V(Y)$ are actually squares.  We can estimate the modulus of $Y$, which is the ratio of the length to the circumference using these squares.  Specifically, we note that the circumference of $Y$ is precisely the length of the diagonal of a square, while the length of $Y$ is approximately half the number of squares, times the length of a diagonal.  Since the number of squares is $|\ell^H \cap \ell^V| \stackrel{+}{\asymp} d_\gamma(\bar \gamma,\nu)$, we see that the modulus is $2d_\gamma(\bar \gamma,\nu)$, up to a uniform additive error.  When $d_\gamma(\bar \gamma,\nu)$ is sufficiently large, the reciprocal of this modulus provides an upper bound for the extremal length
\begin{equation} \label{eq : extremal length bound} \Ext_{t(\gamma)}(\gamma) \stackrel{*}{\prec} \frac{1}{d_\gamma(\bar \gamma,\nu)}.
\end{equation}
We note that this estimate was under the assumption that $\cyl_0(\gamma)$ was a non-degenerate annulus.  In fact, if $d_\gamma(\bar \gamma,\nu)$ is sufficiently large (e.g.~at least $5$), then $\cyl_0(\gamma)$ is indeed nondegenarate.

\begin{proof}[Proof of Lemma~\ref{lem : always contain our cylinders}.]
Suppose that $t_k \to \infty$ is a sequence of times, $Y(t_k) \subset X(t_k)$ is a sequence of subsurfaces with $q(t)$--geodesic boundary, embedded on the interior and having no peripheral Euclidean cylinders, unless $Y$ is itself a Euclidean cylinder in which case we assume it is a maximal Euclidean cylinder.  We further assume that $\Ext_{X(t_k)}(\partial Y(t_k)) \to 0$.  We pass to a subsequence, also denoted $\{t_k\}$, and assume that either $Y(t_k)$ is nonannular and no nonperipheral curve lies in the sequence $\{\gamma_l\}$, or that $Y(t_k)$ is a cylinder whose core is not a curve from our sequence $\{\gamma_l\}$.
To prove the lemma, it suffices to prove that $\Area(Y(t_k)) \to 0$, for this implies that such subsurfaces $Y(t_k)$ cannot be a component of any $Y_j(t_k)$ from Theorem~\ref{thm : Lenzhen Masur}.

Decompose the sequence into an annular subsequence and non-annular subsequence, and we consider each case separately.  For the non-annular subsurfaces, we bound the area of $Y(t_k)$ using the inequality (\ref{eq : area bound nonannular strips}).  Specifically, we note that since no $\gamma_l$ is homotopic to a nonperipheral curve in $Y(t_k)$, Proposition~\ref{prop : anncoeff + coeffbd} provides a uniform bound for $d_W(\bar \gamma,\nu)$ for all subsurfaces $W \subset Y(t_k)$.  By  Theorem~\ref{thm : i=dY}, follows that $i(\pi_Y(\bar \gamma),\pi_Y(\nu))$ is uniformly bounded.  Since the extremal length of $\partial Y(t_k)$ is tending to zero, so is the $q(t_k)$--length, and so also the horizontal and vertical variations: 
\[ \lim_{k \to \infty} v_{q(t_k)}(\partial Y(t_k)) = 0 \quad \mbox{ and } \quad \lim_{k \to \infty}h_{q(t_k)}(\partial Y(t_k)) = 0. \]
Combining this with (\ref{eq : area bound nonannular strips}) proves $\Area(Y(t_k)) \to 0$, as required.

The annular case is similar: Again by Proposition~\ref{prop : anncoeff + coeffbd} since the core curve $\alpha_k$ of $Y(t_k)$ is not any curve from the sequence $\{\gamma_l\}$, we have that $d_{\alpha_k}(\bar \gamma,\bar \nu)$ is uniformly bounded, while the horizontal and vertical variations of $\alpha_k$ tend to zero (since the extremal length, and hence $q(t_k)$--length, tends to $0$).  Appealing to (\ref{eq : area bound annular strip}) proves that $\Area(Y(t_k)) \to 0$ as $k \to \infty$ in this case, too.
\end{proof}

\section{Constructions}\label{sec : constructions}

In this section we provide examples of sequences of curves satisfying $\mathcal P$, and hence to which the results of Sections 3-6 apply.
 
\subsection{Basic setup} \label{sec : basic setup for examples}

Consider a surface $S$ and pairwise disjoint, non-isotopic curves $\gamma_0,\ldots,\gamma_{m-1}$.  For each $k$, let $\Gamma_k = (\gamma_0 \cup \ldots \cup \gamma_{m-1}) - \gamma_k$, and let $X_k$ be the component of $S$ cut along $\Gamma_k$ containing $\gamma_k$.  For each $k$ we assume
\begin{enumerate}
\item $\partial X_k$ contains both $\gamma_{k+1}$ and $\gamma_{k-1}$ (with indices taken modulo $m$), 
\item we have chosen $f_k \colon S \to S$ a {\em fixed} homeomorphism which is the identity on $S \setminus X_k$, and pseudo-Anosov on $X_k$,
\item the composition of $f_k$ and the Dehn twist $\mathcal D_{\gamma_k}^r$, denoted $\mathcal D_{\gamma_k}^r f_k$, has translation distance at least $16$ on the arc and curve graph ${\mathcal AC}(X_k)$ for any $r \in \mathbb Z$, and
\item there is some $b > 0$ so that $i(\gamma_k,f_k(\gamma_k)) = b$, independent of $k$.
\end{enumerate}

For $0 \leq k,h \leq m-1$, let $\mathcal J(k,h)$ be the interval from $k$ to $h$, mod $m$. This means that if $k <h$ then $\mathcal J(k,h) = \{k,k+1,\ldots,h\}$ is the interval in $\mathbb Z$ from $k$ to $h$, while if $h < k$, then 
\[ \mathcal J(k,h) = \{k,k+1,\ldots, m-1,0,\ldots, h\}. \]
If $k = h$, then $\mathcal J(k,h) = \{k\} = \{h\}$.

For any $0 \leq k,h \leq m-1$, set
\[ X_{k,h} = \bigcup_{l \in \mathcal J(k,h)} X_l. \]
If $k = h$, note that $X_{k,h} = X_k = X_h$.  In general, $X_{k,h}$ is the component of $S$ cut along $\Gamma_{k,h} = \gamma_{h+1} \cup \ldots \cup \gamma_{k-1}$ containing all the curves $\gamma_k,\ldots,\gamma_h$.  That there is such a component follows inductively from the fact that $\gamma_{l \pm 1} \subseteq \partial X_l$, with indices taken mod $m$.

We also define
\[ F_{k,h} = f_k \circ f_{k+1} \circ \cdots \circ f_h.\]
where we are composing $f_l$ over $l \in \mathcal J(k,h)$.
Because $f_l$ is supported on $X_l$, it follows that for all $0 \leq k,h \leq m-1$,
\[ \gamma_k,\ldots,\gamma_h,F_{k,h}(\gamma_h) \subset X_{k,h}.\]
In fact, the first and last curves in this sequence fill $X_{k,h}$.
\begin{lem} \label{L : example filling}
For each $0 \leq k,h \leq m-1$
\[ \{ \gamma_k,F_{k,h}(\gamma_h) \} \]
fills $X_{k,h}$.  In particular, $i(\gamma_l,F_{k,h}(\gamma_h)) \neq 0$ for all $l \in \mathcal J(k,h)$.
\end{lem}
\begin{remark} \label{rem : general example fills S}
In the case $k = h+1$ (mod $m$), we note that $X_{h+1,h} = S$ and the lemma states that
\[ \{\gamma_{h+1},F_{h+1,h}(\gamma_h)\} = \{\gamma_k,f_kf_{k+1}\cdots f_h(\gamma_h)\} \]
fills $S$.  We also observe that for all $j \in \mathcal J(k,h)$, $X_{k,j} \subset X_{k,h}$.  It follows that $\gamma_k,\gamma_{k+1},\ldots,\gamma_h$ and $F_{k,k}(\gamma_k),\ldots,F_{k,h}(\gamma_h)$ are contained in $X_{k,h}$.
\end{remark}

In the following proof, we write $\pi_{k,h}(\delta)$ for the arc-projection to ${\mathcal AC}(X_{k,h})$ of a curve $\delta$.  This is just the isotopy class of arcs/curves of $\delta$ intersected with $X_{k,h}$.  Likewise, $d_{k,h}(\delta,\delta')$ is the distance between $\pi_{h,k}(\delta)$ and $\pi_{h,k}(\delta')$ in ${\mathcal AC}(X_{k,h})$.  We similarly define $\pi_k$ and $d_k$ for the case $k = h$.
\begin{proof}
The last statement follows from the first assertion since, for all $l \in \mathcal J(k,h)$, $i(\gamma_l,\gamma_k) = 0$, and so assuming $\{ \gamma_k,F_{k,h}(\gamma_h) \}$ fills, we must have $i(\gamma_l,F_{k,h}(\gamma_h)) \neq 0$.

The conditions on the curves and homeomorphisms are symmetric under cyclic permutation of the indices, so it suffices to prove the lemma for $h = m-1$ and $0 \leq k \leq h$ (which is slightly simpler notationally).  We write $j = h-k$ and must prove that $\{\gamma_{h-j},F_{h-j,h}(\gamma_h) \}$ fills $X_{h-j,h}$.  We prove this by induction on $j$.

The base case is $j=0$, in which case we are reduced to proving that $\{\gamma_h,f_h(\gamma_h)\}$ fills $X_h$.  This follows from the fact that $f_h$ has translation distance at least $16$ on ${\mathcal AC}(X_h)$, and hence $d_h(\gamma_h,f_h(\gamma_h)) \geq 16$.

Suppose that for some $0 < j \leq h$, $\{\gamma_{h+1-j},F_{h+1-j,h}(\gamma_h)\}$ fill $X_{h+1-j,h}$, and we must prove that $\{\gamma_{h-j},F_{h-j,h}(\gamma_h)\}$ fills $X_{h-j,h}$.

Note that since $\gamma_{h-j+1} \subset \partial X_{h-j}$, and $i(\gamma_{h-j+1},F_{h+1-j,h}(\gamma_h)) \neq 0$ (because they fill $X_{h+1-j,h}$), it follows that $F_{h+1-j,h}(\gamma_h)$ has nontrivial projection to $X_{h-j}$.  On the other hand, because $\gamma_{h-j}$ is disjoint from $X_{h+1-j,h}$ (it is in fact a boundary component), it follows that $i(\gamma_{h-j},F_{h+1-j,h}(\gamma_h)) = 0$, hence $d_{h-j}(\gamma_{h-j},F_{h+1-j,h}) = 1$.  Since $f_{h-j}$ translates by at least $16$ on ${\mathcal AC}(X_{h-j})$, it follows that
\begin{eqnarray*}
d_{h-j}(F_{h-j,h}(\gamma_h),\gamma_{h-j}) & = & d_{h-j}(f_{h-j}(F_{h+1-j,h}(\gamma_h)),\gamma_{h-j})\\
& \geq & d_{h-j}(f_{h-j}(F_{h+1-j,h}(\gamma_h)),F_{h+1-j,h}(\gamma_h))\\
& & \quad - d_{h-j}(F_{h+1-j,h}(\gamma_h),\gamma_{h-j})\\
& \geq & 16-1 = 15.
\end{eqnarray*}

Now suppose $\{\gamma_{h-j},F_{h-j,h}(\gamma_h)\}$ does not fill $X_{h-j,h}$.  Let $\delta$ be an essential curve in $X_{h-j,h}$ which is disjoint from both $\gamma_{h-j}$ and $F_{h-j,h}(\gamma_h)$.  Observe that $\delta$ cannot intersect the subsurface $X_{h-j}$ essentially, for otherwise
\[ d_{h-j}(\gamma_{h-j},F_{h-j,h}(\gamma_h)) \leq d_{h-j}(\gamma_{h-j},\delta) + d_{h-j}(\delta,F_{h-j,h}(\gamma_h)) \leq 2\]
a contradiction.

Therefore, $\delta$ is contained in $X_{h-j,h}-X_{h-j} \subset X_{h+1-j,h}$.  We first claim that $\delta$ must be an essential curve in $X_{h+1-j,h}$.  If not, then it is contained in the boundary.  However, any boundary component of $X_{h+1-j,h}$ which is essential in $X_{h-j,h}$ is contained (and essential) in $X_{h-j}$.  This is a contradiction.

Now since $\delta$ is essential in $X_{h+1-j,h}$, by the hypothesis of the induction we have
\[ 0 \neq i(\delta,\gamma_{h+1-j}) + i(\delta,F_{h+1-j,h}(\gamma_h)) = i(\delta,\gamma_{h+1-j}) + i(\delta,F_{h-j,h}(\gamma_h)).\]
The last equality follows from the fact that $F_{h-j,h}$ differs from $F_{h+1-j,h}$ only in $X_{h-j}$, which is disjoint from $\delta$.  Finally, we note that $\gamma_{h+1-j} \subseteq \partial X_{h-j}$, and hence $i(\delta,\gamma_{h+1-j}) =0$.  Consequently,
\[ i(\delta,F_{h-j,h}(\gamma_h)) \neq 0 \]
contradicting our choice of $\delta$.
Therefore, $\{\gamma_{h-j},F_{h-j,h}(\gamma_h) \}$ fills $X_{h-j,h}$.  This completes the induction, and hence the proof of the lemma.
\end{proof}

\begin{lem} \label{L : example filling2} For all $0 \leq k \leq m-1$,
\[ i(\gamma_k,F_{k,k-1}f_k(\gamma_k)) = i(\gamma_k,f_kf_{k+1}\cdots f_{k-1}f_k(\gamma_k)) \neq 0.\]
\end{lem}
\begin{proof}
We recall from the previous proof that $\{\gamma_{k+1},F_{k+1,k}(\gamma_k)\}$ not only fills $S$, but satisfies
\[ d_{k+1}(\gamma_{k+1},F_{k+1,k}(\gamma_k)) \geq 15.\]
Since $\gamma_{k+1} \subseteq \partial X_k$ and $\gamma_k \subseteq \partial X_{k+1}$ and $X_k$ and $X_{k+1}$ overlap, Theorem~\ref{thm : behineq} (see also Remark~\ref{rem : Behrstock sharp}) implies
\[ d_k(\gamma_k,F_{k+1,k}(\gamma_k)) \leq 4.\]
Since $f_k$ translates at least $16$ on ${\mathcal AC}(X_k)$, it follows that
\begin{eqnarray*}
d_k(\gamma_k,f_kF_{k+1,k}(\gamma_k)) & \geq & d_k(F_{k+1,k}(\gamma_k),f_kF_{k+1,k}(\gamma_k)) - d_k(\gamma_k,F_{k+1,k}(\gamma_k))\\
& \geq & 16 - 4 \geq 12.
\end{eqnarray*}
Since $f_kF_{k+1,k} = F_{k,k-1}f_k$, the lemma follows.
\end{proof}

\subsection{General construction}

Let $\{e_k\}_{k=0}^\infty$ be a sequence of integers satisfying Inequality  \eqref{eq : ek} for $a > 2$ sufficiently large as so as to satisfy (\ref{eq : restriction on a}) and hence (\ref{eq : intgkgi}) in Theorem~\ref{thm : intgkgi} (see the convention at the end of Section \ref{S : intersection number estimates}).


For $k \geq 0$, let $\bar{k}\in\{ 0,\ldots,m-1\}$ be the residue mod $m$, and for $k \geq m$ define
\[ \mathcal D_k = \mathcal D_{\gamma_{\bar k }}^{e_{k-m}} \mbox{ and } \phi_k = \mathcal D_k f_{\bar k}. \]
The sequence of curves $\{ \gamma_k \}_{k=0}^\infty$ is defined as follows.
\begin{enumerate}
\item The first $m$ curves are $\gamma_0,\ldots,\gamma_{m-1}$, as above.
\item For $k \geq m$, set
\[ \gamma_k = \phi_m \phi_{m+1}\cdots \phi_k(\gamma_{\bar k}).\]
\end{enumerate}
\begin{remark} \label{rem : annoying initial case} We could have avoided having the first $m$ curves as special cases and alternatively defined a sequence $\{\delta_k\}_{k \geq 0}$ by $\delta_k = \phi_0 \cdots \phi_k(\gamma_{\bar k})$ for all $k \geq 0$.  This sequence differs from ours by applying the homeomorphism $\phi_0\cdots \phi_{m-1}$.  This is a useful observation when it comes to describing consecutive elements in the sequence, but our choice allows us to keep $\gamma_0,\ldots,\gamma_{m-1}$ as the first $m$ curves.
\end{remark}
\begin{prop} \label{prop : general construction} With the conditions above, the sequence $\{\gamma_k\}_{k=0}^\infty$ satisfies $\mathcal P$ for some $0 < b_1 \leq b \leq b_2$ (where $b$ is the constant assumed from the start).
\end{prop} 
To simplify the proof, we begin with
\begin{lem} \label{lem : standard form for example} For any $2m$ consecutive curves $\gamma_{k-m},\ldots,\gamma_{k+m-1}$, there is a homeomorphism $H_{k}:S \to S$ taking these curves to the curves
\[ \gamma_{\bar k} \, , \, \ldots \, , \, \gamma_{\overline{k+m-1}} \, , \, f_{\bar k}(\gamma_{\bar k}) \, , \, \ldots \, , \, f_{\bar k}\cdots f_{\overline{k+m-1}}(\gamma_{\overline{k+m-1}})\]
(in the same order).
Furthermore, the homeomorphism can be chosen to take $\gamma_{k+m}$ to
\[ \mathcal D_{f_{\bar k}(\gamma_{\bar k})}^{e_k}(f_{\bar k} \cdots f_{\overline{k+m-1}}f_{\bar k}(\gamma_{\bar k})). \]
\end{lem}
\begin{proof}
We prove the lemma assuming $k \geq 2m$ to avoid special cases (the general case can be easily derived from Remark~\ref{rem : annoying initial case}, for example).  We define
\[ H_k = (\phi_m \cdots \phi_{k-1} \mathcal D_k \mathcal D_{k+1} \cdots \mathcal D_{k+m-1})^{-1}.\]

Let $h,h'\in\{0,...,m-1\}$ and note that since $i(\gamma_h,\gamma_{h'}) = 0$, $\mathcal D_{\gamma_{h'}}(\gamma_h) = \gamma_h$.  Furthermore, if $h \neq h'$, from the fact that $f_h$ is supported on $X_h$ and $\gamma_{h'}$ is disjoint from $X_h$ we easily deduce $\mathcal D_{\gamma_{h'}}$ and $f_h$ commute, and $\phi_h(\gamma_{h'}) = \gamma_{h'}$.

From these facts we observe that for $k - m \leq j \leq k-1$, we have
\begin{eqnarray*} H_k^{-1}(\gamma_j) & = & \phi_m \cdots \phi_{k-1} \mathcal D_k \mathcal D_{k+1} \cdots \mathcal D_{k+m-1}(\gamma_{\bar j})\\
& = & \phi_m \cdots \phi_{k-1}(\gamma_{\bar j})\\
& = & \phi_m \cdots \phi_j(\gamma_{\bar j}) = \gamma_j,
\end{eqnarray*}
while for $k \leq j \leq k+m-1$, we have
\begin{eqnarray*}
H_k^{-1}(f_{\bar k} \cdots f_{\bar j}(\gamma_{\bar j})) & = & \phi_m \cdots \phi_{k-1} \mathcal D_k \cdots \mathcal D_{k+m-1} f_{\bar k} \cdots f_{\bar j}(\gamma_j)\\
& = & \phi_m \cdots \phi_{k-1} \mathcal D_k f_{\bar k} \cdots \mathcal D_j f_{\bar j} \mathcal D_{j+1} \cdots \mathcal D_{k+m-1}(\gamma_{\bar j})\\
& = & \phi_m \cdots \phi_j \mathcal D_{j+1} \cdots \mathcal D_{k+m-1}(\gamma_{\bar j})\\
& = & \phi_m \cdots \phi_j(\gamma_{\bar j}) = \gamma_j.
\end{eqnarray*}
This completes the proof of the first statement.

Next, since $\mathcal D_{k+m} = \mathcal D_{\gamma_{\bar k}}^{e_k}$, we have
\begin{eqnarray} \label{eq : identifying the dehn twist} \notag f_{\bar k} \cdots f_{\overline{k+m-1}} \phi_{k+m}(\gamma_{\bar k}) & = & f_{\bar k} \cdots f_{\overline{k+m-1}} \mathcal D_{k+m} f_{\bar k}(\gamma_k)\\ \notag
& = & f_{\bar k } \mathcal D_{k+m} f_{\overline{k+1}} \cdots f_{\overline{k+m-1}} f_{\bar k}(\gamma_{\bar k })\\
& = & f_{\bar k} \mathcal D_{k+m} f_{\bar k}^{-1} f_{\bar k} \cdots f_{\overline{k+m-1}} f_{\bar k}(\gamma_{\bar k})\\ \notag
& = & f_{\bar k} \mathcal D_{\gamma_{\bar k}}^{e_k}f_{\bar k}^{-1} f_{\bar k} \cdots f_{\overline{k+m-1}} f_{\bar k}(\gamma_{\bar k})\\ \notag
& = & \mathcal D_{f_{\bar k}(\gamma_{\bar k})}^{e_k} f_{\bar k } \cdots f_{\overline{k+m-1}}f_{\bar k}(\gamma_{\bar k}).
\end{eqnarray}
Applying $H_k^{-1}$ to the left hand side gives $\gamma_{k+m}$, proving the last statement.
\end{proof}

\begin{proof}[Proof of Proposition~\ref{prop : general construction}.]
For any $2m$ consecutive curves in our sequence, $\gamma_{k-m},\ldots,\gamma_{k+m-1}$, let $H_k \colon S \to S$ be the homeomorphism from Lemma~\ref{lem : standard form for example} putting these curves into the standard form described by that lemma.  Since $H_k$ sends the first $m$ to $\gamma_{\bar k},\ldots,\gamma_{\overline{k+m-1}}$, it follows that these curves are pairwise disjoint.  Moreover, the set of all $2m$ curves fills $S$ by Lemma~\ref{L : example filling} and Remark~\ref{rem : general example fills S}  (in fact, the first and last alone fill $S$).
Therefore, the sequence satisfies conditions (i) and (ii) of $\mathcal P$.  

To prove that condition (iii) is also satisfied we need to define $\gamma_{k+m}'$ so that $\gamma_{k+m} = \mathcal D_{\gamma_k}^{e_k}(\gamma_{k+m}')$, and verify the intersection conditions.
We fix $k \geq 2m$ and define
\[ \gamma_{k+m}' = \phi_m \cdots \phi_{m+k-1} f_{\bar k}(\gamma_{\bar k})\]
(the case of general $k \geq m$ is handled by special cases or by appealing to Remark~\ref{rem : annoying initial case}).
Note that by definition, $\gamma_{k+m} = \phi_m \cdots \phi_{m+k-1} \phi_{m+k}(\gamma_{\bar k})$ and applying $H_k$ to $\gamma_k$ and $\gamma_{k+m}$, Lemma~\ref{lem : standard form for example} gives us
\[ H_k(\gamma_k) = f_{\bar k}(\gamma_{\bar k}) \quad \mbox{ and } \quad H_k(\gamma_{k+m}) = \mathcal D_{f_{\bar k}(\gamma_{\bar k})}^{e_k}(f_{\bar k} \cdots f_{\overline{k+m-1}}f_{\bar k}(\gamma_{\bar k})) \]
Then, as in the proof of Lemma~\ref{lem : standard form for example} (compare \eqref{eq : identifying the dehn twist}), we have
\[ H_k(\gamma_{k+m}') =  f_{\bar k} \cdots f_{\overline{k+m-1}}f_{\bar k}(\gamma_{\bar k})
\]
Therefore,
\[H_k(\gamma_{k+m}) = \mathcal D_{H_k(\gamma_k)}^{e_k}(H_k(\gamma_{k+m}')) = H_k(\mathcal D_{\gamma_k}^{e_k}(\gamma_{k+m}')),\]
so $\gamma_{k+m} = \mathcal D_{\gamma_k}^{e_k}(\gamma_{k+m}')$.

To prove the intersection number conditions on $i(\gamma_{k+m}',\gamma_j)$ from property (iii) of $\mathcal P$, it suffices to prove them for the $H_k$--images.  Thus, for $j \in \{k+1,\ldots,k+m-1 \}$ we note that by Lemma~\ref{lem : standard form for example}, $H_k(\gamma_j) = f_{\bar k} \cdots f_{\bar j}(\gamma_{\bar j})$, and hence
\begin{eqnarray*} i(\gamma_j,\gamma_{k+m}') & = & i(f_{\bar k} \cdots f_{\bar j}(\gamma_{\bar j}),f_{\bar k} \cdots f_{\overline{k+m-1}}f_{\bar k}(\gamma_{\bar k}))\\
& = & i(\gamma_{\bar j},f_{\overline{j+1}}\cdots f_{\overline{k+m-1}}f_{\bar k}(\gamma_{\bar k}))\\
& = & i(\gamma_{\bar j},\gamma_{\bar k}) = 0.
\end{eqnarray*}
The second-to-last equality is obtained by applying $(f_{\overline{j+1}} \cdots f_{\overline{k+m-1}}f_{\bar k})^{-1}$ to both entries, and observing that this fixes $\gamma_{\bar j}$ (c.f.~the proof of Lemma~\ref{lem : standard form for example}).

On the other hand, for $j=k$, the same basic computation shows
\[ i(\gamma_k,\gamma_{k+m}') = i(\gamma_{\bar k},f_{\bar k}(\gamma_{\bar k})) = b\]
by assumption (4).

Finally, similar calculations show that for $j \in \{k-m,\ldots,k-1\}$, by Lemmas~\ref{L : example filling} and \ref{L : example filling2}, we have
\[ i(\gamma_j,\gamma_{k+m}') = i(\gamma_{\bar j},f_{\bar k} \cdots f_{\overline{k+m-1}}f_{\bar k}(\gamma_{\bar k})) \neq 0.\]
There are only finitely many possible choices of $\bar j$ and $\bar k$, so the values are uniformly bounded between two constants $b_1 < b_2$.  Without loss of generality, we may assume $b_1 \leq b \leq b_2$.  This completes the proof.
\end{proof}

While any sequence of curves as above satisfies the conditions in sections in $\mathcal P$ from Definition~\ref{def : P}, we will need one more condition when analyzing the limits of Teichm\"uller geodesics.  It turns out that any construction as above also satisfies this property.  We record this property here for later use.

\begin{lem}\label{lem : P(iv)}  Suppose the sequence $\{\gamma_k\}_{k =0}^\infty$ is constructed as above.  If $\gamma_k,\gamma_h$ are any two curves with $m \leq h-k < 2m-1$, then $\gamma_k$ and $\gamma_h$ fill a subsurface whose boundary consists entirely of curves in the sequence.   Furthermore, for any $k \leq j \leq h$, $\gamma_j$ is either contained in this subsurface, or is disjoint from it.  If $h-k \geq 2m-1$, then $\gamma_k$ and $\gamma_h$ fill $S$.
\end{lem}
\begin{proof}
First assume $m \leq h-k \leq 2m-1$.  Applying the homeomorphism $H_k \colon S \to S$ from Lemma~\ref{lem : standard form for example}, $\gamma_k$ and $\gamma_h$ are sent to $\gamma_{\bar k}$ and $f_{\bar k} \cdots f_{\bar h}(\gamma_{\bar h}) = F_{\bar k, \bar h}(\gamma_{\bar h})$, respectively.
This fills the surface $X_{\bar k,\bar h}$ which has boundary contained in $\gamma_0 \cup \ldots \cup \gamma_{m-1}$.   By Lemma~\ref{L : example filling} it follows that $H_k^{-1}(X_{\bar k,\bar h})$ is filled by $\{\gamma_k,\gamma_h\}$ and has boundary in $H_k(\gamma_0) \cup \ldots \cup H_k(\gamma_{m-1})$.  All the components of this multicurve are in our sequence, as required for the first statement.

For each $k \leq j \leq h-m$ and $k+m \leq j \leq h$, $\bar j \in \mathcal J(\bar k,\bar h)$, and as pointed out in Remark~\ref{rem : general example fills S}, $\gamma_{\bar j}$ and $F_{\bar k,\bar j}(\gamma_{\bar j})$ are contained in $X_{\bar k,\bar h}$.   Consequently, for these values of $j$, $\gamma_j \in H_k(X_{\bar k,\bar h})$.   On the other hand, if $k < j < h$, and $j$ does not fall into one of the above two cases, then $h-m+1 \leq j \leq k+m-1$, which implies $0 \leq j-k,h-j \leq m-1$ and hence $i(\gamma_j,\gamma_k) = i(\gamma_j,\gamma_h) = 0$, and hence $\gamma_j$ is disjoint from $H_k(X_{\bar k,\bar h})$.  This completes the proof of the second statement.

When $h-k = 2m-1$, $X_{\bar k,\bar h} = S$, and hence $\{\gamma_{\bar k},F_{\bar k,\bar h}(\gamma_{\bar h})\}$ fills $S$.   Consequently $\{\gamma_k,\gamma_h\}$ also fills $S$.

Now we must prove that for $h-k \geq 2m-1$, that $\gamma_k$ and $\gamma_h$ fill $S$.  The proof is by induction, but we need a little more information in the induction.  For simplicity, we assume that $k \geq m+1$ to avoid special cases.  

To describe the additional conditions, for $k < l$,  let $\Phi_l = \phi_m \cdots \phi_l$, so that $\Phi_{k+m-1}^{-1}$ sends the curves $\gamma_k,\ldots,\gamma_h$ (in order) to the curves
\[ \gamma_{\bar k},\ldots,\gamma_{\overline{k+m-1}},\phi_{k+m}(\gamma_{\overline{k+m}}),\ldots,\phi_{k+m}\cdots\phi_h(\gamma_{\bar h}).\]
With this notation, we now wish to prove by double induction (on $k$ and $h-k$) that for all $m+1 \leq k < h$ with $h-k \geq 2m-1$ we have
\[ \{\gamma_k,\gamma_h\} \mbox{ fills } S \mbox{ and } d_{\Phi_{k+m-1}(X_{\bar k})}(\gamma_k,\gamma_h) \geq 12. \]
The base case is $h-k = 2m-1$ and any $k \geq m+1$.  We have already pointed out that $\{\gamma_k,\gamma_h\}$ fills $S$.  We note that applying $\Phi_{k+m}^{-1}$ takes $\gamma_{k+1},\ldots,\gamma_h$ to
\[ \gamma_{\overline{k+1}},\ldots,\gamma_{\overline{k+m}},\phi_{k+m+1}(\gamma_{\overline{k+m+1}}),\ldots, \phi_{k+m}\cdots\phi_{k+2m-1}(\gamma_{\overline{k+2m-1}}).
\]
For the first and last curves $\{\gamma_{\overline{k+1}},\phi_{k+m+1}\cdots\phi_{k+2m-1}(\gamma_{\overline{k+2m-1}})\}$ we see that these fill $X_{\overline{k+1},\overline{k+2m-1}} = X_{\overline{k+1},\overline{k-1}}$ which has $\gamma_{\bar k}$ as a boundary component. Since $\pi_{\bar k}(\phi_{k+m+1}\cdots\phi_{k+2m-1}(\gamma_{\overline{k+2m-1}}))$ is disjoint from $\gamma_{\bar k}$ it follows that applying $\phi_{k+m}$ to this last curve $\phi_{k+m+1}\cdots\phi_{k+2m-1}(\gamma_{\overline{k+2m-1}})$ we have
\[ d_{\bar k}(\gamma_{\bar k},\phi_{k+m}\phi_{k+m+1}\cdots \phi_{k+2m-1}(\gamma_{\overline{k+2m-1}})) \geq 14 > 12.\]
But notice that $\Phi_{k+m-1}^{-1}(\gamma_{k+2m-1}) = \phi_{k+m}\cdots \phi_{k+2m-1}(\gamma_{\overline{k+2m-1}})$ while on the other hand $\Phi_{k+m-1}^{-1}(\gamma_k) = \gamma_{\bar k}$, hence
\[ d_{\Phi_{k+m-1}(X_{\bar k})}(\gamma_k,\gamma_{k+2m-1}) \geq 12,\]
as required for the base case.

For the induction step, the proof is quite similar. We assume that the statement holds for all $k \geq m+1$ and all $2m-1 \leq h-k \leq N$, and prove it for $h-k=N+1$.  Since $h-(k+1) = N$ and $k+1 \geq m+2 \geq m+1$, by the inductive assumption it follows that $\{\gamma_{k+1},\gamma_h\}$ fills $S$ and that 
\[ d_{\Phi_{k+m}(X_{\overline{k+1}})}(\gamma_{k+1},\gamma_h) \geq 12.\]
Therefore, applying $\Phi_{k+m}^{-1}$, we have
\[ d_{\overline{k+1}}(\gamma_{\overline{k+1}},\phi_{k+m+1} \cdots \phi_h(\gamma_{\bar h})) \geq 12.\]
The homeomorphism $\Phi_{k+m}^{-1}$ sends $\gamma_k,\ldots,\gamma_h$ to the sequence
\[ \phi_{k+m}^{-1}(\gamma_{\bar k}),\gamma_{\overline{k+1}},\ldots,\gamma_{\overline{k+m}},\phi_{k+m+1}(\gamma_{\overline{k+1}}),\ldots,\phi_{k+m+1} \cdots \phi_h(\gamma_{\bar h}).\]
Since $\gamma_{\bar k} \subset \partial X_{\overline{k+1}}$ and $\gamma_{\overline{k+1}} \subset \partial X_{\bar k}$, Theorem~\ref{thm : behineq} (see also Remark~\ref{rem : Behrstock sharp}) ensures that we have
\[ d_{\bar k}(\gamma_{\bar k},\phi_{k+m+1} \cdots \phi_h(\gamma_{\bar h})) \leq 4.\]
Applying $\phi_{k+m}$ (which translates by at least $16$ on $\mathcal C(X_{\bar k})$) to the second curve  we get
\[ d_{\bar k}(\gamma_{\bar k},\phi_{k+m}\phi_{k+m+1}\cdots \phi_h(\gamma_{\bar h}) \geq 12.\]
In particular, we have
\[ d_{\Phi_{k+m-1}(X_{\bar k})}(\gamma_k,\gamma_h) \geq 12.\]
This proves part of the requirement on $\gamma_k,\gamma_h$.  

We must also show that $\{\gamma_k,\gamma_h\}$ fills $S$.  We will show that the $\Phi_{k+m-1}$--image, $\{\gamma_{\bar k},\phi_{k+m} \cdots \phi_h(\gamma_{\bar h})\}$ fills $S$, which will suffice.  To see this, take any essential curve $\delta$ and suppose it is disjoint from both $\gamma_{\bar k}$ and $\phi_{k+m} \cdots \phi_h(\gamma_{\bar h})$.  Then note that $\delta$ must have empty projection to $X_{\bar k}$, for otherwise the triangle inequality implies that the distance from $\pi_{\bar k}(\gamma_{\bar k})$ to $\pi_{\bar k}(\phi_{k+m}\cdots \phi_h(\gamma_{\bar h}))$ is at most $4$, a contradiction to the fact that
\[ d_{\bar k}(\gamma_{\bar k},\phi_{k+m+1} \cdots \phi_h(\gamma_{\bar h})) = d_{\Phi_{k+m-1}(X_{\bar k})}(\gamma_k,\gamma_h) \geq 12.\]  
Since $\{\gamma_{\overline{k+1}},\phi_{k+m+1}\cdots \phi_h(\gamma_{\bar h})\}$ fills $S$, $\delta$ must intersect one of these curves.  However, $\gamma_{\overline{k+1}}$ is contained in the boundary of $X_{\bar k}$, and hence $\delta$ is disjoint from this.  Consequently, $\delta$ must intersect $\phi_{k+m+1} \cdots \phi_h(\gamma_{\bar h})$.  Since $\phi_{k+m}$ is supported on $X_{\bar k}$ which is disjoint from $\delta$ we have
\begin{eqnarray*}
0 \neq i(\delta,\phi_{k+m+1}\cdots \phi_h(\gamma_{\bar h})) &=& i(\phi_{k+m}^{-1}(\delta),\phi_{k+m+1}\cdots \phi_h(\gamma_{\bar h})) \\
&=& i(\delta,\phi_{k+m}\cdots \phi_h(\gamma_{\bar h})).
\end{eqnarray*}
This contradicts our initial assumption on $\delta$, and hence no such $\delta$ exists and $\{\gamma_{\bar h},\phi_{k+m}\cdots \phi_h(\gamma_{\bar h})\}$ fills $S$ as required.  This completes the proof.
\end{proof}

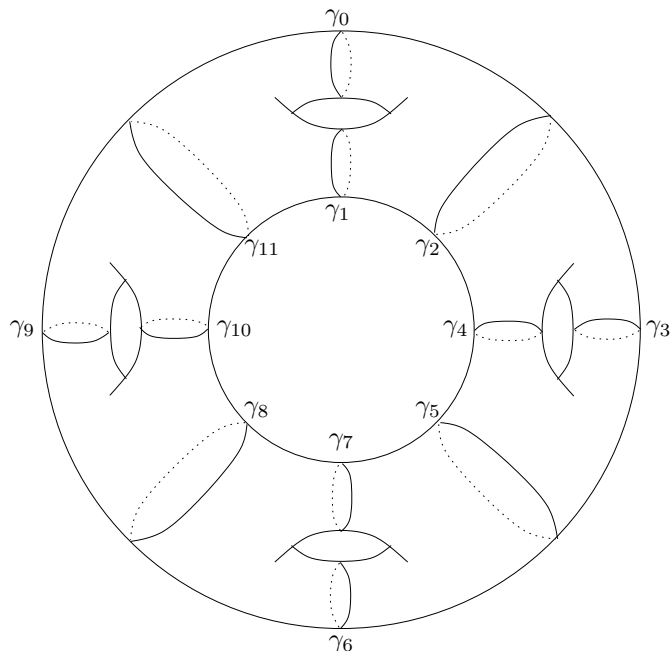
\begin{figure} [htb]
\labellist
\small\hair 2pt
\pinlabel $\gamma_0$ [b] at 161 325
\pinlabel $\gamma_1$ [b] at 161 219
\pinlabel $\gamma_2$ [b] at 210 200
\pinlabel $\gamma_3$ [b] at 335 156
\pinlabel $\gamma_4$ [b] at 225 156
\pinlabel $\gamma_5$ [b] at 210 112
\pinlabel $\gamma_6$ [b] at 163 -15
\pinlabel $\gamma_7$ [b] at 163 95
\pinlabel $\gamma_8$ [b] at 117 112
\pinlabel $\gamma_9$ [b] at -10 156
\pinlabel $\gamma_{10}$ [b] at 105 156
\pinlabel $\gamma_{11}$ [b] at 120 200
\endlabellist
\begin{center}
\includegraphics[width=8cm]{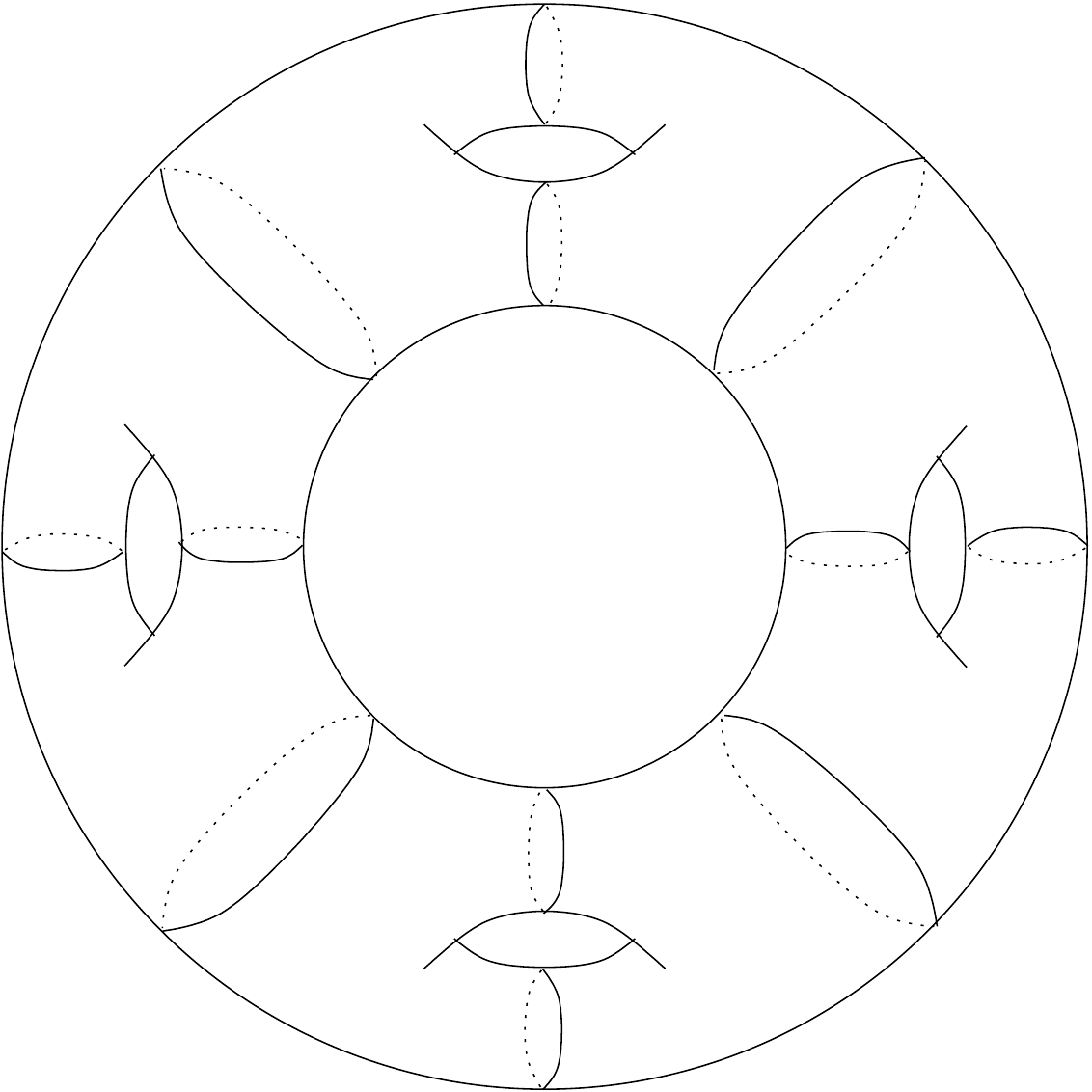} \caption{The pairwise disjoint curves $\gamma_0,\ldots,\gamma_{m-1}$ for the first family of examples in the case of genus $5$ (and hence $m = 12$).}
\label{F : genus5example3}
\end{center}
\end{figure}

\subsection{Specific examples}
Here we provide two specific families of examples of the general construction, but it is quite flexible and easy to build many more examples.  We need to describe $\gamma_0,\ldots,\gamma_{m-1}$, together with the rest of the data from the beginning of Section~\ref{sec : basic setup for examples}.   For this, we will first ensure that all of our subsurfaces $X_k$ have the property that $\gamma_{k\pm 1} \subseteq \partial X_k$ (indices mod $m$).  This is the first of the four conditions required.  For the other three conditions, it will be enough to choose the sequence so that for any $0 \leq k,h \leq m-1$, there is a homeomorphism of pairs $(X_k,\gamma_k) \cong (X_h,\gamma_h)$.  For then, we can choose $f_0 \colon S \to S$ any homeomorphism which is the identity on $S \setminus X_k$, pseudo-Anosov on ${\mathcal AC}(X_k)$ with translation distance at least $15$, and then use the homeomorphisms $(X_0,\gamma_0) \cong (X_k,\gamma_k)$ to conjugate $f_0$ to homeomorphisms $f_k \colon S \to S$.

\subsubsection{Maximal dimensional simplices}\label{subsec : max}
For the first family of examples, we can choose a pants decomposition on $S_{g,0}$ a closed genus $g \geq 3$ surface as shown in Figure~\ref{F : genus5example3}.  Each $X_k$ is homeomorphic to a $4$--holed sphere, and $\gamma_k \subset X_k$ is an essential curve.  Any two $(X_k,\gamma_k)$ and $(X_h,\gamma_h)$ are clearly homeomorphic pairs.  In this case $m = 3g-3$, and the limiting lamination $\nu$ from Proposition~\ref{prop : P implies loc-to-global} defines a simplex of measures with maximal possible dimension in $\PML(S)$ by Theorem~\ref{thm : all ergodic measures}.  One can also construct examples in genus $2$ by taking $\gamma_0,\gamma_1,\gamma_2$ to be a pants decomposition of non-separating curves.

\subsubsection{Non-maximal examples}\label{subsec : nonmax}
For our second family, we choose $m = g-1$, and take a sequence $\gamma_0,\ldots,\gamma_{m-1}$ as shown in Figure~\ref{F : genus5example2}.  Here each $X_k$ is homeomorphic to a surface of genus $2$ with two boundary components and $\gamma_k$ is a curve that cuts $X_k$ into two genus $1$ surfaces with two boundary components.

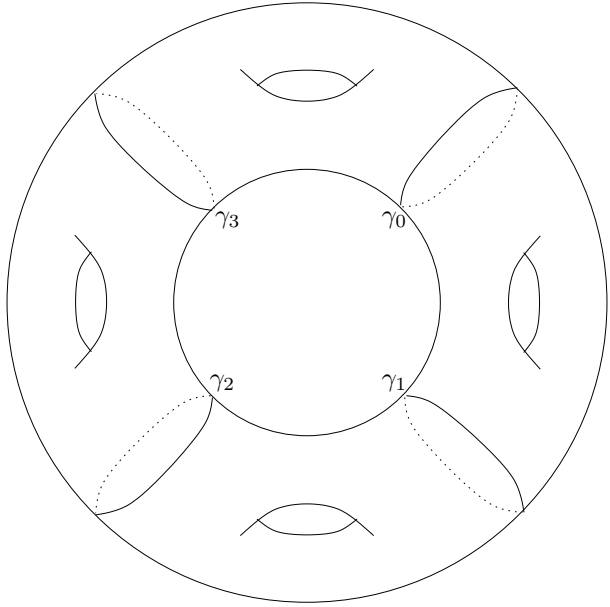
\begin{figure} [htb]
\labellist
\small\hair 2pt
\pinlabel $\gamma_0$ [b] at 210 200
\pinlabel $\gamma_1$ [b] at 210 112
\pinlabel $\gamma_2$ [b] at 117 112
\pinlabel $\gamma_3$ [b] at 120 200
\endlabellist
\begin{center}
\includegraphics[width=8cm]{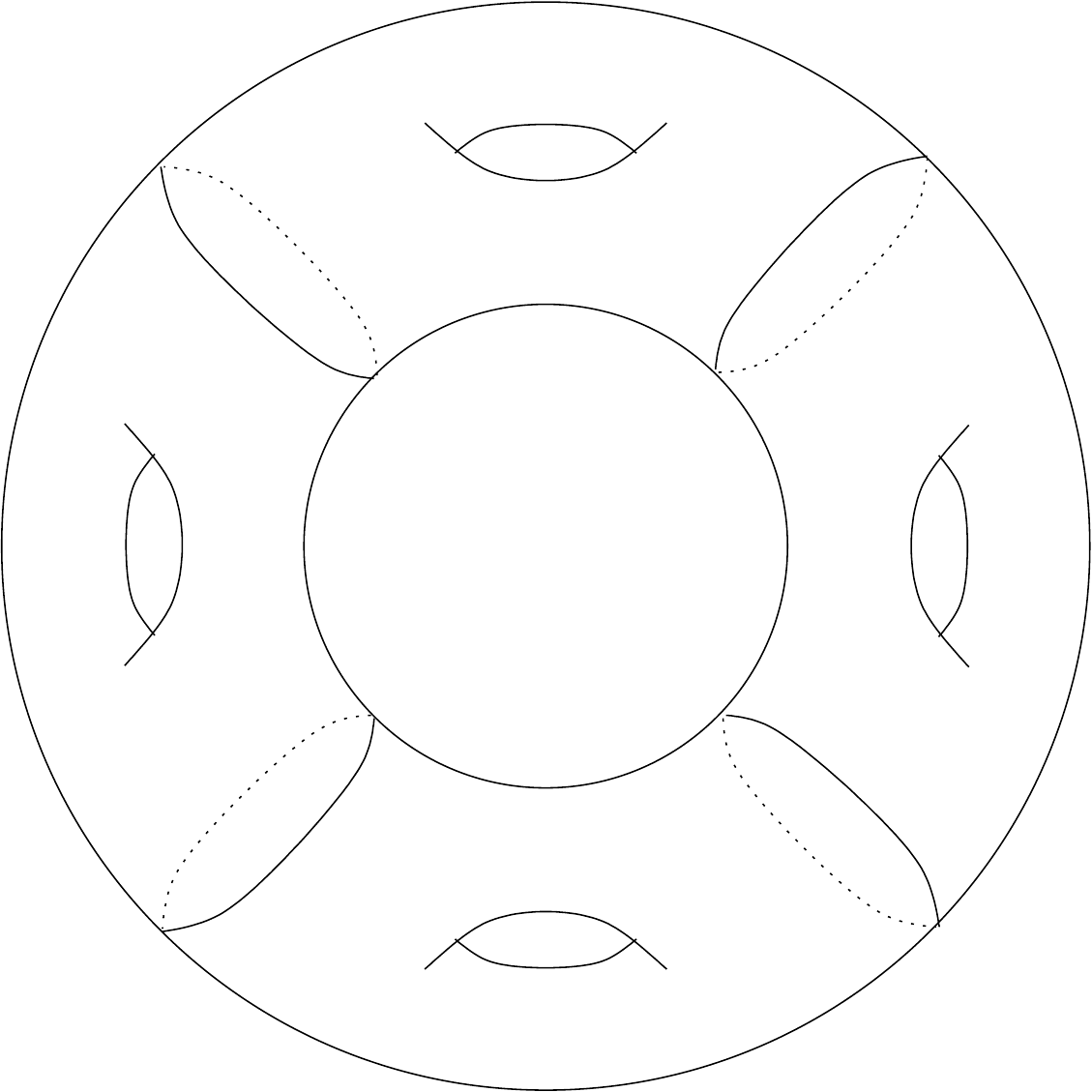} \caption{The pairwise disjoint curves $\gamma_0,\ldots,\gamma_{m-1}$ for the second family in the case of genus $5$ (and hence $m = 4$).}
\label{F : genus5example2}
\end{center}
\end{figure}

\section{Teichm\"{u}ller geodesics and active intervals}\label{sec : actinterval} 

In \cite{rshteich},\cite{rcombteich} and \cite{rteichhyper} the fourth author has developed techniques to control the length-functions and twist parameters along Teichm\"{u}ller geodesics in terms of subsurface coefficients.  In \cite{nonuniqueerg} this control was used to study the limit sets of Teichm\"{u}ller geodesics in the Thurston compactification of Teichm\"{u}ller space. Here we also appeal to this control. Most of the estimates in this section are similar to the ones in $\S 6$ of \cite{nonuniqueerg}. 

For the remainder of this section and the next we assume that $\{\gamma_k\}_{k=0}^\infty$ is a sequence of curves satisfying the condition $\mathcal{P}$ from Definition~\ref{def : P} with $a > 1$ large enough to satisfy (\ref{eq : restriction on a}) and consequently so that (\ref{eq : intgkgi}) in Theorem~\ref{thm : intgkgi} holds, and the sequence of powers $\{e_{k}\}_{k=0}^\infty$ satisfy the growth condition (\ref{eq : ek}) for this $a$.  For $h = 0,\ldots,m-1$, let $\gamma_i^h = \gamma_{im+h}$, as usual.

Let $\nu$ be the nonuniquely ergodic lamination determined by the sequence (see Theorem \ref{thm : minfill} and Corollary \ref{cor : nue}). Furthermore let $\bar{\nu}^{h}$, for $h=0,...,m-1$, be the ergodic measures from Theorems \ref{thm : MLlimitgi} and \ref{thm : all ergodic measures}, so that $\gamma_i^h \to \bar \nu^h$ in $\mathcal{PML}(S)$, for each $h$.  Let
\[ \bar{\nu}=\sum_{h=0}^{m-1}x_{h}\bar{\nu}^{h}, \]
for any $x_{h}>0$ for each $h=0,...,m-1$. 
 
Let $X\in \Teich(S)$ and $\mu$ be a short marking at $X$.  By \cite{hubmas}, there is a unique Teichm\"{u}ller geodesic ray starting at $X$ with vertical foliation $\bar{\nu}$, and we let $\bar \eta$ be the horizontal foliation (with support $\eta$). Denote the Teichm\"{u}ller geodesic ray by $r:[0,\infty)\to\Teich(S)$. For a $t\in\mathbb{R}$, we sometimes denote $r(t)=X_{t}$ and denote the quadratic differential at $X_{t}$ by $q_{t}$.  We write $v_t(\alpha),h_t(\alpha),\ell_t(\alpha)$ for the $q_t$--vertical variation, $q_t$--horizontal variation, and $q_t$--length of $\alpha$, respectively.  In particular, $v_t(\alpha) = \exp(-t)i(\alpha,\bar \eta)$, $h_t(\alpha) = \exp(t)i(\alpha,\bar \nu)$, and $\ell_t(\alpha) \stackrel * \asymp v_t(\alpha) + h_t(\alpha)$.  

We write $\hyp_t(\alpha) = \hyp_{X_t}(\alpha)$, the $X_t$--hyperbolic length of $\alpha$ and $w_t(\alpha) = w_{X_t}(\alpha)$ for the $X_t$--width, and recall from (\ref{eq : w}) that
\[ w_t(\alpha)\stackrel{+}{\asymp} 2\log\Big(\frac{1}{\hyp_t(\alpha)}\Big). \]
We also recall that $\epsilon_0 > 0$ is the Margulis constant, and that any two hyperbolic geodesics of length at most $\epsilon_0$ must be embedded and disjoint.
 
For any curve $\alpha$ let $\cyl_t(\alpha)$ be the maximal flat cylinder foliated by all geodesic representatives of $\alpha$ in the $q_t$ metric, as in \S\ref{sec : ergodic limits general case}, and let $mod(\cyl_t(\alpha))$ denote its modulus.  Fix $M>0$ sufficiently large so that for any curve $\alpha$ with $mod(\cyl_t(\alpha)) \geq M$, for some $t \in \mathbb R$, then $\hyp_t(\alpha) \leq \epsilon_0$.
For any $k\in\mathbb{N}$, let $J_{\gamma_k}$, also denoted $J_k$, be the {\it active interval of $\gamma_k$}
\[ J_k = \{t \in [0,\infty) \mid mod(\cyl_t(\gamma_k)) \geq M \}.\]

Write $J_{k}=[\underline{a}_{k},\bar{a}_{k}]$ and denote the midpoint of $J_{k}$ by $a_{k}$ (the balance time of $\gamma_k$ along the geodesic, i.e.~the unique $t$ when $v_t(\gamma_k) = h_t(\gamma_k)$). For each $h\in\{0,...,m-1\}$ and $i \geq 0$, we also write $J_{im+h}=J^{h}_{i}$, $a_i^h = a_{im+h}$, $\underline{a}^{h}_{i}=\underline{a}_{im+h}$, and $\bar{a}^{h}_{i}=\bar{a}_{im+h}$, to denote the data associated to $\gamma_i^h = \gamma_{im+h}$.

\begin{prop}\label{prop : activeint}\textnormal{(Active intervals of curves in the sequence)}
With the assumptions and notation as above, we have the following.

\begin{enumerate}[(i)]
\item For $k$ sufficiently large, $J_{k}\neq \emptyset$. Moreover $J_k \cap J_l = \emptyset$ whenever $i(\gamma_k,\gamma_l) \neq 0$.
\item For $0 \leq f < k$ sufficiently large with $k-f \geq m$, $J_f$ occurs before $J_k$.  Consequently, some tail of each subsequence $\{J_i^h\}_{i=0}^\infty$ appears in order.
\item For $k$ sufficiently large and a multiplicative constant depending only on $\nu$ and $X$
$$\hyp_{a_k}(\gamma_k)\stackrel{*}{\asymp}\frac{1}{d_{\gamma_k}(\mu,\nu)} \stackrel * \asymp \frac 1{e_k}.$$
\item   For an additive constant depending only on $\nu$, $X$, and $M$, we have \[|J_k|\stackrel{+}{\asymp}\log d_{\gamma_k}(\mu,\nu) \stackrel+\asymp \log(e_k).\]
\end{enumerate}
\end{prop} 

The following will be convenient for the proof of Proposition~\ref{prop : activeint}.
\begin{lem} \label{lem : horizontal like marking} With notation and assumptions above, there exists $k_0 \geq 0$ sufficiently large so that if $Y \subseteq S$ is a subsurface such that for some $k \geq k_0$, $d_S(\gamma_k,\partial Y) \leq 2$, then
\[ d_Y(\mu,\nu) \stackrel + \asymp_{G+1} d_Y(\eta,\nu).\]
where $G$ is the constant from Theorem~\ref{thm : bddgeod} (for a geodesic).
\end{lem}
\begin{proof}  Let $g$ be a geodesic in $\mathcal C(S)$ from (any curve in) $\mu$ limiting to $\eta$ if $\eta$ is an ending lamination, or from $\mu$ to any curve $\alpha$ with $i(\alpha,\bar \eta) =0$ otherwise.  Since $\eta$ and $\nu$ fill $S$, and $\gamma_k \to \nu \in \partial \mathcal C(S)$, the distance from $\gamma_k$ to $g$ tends to infinity with $k$.  For $Y$ and $\gamma_k$ as in the statement of the lemma, $d_S(\partial Y,\gamma_k) \leq 2$, and hence for $k$ sufficiently large, $\partial Y$ has distance at least $4$ from $g$.  Consequently, $\partial Y$ intersects every curve on $g$, and Theorem~\ref{thm : bddgeod} guarantees that $\diam_{Y}(g)\leq G$. Thus for all $\beta\in g$, $d_Y(\beta,\mu) \leq G$.  Since $g$ limits to $\eta$ (or one of it is curves is disjoint from $\eta$), it follows that $d_Y(\eta,\mu) \leq G+1$, and so the lemma follows from the triangle inequality in $\mathcal C(Y)$.
\end{proof}

\begin{proof}[Proof of Proposition~\ref{prop : activeint}.]  From \cite{rshteich}, if $d_{\gamma_k}(\eta,\nu)$ is sufficiently large, then at the balance time $a_k$, $\cyl_{a_k}(\gamma_k)$ has modulus at least $M$.  For all $k$ sufficiently large, Proposition \ref{prop : anncoeff + coeffbd}(\ref{eq : mu gi big proj}) and Lemma~\ref{lem : horizontal like marking} imply
\[ d_{\gamma_k}(\eta,\nu) \stackrel + \asymp d_{\gamma_k}(\mu,\nu) \stackrel + \asymp e_k. \]
By construction, $e_k \to \infty$ as $k \to \infty$, and hence $J_k \neq \emptyset$ for all sufficiently large $k$.
Furthermore, for all $t \in J_k$, we have $\hyp_t(\gamma_k) \leq \epsilon_0$.  Since two curves with length bounded by $\epsilon_0$ are disjoint, part (i) follows.

By Proposition \ref{prop : anncoeff + coeffbd}(\ref{eqn : gamma_i, big projection}) we have 
\[d_{\gamma_k}(\gamma_f,\gamma_l) \stackrel + \asymp e_k \]
for all $0\leq f < k < l$ with $l-k, k-f \geq m$.
Let $K \geq 0$ be such that for all $k \geq K$, $e_k > B_0$, where $B_0$ is the constant from Proposition~\ref{thm : behineq}.  Thus for all $K < f < k < l$ with $l-k,k-f \geq m$ we have
\[ d_{\gamma_f}(\gamma_k,\gamma_l) \leq B_0.\]
Since $\gamma_k \to \nu \in \partial \mathcal C(S)$, the triangle inequality in $\mathcal C(\gamma_k)$ implies that
\[ d_{\gamma_f}(\gamma_k,\nu) \stackrel + \asymp 0\]
for all $K \leq f < k$ with $k-f \geq m$.  Let $K_0 \geq K$ be sufficiently large so that if $f \geq K_0$ then $d_{\gamma_f}(\eta,\nu) \stackrel + \asymp e_f$.  Thus, for $k-f \geq m$, $f \geq K_0$, at the balance time $t=a_f$ of $\gamma_f$, the $q_t$--geodesic representative of $\gamma_k$ is more vertical than horizontal, and hence $a_f < a_k$.   By part (i), the intervals $J_f$ and $J_k$ are disjoint, so part (ii) holds.
(See also the discussion in Proposition 5.6 of \cite{rshteich}.)

For part (iv), observe that by \cite{rshteich}, the modulus of $\cyl_t(\gamma_k)$ satisfies
\begin{equation} \label{eq : modulus estimate} mod(\cyl_t(\gamma_k)) \stackrel * \asymp \frac{d_{\gamma_k}(\eta,\nu)}{\cosh^2(t-a_k)}
\end{equation}
For $k$ is sufficiently large, Lemma~\ref{lem : horizontal like marking} implies $d_{\gamma_k}(\eta,\nu) \stackrel + \asymp d_{\gamma_k}(\mu,\nu) \stackrel + \asymp e_k$.  At the endpoint $\bar{a}_k$ of $J_k$,  $mod(\cyl_{\bar{a}_k}(\gamma_k)) = M$.  Since $|J_k| = 2 (\bar{a}_k - a_k)$, we have
\[ M \stackrel * \asymp \frac{e_k}{\cosh^2(\frac12 |J_k|)}.\]
Taking logarithms we obtain $\log(e_k) - |J_k| \stackrel + \asymp \log(M)$,
proving part (iv).

We proceed to the proof of part (iii). Following Rafi in \S6 of \cite{rshteich}, we introduce the following constants associated to a curve $\alpha\in\mathcal{C}(S)$ and an essential subsurface $Y\subseteq S$ with $\alpha\subseteq\partial{Y}$ (when $Y$ is an annulus, recall that $\alpha\subseteq\partial{Y}$ means that $\alpha$ is the core curve of $Y$).
\begin{itemize}
\item If $Y$ is a non-annular subsurface, an arc $\beta$ in $Y$ is a {\em common $K$--quasi-parallel of $\pi_Y(\eta)$ and $\pi_Y(\nu)$ for $\alpha$ and $Y$} if $\beta$ transversely intersects $\alpha$ and 
$$\max\{i(\beta,\pi_Y(\eta)),i(\beta,\pi_Y(\nu))\}\leq K.$$
Here $\pi_Y(\eta)$ denotes the arc-and-curve projection of $\eta$: the union of arcs and curves obtained by intersecting $\eta$ with $Y$ (likewise for $\nu$).
Define $K(Y) = \log K$, where $K$ is the smallest number so that $\eta$ and $\nu$ have a common $K$--quasi-parallel.
\item If $Y$ is an annular subsurface, let $K(Y)=d_{Y}(\eta,\nu)$. 
\end{itemize}
Now define $K_{\alpha}$ to be the largest $K(Y)$ where $\alpha\subseteq \partial{Y}$.  Then Theorem 6.1 of \cite{rshteich} implies that $\hyp_a(\alpha)\stackrel{*}{\asymp}\frac{1}{K_{\alpha}}$, where $a$ is the balance time of $\alpha$ along the geodesic ray $r$.

In what follows we show that for all sufficiently large $k$, $K_{\gamma_k}$ is approximately equal to $e_k$. Since we will be interested in subsurfaces $Y$ with $\gamma_k \subseteq \partial Y$  (or subsurfaces of those, $Z \subseteq Y$), we can apply to Lemma~\ref{lem : horizontal like marking} deducing that $d_Y(\eta,\nu) \stackrel + \asymp d_Y(\mu,\nu)$.  We will assume that $k$ is sufficiently large for this to hold, and will use this without further mention.

First suppose $Y$ is the annulus with core curve $\gamma_k$, and observe that by Proposition~\ref{prop : anncoeff + coeffbd} and Lemma~\ref{lem : horizontal like marking}, $d_Y(\eta,\nu) \stackrel + \asymp d_{Y}(\mu,\nu) \stackrel + \asymp e_{k}$, thus $K(Y) \stackrel + \asymp e_k$.  So we consider the case that $Y$ is a non-annular subsurface with $\gamma_k \subseteq \partial Y$, and prove that for sufficiently large $k$, $K(Y) \prec e_k$.

If $Y$ contains no curves $\gamma_k$ from the sequence as essential curves, then for every subsurface $Z \subseteq Y$, by Proposition \ref{prop : anncoeff + coeffbd} and Lemma~\ref{lem : horizontal like marking} we have $d_Z(\eta,\nu) \stackrel + \asymp d_Z(\mu,\nu) \stackrel + \asymp 0$.   Then choosing the threshold $A$ in Theorem \ref{thm : i=dY} larger than the upper bound on these projections, and applying the theorem to $\pi_Y(\eta),\pi_Y(\nu)$, we see that $i(\pi_Y(\eta),\pi_Y(\nu)) \stackrel + \asymp 0$.  In this case we have $K(Y) \stackrel + \asymp 0$, and so $K(Y) \prec e_k$ for all sufficiently large $k$.

Next we suppose there are curves from our sequence contained in $Y$.  Let $\{\gamma_{l}\}_{l\in\mathcal{L}}\subseteq \{\gamma_f\}_{f=0}^\infty$ where $\mathcal{L}$ is an ordered subset of $\mathbb{N}$ which is the set of curves from our sequence which are contained in $Y$.  From (\ref{eq : int}) in Theorem \ref{thm : locglob} we see that $\mathcal{L}\subseteq\{k-m+1,...,k+m-1\}$ since any other curve in the sequence intersects $\gamma_{k}$. We proceed to find an upper bound for the factor $K(Y)$. For this purpose let $\beta \subseteq \pi_Y(\gamma_{k+m})$ be any component arc of the projection.  Then from Theorem \ref{thm : i=dY} and Lemma~\ref{lem : horizontal like marking} we have
\[ i(\beta,\pi_Y\nu)\asymp \sum_{\substack{W\subseteq Y,\\ \text{non-annular}}} \{d_{W}(\gamma_{k+m},\nu)\}_{A}+\sum_{\substack{W\subseteq Y,\\ \text{annular}}} \log\{ d_{W}(\gamma_{k+m},\nu)\}_{A}.
\]
and
\begin{eqnarray*} i(\beta,\pi_Y\eta) & \asymp & \sum_{\substack{W\subseteq Y,\\ \text{non-annular}}} \{d_{W}(\gamma_{k+m},\eta)\}_{A}+\sum_{\substack{W\subseteq Y,\\ \text{annular}}} \log\{ d_{W}(\gamma_{k+m},\eta)\}_{A}\\
& \asymp & \sum_{\substack{W\subseteq Y,\\ \text{non-annular}}} \{d_{W}(\gamma_{k+m},\mu)\}_{A}+\sum_{\substack{W\subseteq Y,\\ \text{annular}}} \log\{ d_{W}(\gamma_{k+m},\mu)\}_{A}.
\end{eqnarray*}
Choose the threshold constant $A$ from Theorem~\ref{thm : i=dY} larger than the constant $R(\mu)$ from Proposition~\ref{prop : anncoeff + coeffbd}.   Appealing to that proposition and the fact that any $l\in\mathcal{L}$ is less than $k+m$, the first of these equations implies that $i(\beta,\pi_Y\nu) \asymp 0$.  For the second set of equations, note that any $l\in\mathcal{L}$ with $\gamma_{l}\pitchfork\gamma_{k+m}$ has $l \leq k$.  Therefore, by Theorem~\ref{thm : i=dY} and the fact that $\{e_f\}$ is increasing,  we have
\begin{eqnarray*}
i(\beta,\pi_Y\mu) & \asymp & \sum_{l \in \mathcal L} \log\{d_{\gamma_l}(\gamma_{k+m},\mu) \}_A  \prec \sum_{l = k-m+1}^{k} \log(d_{\gamma_l}(\gamma_{k+m},\mu))\\
& \asymp & \sum_{l=k-m+1}^{k} \log(e_l) \prec m \log(e_k) \prec e_k.\\
\end{eqnarray*}
Therefore, $\beta$ is a $K$--quasi-parallel with $K \prec e_k$.  Consequently, $K(Y) \leq \log(K) \prec \log(e_k) \prec e_k$.  This completes the proof of part (iii), and hence the proposition.
\end{proof}

Next we list some estimates for the locations of the intervals $J^{h}_{i} \subseteq [0,\infty)$, and provide more information on the relative positions of the intervals.

\medskip

Let $h\in\{0,...,m-1\}$. From part (i) and (iv) of Proposition \ref{prop : activeint}, together with the definitions, we have that for $i$ sufficiently large
\begin{eqnarray}\underline{a}^{h}_{i}&\stackrel{+}{\asymp}& a^{h}_{i}-\frac{\log e^{h}_{i}}{2}\;\;\text{and}\label{eq : lJ}\;\; \\
\bar{a}^{h}_{i}&\stackrel{+}{\asymp}& a^{h}_{i}+\frac{\log e^{h}_{i}}{2}.\label{eq : rJ}
\end{eqnarray}
Together with these estimates, the next lemma tells us the location of the active intervals, up to an additive error.
\begin{lem} \label{lem : a midpoint location} For any $h = \{0,\ldots,m-1\}$ and $i$ sufficiently large 
\begin{equation}\label{eq : mJ}a^{h}_{i}\stackrel{+}{\asymp}\sum_{j=0}^{i-1}\log be^{h}_{j}+\frac{\log e^{h}_{i}}{2}-\frac{\log x_{h}}{2}.\end{equation} 
The additive error depends on $X$, $\gamma_0^h$, and $\nu$.
\end{lem}
\begin{proof} The proof of this lemma is similar to that of \cite[Lemma 6.3]{nonuniqueerg}, so we just sketch the proof.
Choose $i$ sufficiently large so that $J_i^h \neq \emptyset$ and $a^h_i > 0$, and so that we may estimate $i(\gamma_i^h,\mu)$ using Lemma~\ref{lem : int estimate general} (since $\mu$ is a finite set of curves).  Then appealing to the fact that $X$ is a fixed surface and $\mu$ a short marking, we have
\begin{equation} \label{eq : estimate to find balance} v_0(\gamma_i^h) \stackrel * \asymp l_0(\gamma_i^h) \stackrel * \asymp \hyp_0(\gamma_i^h) \stackrel * \asymp i(\gamma_i^h,\mu) \stackrel * \asymp A(0,h+im) = \prod_{j=0}^{i-1} b e_j^h.
\end{equation}
Since $v_t(\gamma_i^h)h_t(\gamma_i^h)$ is constant in $t$, and $v_{a_i^h}(\gamma_i^h) = h_{a_i^h}(\gamma_i^h)$, we have, for $i$ sufficiently large
\begin{eqnarray*}
v_{a_i^h}^2(\gamma_i^h) & = & v_{a_i^h}(\gamma_i^h)h_{a_i^h}(\gamma_i^h)\\
& = & v_0(\gamma_i^h) h_0(\gamma_i^h) \\
& \stackrel * \asymp & i(\gamma_i^h,\mu) i(\gamma_i^h,\bar \nu)\\
& \stackrel * \asymp & i(\gamma_i^h,\mu) \Bigl( \sum_{d =0}^{m-1} x_d i(\gamma_i^h,\bar \nu^d) \Bigr)
\end{eqnarray*}
Since $\mu$ is a fixed set of curves and $\gamma_0^h$ a fixed curve, $i(\gamma_0^h,\gamma_i^h) \stackrel * \asymp i(\mu,\gamma_i^h)$ for all $i$ sufficiently large.  Thus from (\ref{eq : ighighi+1}), for $h \neq d$, $d \in \{0,\ldots,m-1\}$, we have
\[ i(\gamma_i^h,\bar \nu^h) \stackrel * \asymp \frac{1}{i(\gamma^h_{i+1},\mu)} \mbox{ and } i(\gamma_i^h,\bar \nu^d)i(\gamma^h_{i+1},\mu) \to 0\]
The above estimates and Lemma~\ref{lem : int estimate general} imply that for $i$ sufficiently large
\[v_{a_i^h}^2(\gamma_i^h) \stackrel * \asymp x_h \frac{i(\gamma_i^h,\mu)}{i(\gamma_{i+1}^h,\mu)} \stackrel * \asymp \frac{x_h}{be_i^h}. \]
Combining this with (\ref{eq : estimate to find balance}) we have
\[ \exp(a_i^h) = \frac{v_0(\gamma_i^h)}{\exp(-a_i^h)v_0(\gamma_i^h)} = \frac{v_0(\gamma_i^h)}{v_{a_i^h}(\gamma_i^h)} \stackrel * \asymp \frac{\prod_{j=0}^{i-1} b e_j^h}{\sqrt{x_h/be_i^h}}.\]
Solving for $a_i^h$ and taking logarithms (discarding a constant $\log b$) proves (\ref{eq : mJ}), completing the proof.
\end{proof}

\begin{lem}\label{lem : af af+m}
For any $k$ sufficiently large, $\bar{a}_{k}\stackrel{+}{\asymp} \underline{a}_{k+m}$, with additive error depends on $X$, $M$, $\gamma_0^h$, and $\nu$.
\end{lem}
\begin{proof} Let $k=im+h$ where $h\in\{0,...,m-1\}$. 
From (\ref{eq : lJ}), (\ref{eq : rJ}) and (\ref{eq : mJ}) we calculate 
\begin{eqnarray*}
\underline{a}_{k+m}-\bar{a}_{k}&=&\underline{a}^{h}_{i+1}-\bar{a}^{h}_{i}\\
&\stackrel{+}{\asymp}& \sum_{j=0}^{i}\log be^{h}_{j}+\frac{\log e^{h}_{i+1}}{2}-\frac{\log x_{h}}{2}-\frac{\log e^{h}_{i+1}}{2}\\
&-&\Big( \sum_{j=0}^{i-1}\log be^{h}_{j}+\frac{\log e^{h}_{i}}{2}-\frac{\log x_{h}}{2}+\frac{\log e^{h}_{i}}{2}\Big)\\
&=&\log be^{h}_{i}-\log e^{h}_{i}=\log b.
\end{eqnarray*}
Therefore $\bar{a}_{k}\stackrel{+}{\asymp} \underline{a}_{k+m}$ since $\log b$ is a constant.
\end{proof}

Let $k,l\in\mathbb{N}$ and $0< l-k\leq m$. Suppose that $k\equiv h \mod m$ and $l\equiv d \mod m$ where $h,d\in\{0,...,m-1\}$. Then for the pair $(k,l)$ one of the following two hold: 
\begin{eqnarray}
 h&<&d \textnormal{ and $\exists i\in\mathbb{N}$, so that $k=mi+h$ and $l=mi+d$, or}\label{fl : 1} \\
 h&>&d \textnormal{ and $\exists i\in\mathbb{N}$, so that $k=mi+h$ and $l=m(i+1)+d$}\label{fl : 2}.
\end{eqnarray}
\begin{notation}
Let $\{x_{i}\}_{i=0}^{\infty}$ and $\{y_{i}\}_{i=0}^{\infty}$ be sequences of real numbers. We write $x_{i}\ll y_{i}$ if $x_{i}< y_{i}$ for all $i$ sufficiently large and $y_i - x_i \to\infty$ as $i\to\infty$. 
\end{notation}
\begin{lem}\label{lem : alaf}
For $k,l\in\mathbb{N}$ sufficiently large where $0\leq l-k <m$ the following holds:
\begin{equation}\label{eq : lara}\bar{a}_{k-m}<\underline{a}_{l}\ll\bar{a}_{k}.\end{equation}
\end{lem}
\begin{proof}
The proof is similar to the proof of Lemma 7.3 of \cite{nonuniqueerg}. For the first inequality, note that $l-(k-m) \geq m$. By Proposition~\ref{prop : activeint} parts (i) and (ii), $J_{k-m}$ occurs before $J_l$, and so we have $\bar{a}_{k-m}<\underline{a}_{l}$.

Now we show that $\underline{a}_{l}\ll \bar{a}_{k}$. If $l=k$, then since $|J_k| \to \infty$ as $k \to \infty$, we have $\underline{a}_k\ll \bar{a}_k$.  Now assume $k < l$ and let $k\equiv h \mod m$ and $l\equiv d \mod m$ with $h,d\in\{0,...,m-1\}$. First, suppose that (\ref{fl : 1}) holds so $h  < d$.  Using (\ref{eq : ek}), (\ref{eq : lJ}), (\ref{eq : rJ}) and (\ref{eq : mJ}), and the fact that $e_k \geq a^{k-f} e_f$ for $k > f$, we have
\begin{eqnarray*}
\bar{a}_{k}-\underline{a}_{l}&=&\bar{a}^{h}_{i}-\underline{a}^{d}_{i}\\
&\stackrel{+}{\asymp}&\sum_{j=0}^{i-1}\log be^{h}_{j}+\log e^{h}_{i}-\frac{1}{2}\log x_{h}-\sum_{j=0}^{i-1}\log be^{d}_{j}+\frac{1}{2}\log x_{d}\\
&=&\sum_{j=0}^{i-1}\log\frac{e^{h}_{j}}{e^{d}_{j}}+\log e^{h}_{i}+\frac{1}{2}\log \frac{x_{d}}{x_{h}}\\
&=&\sum_{j=1}^{i}\log\frac{e^{h}_{j}}{e^{d}_{j-1}} + \log e^{h}_{0}+\frac{1}{2}\log\frac{x_{d}}{x_{h}}\\
&\geq& \sum_{j=1}^{i}(m+h-d)\log a+\frac{1}{2}\log\frac{x_{d}}{x_{h}}\\
&=&i(m+h-d)\log a+\frac{1}{2}\log\frac{x_{d}}{x_{h}}.
\end{eqnarray*}
Now since $m+h-d>0$, the last term goes to $\infty$ as $i\to\infty$. 

Next suppose that (\ref{fl : 2}) holds so $h > d$.  Then we similarly have
\begin{eqnarray*}
\bar{a}_{k}-\underline{a}_{l}&=&\bar{a}^{h}_{i}-\underline{a}^{d}_{i+1}\\
&\stackrel{+}{\asymp}&\sum_{j=1}^{i-1}\log be^{h}_{j}+\log e^{h}_{i}-\sum_{j=1}^{i}\log be^{d}_{j}+\frac{1}{2}\log\frac{x_{d}}{x_{h}}\\
&=&\sum_{j=1}^{i}\log\frac{e^{h}_{j}}{e^{d}_{j}}+\frac{1}{2}\log \frac{x_{d}}{x_{h}}-\log b\\
&=&\sum_{j=1}^{i}\log\frac{e^{h}_{j}}{e^{d}_{j}}+\frac{1}{2}\log\frac{x_{d}}{x_{h}}-\log b\\
&=&i(h-d)\log a+\frac{1}{2}\log\frac{x_{d}}{x_{h}}-\log b.
\end{eqnarray*}
Now since $h-d>0$, the last term goes to $\infty$ as $i\to\infty$. 
\end{proof}

To obtain a greater control over the arrangement of intervals $J_{k}$ along the Teichm\"{u}ller geodesic ray (see Lemma \ref{lem : af+al} below) we consider the following growth conditions, in addition to (\ref{eq : ek}): 
\begin{equation}\label{eq : G1}
e_{k+1}\geq (\prod_{j=0}^{k} e_{j})^{2}.
\end{equation}
Such sequences exist simply by setting $e_0  \geq a$ and defining $e_k$ recursively, ensuring at every step that (\ref{eq : G1}) is satisfied.

Condition (\ref{eq : G1}) has the following consequence.
\begin{lem} \label{lem : case consequence of super exp growth}
Suppose a sequence $\{e_{k}\}_{k}$ satisfies (\ref{eq : ek}) and (\ref{eq : G1}). 
\begin{equation}\label{eq : G1sp}
\begin{array}{l} \text{If (\ref{fl : 1}) holds, then}\;\;
 \displaystyle{\tfrac{(e^{d}_{i})^\frac{1}{2}}{e^h_i}\prod_{j=0}^{i-1}\tfrac{e^{d}_{j}}{e^{h}_{j}} \to \infty},\text{ and}\\
\text{If (\ref{fl : 2}) holds, then}\;\; \displaystyle{(e^{d}_{i+1})^\frac{1}{2}\prod_{j=0}^{i}\tfrac{e^{d}_{j}}{e^{h}_{j}} \to \infty}.
\end{array}
\end{equation}
\end{lem}
\begin{proof}
Let $k\equiv d \mod m$ and $l\equiv h \mod m$, where $d,h\in\{0,...,m-1\}$.  First suppose that (\ref{fl : 1}) holds so $h < d$. Since $\{e_{k}\}$ is increasing (more than) exponentially fast
\[ \prod_{j=0}^{i-1} \tfrac{e^{d}_{j}}{e^{h}_{j}} \to \infty.\]
Moreover, by (\ref{eq : G1}) we have $(e^{d}_{i})^{\frac{1}{2}}\geq e^{h}_{i}$, that is, $(e^{d}_{i})^{\frac{1}{2}}/e^{h}_{i} \geq 1$.  Thus (\ref{eq : G1sp}) follows.

Now suppose that (\ref{fl : 2}) holds so $h > d$. Then 
$$(e^{d}_{i+1})^{\frac{1}{2}}\geq \prod_{j=0}^{m(i+1)+d-1} e_{j}\geq \prod_{j=0}^{i} e^{h}_{j}.$$
 where the second inequality holds because $m(i+1)+d>mi+h$.  Therefore, condition (\ref{eq : G1sp}) easily follows in this case as well.
\end{proof}

\begin{lem} \label{lem : af+al}
Suppose that the growth condition (\ref{eq : G1}) holds.  Then for $k,l\in\mathbb{N}$ sufficiently large with $0< l-k <m$ we have
\begin{eqnarray}\label{eq : raa}\bar{a}_{k}& \ll & a_{l}  \end{eqnarray}
\end{lem}
\begin{proof} Let $f\equiv h \mod m$ and $l\equiv d\mod m$ where $h,d\in\{0,...,m-1\}$. 

First suppose that (\ref{fl : 1}) holds so $h < d$.  Then from (\ref{eq : rJ}) and (\ref{eq : mJ})
 we calculate
\begin{eqnarray*}
a_{l}-\bar{a}_{k}&=&a^{d}_{i}-\bar{a}^{h}_{i}\\
&\stackrel{+}{\asymp}&\sum_{j=0}^{i-1}\log be^{d}_{j}+\frac{\log e^{d}_{i}}{2}-\frac{\log x_{d}}{2}-\left( \sum_{j=0}^{i-1}\log be^{h}_{j}+\log e^{h}_{i}-\frac{\log x_{h}}{2}\right)\\
& = & \log\Bigl(\tfrac{(e_i^d)^{\frac12}}{e_i^h}\prod_{j=0}^{i-1}\tfrac{e_j^d}{e_j^h}\Bigr)+\frac{1}{2}\log\frac{x_{h}}{x_{d}} \to \infty
\end{eqnarray*}
where the sequence tends to infinity as $i \to \infty$ by Lemma~\ref{lem : case consequence of super exp growth}.

Now suppose that (\ref{fl : 2}) holds so $h > d$.  Then we have
\begin{eqnarray*}
a_{l}-\bar{a}_{k}&=&a^{d}_{i+1}-\bar{a}^{h}_{i}\\
&\stackrel{+}{\asymp}&\sum_{j=0}^{i}\log be^{d}_{i}+\frac{\log e^{d}_{i+1}}{2}-\frac{\log x_{d}}{2}-\left(\sum_{j=0}^{i-1}\log be^{h}_{j}+\log e^{h}_{i}-\frac{\log x_{h}}{2}\right)\\
& = & \log\Bigl((e^{d}_{i+1})^\frac{1}{2}\prod_{j=0}^{i}\tfrac{e^{d}_{j}}{e^{h}_{j}}\Bigr) + \log b +\frac{1}{2}\log\frac{x_{h}}{x_{d}} \to \infty
\end{eqnarray*}
where again the convergence to infinity as $i \to \infty$ is by Lemma~\ref{lem : case consequence of super exp growth}.
\end{proof}

The following conveniently summarizes the relative positions of intervals for large indices.  See Figure~\ref{fig : actints}.
\begin{lem}\label{lem : af in Jl}
For $k < l$ sufficiently large and $l < k+m$, we have
\[ \underline{a}_k \ll \underline{a}_l \ll a_k \ll \bar a_k < \underline{a}_{k+m} \ll a_l \ll \bar a_l < \underline{a}_{l+m} \ll a_{k+m}.\]
Furthermore
\[ \bar a_k \stackrel + \asymp \underline{a}_{k+m}.\]
\end{lem}
\begin{proof}
This is immediate from Lemmas \ref{lem : af af+m}, \ref{lem : alaf} and \ref{lem : af+al}.
 \end{proof}
 
\begin{center}
\begin{figure}[htb]
\begin{tikzpicture}
\draw [line width = 2] [blue] (0,.5) -- (3,.5);
\draw [line width = 2] [blue] (3.5,.5) -- (12,.5);
\draw [line width = 2] [red] (.7,0) -- (9.7,0);
\draw [line width = 2] [red] (10.2,0) -- (12,0);
\draw [fill] (0,.5) circle [radius=.08];
\draw [fill] (1.5,.5) circle [radius=.08];
\draw [fill] (3,.5) circle [radius=.08];
\draw [fill] (3.5,.5) circle [radius=.08];
\draw [fill] (11.7,.5) circle [radius=.08];
\draw [fill] (.7,0) circle [radius=.08];
\draw [fill] (9.7,0) circle [radius=.08];
\draw [fill] (5.2,0) circle [radius=.08];
\draw [fill] (10.2,0) circle [radius=.08];
\node  at (0,.7) {$^{\underline{a}_k}$};
\node at (1.5,.7) {$^{a_k}$};
\node  at (3,.7) {$^{\bar a_k}$};
\node at (3.8,.7) {$^{\underline{a}_{k+m}}$};
\node at (11.5,.7) {$^{a_{k+m}}$};
\node at (.7,-.3) {$^{\underline{a}_l}$};
\node at (5.25,-.3) {$^{a_l}$};
\node at (9.7,-.3) {$^{\bar a_l}$};
\node at (10.5,-.3) {$^{\underline{a}_{l+m}}$};
\end{tikzpicture}
\caption{Relative positions of active intervals, $k < l < k+m$.}
\label{fig : actints}
\end{figure}
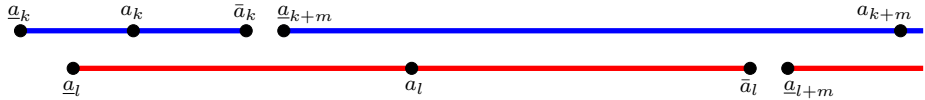
\end{center}

\section{Limit sets of Teichm\"{u}ller geodesics}\label{sec : limset}

In this section, we continue with the assumptions from the previous section on the sequences $\{\gamma_k\}_{k=0}^\infty$ and $\{e_k\}_{k=0}^\infty$ (including both (\ref{eq : ek}) and (\ref{eq : G1})), limiting lamination $\nu \in \partial C(S)$ of $\{\gamma_k\}_{k=0}^\infty$, Teichm\"uller geodesic ray $r(t) = X_t$ with quadratic differential $q_t$ at time $t \in [0,\infty)$, vertical foliations $\bar \nu = \sum_{h=0}^{m-1} x_i \bar \nu^h$ and horizontal foliation $\bar \eta$ for $(X,q) = (X_0,q_0)$, short marking $\mu$ for $X$, and active intervals $J_k = [\underline{a}_k,\bar{a}_k]$ with midpoint $a_k$.  We will also be appealing to all the estimates from the previous sections regarding this data.

In addition, we will need one more condition on $\{\gamma_k\}_{k=0}^\infty$, which we add to the properties $\mathcal P$ assumed already:
For any $k\geq 0$, let 
\[ \sigma_{k}= \gamma_k \cup \gamma_{k+1} \cup \ldots \gamma_{k+m-1}.\]
The additional condition is
\begin{enumerate}[${\mathcal P}(iv)$]
\item\label{con : P(iv)}  Let $\alpha$ be any essential curve in $S\backslash\sigma_{k}$. Then there is no subsurface $Y\subseteq S$ with $\alpha\subseteq\partial{Y}$ which is filled by a collection of the curves in the sequence $\{\gamma_{k}\}_{k=0}^\infty$.
\end{enumerate}
Recall that when $Y$ is an annular subsurface by $\alpha\subseteq\partial{Y}$ we mean that $\alpha$ is the core curve of $Y$. 

\begin{remark}
Note that when $\sigma_{k}$ is a pants decomposition of $S$ the condition $\mathcal P(iv)$ holds vacuously because there are no essential curves in $S\backslash\sigma_{k}$.  Together with the other conditions in $\mathcal P$, the new condition $\mathcal P(iv)$ is equivalent to requiring that any subsurface filled by a subset of $\{\gamma_k\}_{k=0}^\infty$ has as boundary a union of curves in $\{\gamma_k\}_{k=0}^\infty$.  According to Lemma \ref{lem : P(iv)} condition ${\mathcal P(iv)}$ holds for the sequences constructed in $\S$\ref{sec : constructions}.
\end{remark}

Under these assumptions, Theorem~~\ref{thm:main limit theorem} from the introduction, which describes the limit set of $r(t)$ in the Thurston compactification $\overline{\Teich}(S) = \Teich(S) \cup \mathcal{PML}(S)$, can be restated as follows.  Recall that the set of projective classes of measures on $\nu$ is a simplex $\Delta(\nu)$ spanned by the projective classes of the ergodic measures $[\bar \nu^0],\ldots,[\bar \nu^{m-1}]$.
\begin{thm}\label{thm : conv1-skeleton}
The accumulation set of $r(t)$ in $\mathcal{PML}(S)$ is the simple closed curve in the simplex $\Delta(\nu)$ that is the concatenation of edges 
$$\bigl[[\bar{\nu}^{0}],[\bar{\nu}^{1}] \bigr] \cup \bigl[ [\bar{\nu}^1,\bar{\nu}^2] \bigr] \cup \ldots \cup \bigl[ [\bar{\nu}^{m-1}],[\bar{\nu}^{0}] \bigr].$$
\end{thm}

We begin by reducing this theorem to a more manageable statement (Theorem~\ref{thm : reduction 1}), which also provides more information about how the sequence limits to the simple closed curve.  We then briefly sketch the idea of the proof, and describe some of the necessary estimates.  After that we reduce the theorem further to a technical version (Theorem~\ref{thm : reduction 2}), providing even more detailed information about what the limit looks like, and which allows for a more concise proof.  After supplying the final estimates necessary, we carry out the proof.

\subsection{First reduction and sketch of proof.}

By Proposition \ref{prop : activeint}, the intervals $J_k$ are nonempty for all $k$ sufficiently large.  Combining this with Lemma~\ref{lem : af+al}, it follows that for all $k< l$ sufficiently large, $\bar a_k < \bar a_l$, and that $\bar a_l \to \infty$ with $l$. Therefore, the set of intervals $[\bar a_k,\bar a_{k+1}]$ for all sufficiently large $k$, cover all but a compact subset of $[0,\infty)$, and consecutive segments intersect only in their endpoints.  Theorem~\ref{thm : conv1-skeleton} easily follows from
\begin{thm} \label{thm : reduction 1}
Fix $h,h' \in \{0,\ldots,m-1\}$ with $h' \equiv h+1$ mod $m$ and suppose $\{t_i\}$ is a sequence with $t_i \in [\bar a_{im+h},\bar a_{im+ h+ 1}]$ for all sufficiently large $i$.  Then $r(t_i) =X_{t_i}$ accumulates on the edge $\bigl[ [\bar \nu^h],[\bar \nu^{h'}] \bigr] \subset \Delta(\nu)$.  

Furthermore, if $\{t_i-\bar a_{im+h}\}$ is bounded independent of $i$, then
\[ \lim_{i \to \infty} X_{t_i} = [\bar \nu^h].\]
\end{thm}
\begin{proof}[Proof of Theorem~\ref{thm : conv1-skeleton} assuming \ref{thm : reduction 1}.]
From the second part of Theorem~\ref{thm : reduction 1} applied to $t_i = \bar a_{im+h}$, it follows that
\[ \lim_{i \to \infty} X_{\bar{a}_{im+h}} = [\bar \nu^h].\]
for all $h \in \{0,\ldots,m-1\}$.  If $h' \equiv h+1$ as in Theorem~\ref{thm : reduction 1}, then combining this with the first part of that theorem, we see that the accumulation set of the sequence of subsets $\{r([\bar a_{im+h},\bar a_{im+h+1}])\}_{i=0}^\infty \subset \Teich(S)$ is contained in $\bigl[ [\bar \nu^h],[\bar \nu^{h'}] \bigr]$ and contains the endpoints.  Consequently, any Hausdorff limit of this sequence of {\em connected} sets is a connected subset of $\bigl[ [\bar \nu^h],[\bar \nu^{h'}] \bigr]$ containing the endpoints, and hence is equal to $\bigl[ [\bar \nu^h],[\bar \nu^{h'}] \bigr]$.  The accumulation set of this sequence of sets therefore contains $\bigl[ [\bar \nu^h],[\bar \nu^{h'}] \bigr]$, and is thus equal to it.  Since this holds for every $h \in \{0,\ldots , m-1\}$, and the intervals $\{[\bar a_k,\bar a_{k+1}]\}$ cover all but a compact subset of $[0,\infty)$, this completes the proof.
\end{proof}

\begin{remark} \label{rem : simplify to h=0}
Before proceeding we note that the assumptions on $\{\gamma_k\}_{k=0}^\infty$ and $\{e_k\}_{k=0}^\infty$ are ``shift invariant'', meaning that if we start the sequence at any $k_0 \geq 0$, and reindex (without changing the order), the resulting sequence will also satisfy all the required conditions.  Consequently, it suffices to prove Theorem~\ref{thm : reduction 1} for $h =0$ and $h' = 1$.  This greatly simplifies the notation, and allows us to avoid duplicating essentially identical arguments.
\end{remark}

To sketch the proof, we recall that a sequence $\{Z_i\} \subset \Teich(S)$ converges to $[\bar \lambda] \in \mathcal{PML}(S)$ if and only if
\[ \lim_{i \to \infty} \frac{\hyp_{Z_i}(\delta)}{\hyp_{Z_i}(\delta')} = \frac{i(\bar \lambda,\delta)}{i(\bar \lambda,\delta')}\]
for all simple closed curves $\delta,\delta'$ with $i(\bar \lambda,\delta') \neq 0$; see \S \ref{sec : background}.  Thus we must provide sufficient control over the hyperbolic lengths of curves and relate these to intersection numbers with measures on $\nu$.
 
Now the idea of the proof of this theorem is as follows.  For any sufficiently large $t$, we estimate hyperbolic lengths $\hyp_{X_t}(\delta)$ in terms of ``contributions'' from the intersections of $\delta$ with the curves in a bounded length pants decomposition (Proposition~\ref{prop : hXgP}).  When $t$ is in the interval $[\bar a_k,\bar a_{k+1}]$ we choose a bounded length pants decomposition containing either $\sigma_k$ or $\sigma_{k+1}$, depending on more precise information about $t$.  The contributions from the curves in these sub-multicurves dominate the contributions from the other curves (the ratios tend to zero), and so the key is to understand these contributions.

On the active interval $J_l$, the contribution from $\gamma_l$ grows linearly in the first half of the interval (Lemma~\ref{lem : linear growth of width}), but during the second half, they speed up.  Thus near $\bar a_k$, the contribution from $\gamma_k$ will be greater than from the rest of $\sigma_k$, since $\bar a_k$ is still in the first half of $J_l$, for $l = k+1,\ldots,k+m-1$.  As we proceed far beyond $\bar a_k$, the bounded length pants decomposition eventually changes to become $\sigma_{k+1}$.  The contribution from $\gamma_k$ transitions to the contribution from $\gamma_{k+m}$ and until the contribution from $\gamma_{k+1}$ speeds up, this is the dominating term.  However, as the contribution from $\gamma_{k+1}$ speeds up, its contribution eventually takes over.  During the transition, the contribution from $\gamma_l$, for $2 \leq l \leq m-1$ is still dominated by either $\gamma_{k+m}$ or $\gamma_{k+1}$.

With this sketch in mind, we now start to discuss the details.

\subsection{General hyperbolic geometry estimates}
For a curve $\alpha \in \mathcal C(S)$ and $Z \in \Teich(S)$, we have the length and width $\hyp_Z(\alpha)$ and $w_Z(\alpha)$, respectively, as defined in \S\ref{sec : background}.  Given two curves $\alpha,\delta\in\mathcal{C}(S)$ and $Z \in \Teich(S)$ we will also need the {\em twist of $\delta$ about $\alpha$ with respect to $Z$}, denoted $\tw_{\alpha}(\delta,Z)$.  This is defined as
\[ \tw_{\alpha}(\delta,Z) = \diam_\alpha(\pi_\alpha(\delta) \cup \alpha^{\perp Z})\geq 0\]
where $\alpha^{\perp Z}$ is the set of $Z$--geodesics in the annular cover $Y_\alpha$ meeting (the lift of the geodesic representative of) $\alpha$ orthogonally.
\begin{remark} There are different definitions of $\tw_\alpha(\delta,Z)$ in the literature (see e.g.~\cite{minsky-extremal,compteichlip,lineminimateichgeod}).  Some of these come equipped with a sign which we have no need of, and our definition agrees with (the absolute values of) the other definitions, up to a uniformly bounded additive error (at least those we will be appealing to).
\end{remark}

For curves $\alpha,\delta\in\mathcal{C}(S)$ and $Z \in \Teich(S)$ define the {\em contribution to the $Z$--length of $\delta$ coming from $\alpha$} by 
\begin{equation}\label{eq : hypcollar}\hyp_{Z}(\delta,\alpha):=i(\delta,\alpha)\Big[w_{Z}(\alpha)+\tw_{\alpha}(\delta,Z)\hyp_{Z}(\alpha)\Big],\end{equation}  

The next fact, from \cite[Lemma 7.2]{lineminimateichgeod}, provides our primary means of control on hyperbolic lengths.
\begin{prop} \label{prop : hXgP}
Given $L>0$ and $Z\in\Teich(S)$, suppose that $P$ is an $L$--bounded length pants decomposition ( $\hyp_Z(\alpha) \leq L$ for all $\alpha \in P$). Then for any curve $\delta\in\mathcal{C}(S)$ we have
\begin{equation}\Big|\hyp_{Z}(\delta)-\sum_{\alpha\in P}\hyp_{Z}(\delta,\alpha)\Big|=O\Big(\sum_{\alpha\in P}i(\delta,\alpha)\Big)\end{equation}
where the constant of the $O-$notation depends only on $L$.  
\end{prop}

To effectively use this proposition to analyze lengths of curves in $X_t$ as $t \to \infty$ we must develop a better picture of the hyperbolic geometry of bounded length curves in $X_t$.

\subsection{Hyperbolic estimates for $\{\gamma_k\}$.} \label{subsec : hyp est 1}

As in \S\ref{sec : actinterval}, we will write $\hyp_t(\alpha) = \hyp_{X_t}(\alpha)$, $\hyp_t(\delta,\alpha) = \hyp_{X_t}(\delta,\alpha)$, and $w_t(\alpha) = w_{X_t}(\alpha)$.  By a result of Wolpert \cite{wolpert-length-spectra}, hyperbolic lengths change (grow/shrink) at most exponentially in Teichm\"uller distance, and hence we have
\begin{lem} \label{lem : 2 exponential growth} For any curve $\alpha$ and any $t,s \in \mathbb R$, we have
\[ \hyp_t(\alpha) \leq \exp(2(|t-s|)) \hyp_s(\alpha).\]
\end{lem}

From Lemma~\ref{lem : af in Jl}, all sufficiently large $t$ are either contained in exactly $m$ intervals $J_k,\ldots,J_{k+m-1}$ or in exactly $m-1$ intervals $J_{k+1},\ldots,J_{k+m-1}$ and the {\em bounded length} interval $[\bar a_k,\underline{a}_{k+m}]$ (the interval after $J_k$ but before $J_{k+m}$).  In the former case, every curve in $\sigma_k$ has length at most $\epsilon_0$, the Margulis constant.  In the latter case, we can use Lemma~\ref{lem : 2 exponential growth} to bound the length of curves in $\sigma_k$.  It will be useful to have a slight generalization of that, which we state here.

\begin{lem} \label{lem : bounded length pants sigmak}  For any $W > 0$, if $t$ is sufficiently large (depending on $W$), is contained in $J_{k+1},J_{k+2},\ldots,J_{k+m-1}$, and satisfies $0 < t- \bar a_k < W$, then every curve in $\sigma_k$ has $X_t$--length at most $\exp(2W) \epsilon_0$.  
\end{lem}
\begin{proof} Since $\bar a_k$ is in all the intervals $J_k,\ldots,J_{k+m-1}$, we have $\hyp_{\bar a_k}(\gamma_l) \leq \epsilon_0$ for $k \leq l \leq k+m-1$.  Now apply Lemma~\ref{lem : 2 exponential growth}.
\end{proof}
In particular, note that once $k$ is sufficiently large, Lemma~\ref{lem : af af+m} guarantees that $\underline{a}_{k+m}- \bar a_k$ is uniformly bounded by some constant $W_0$, and so setting $L_0 = \exp(2W_0)\epsilon_0$, we see that for any sufficiently large $t$, there is always some $k$ so that all curves of $\sigma_k$ have length at most $L_0$.  In addition, this gives us lower bounds on lengths as well.
\begin{lem} \label{lem : gammak bounded at under ak}  For all $k$ sufficiently large, $\hyp_{\underline{a}_k}(\gamma_k) \stackrel * \asymp 1 \stackrel * \asymp \hyp_{\bar a_k}(\gamma_k)$.
\end{lem}
The multiplicative constant here depends only on $W_0$, the constants in property $\mathcal P$, and the Margulis constant $\epsilon_0$.
\begin{proof} We already have $\hyp_{\underline{a}_k}(\gamma_k) \leq \epsilon_0$, so we need to prove a uniform lower bound.  Since $i(\gamma_k,\gamma_{k-m}) \in [b_1,b_2]$ from $\mathcal P$, and $\hyp_{\underline{a}_k}(\gamma_{k-m}) \leq L_0 = \exp(2W_0)\epsilon_0$, according to Lemma~\ref{lem : collar lemma} we have
\[ \hyp_{\underline{a}_k}(\gamma_k) \geq w_{\underline{a}_k}(\gamma_{k-m})i(\gamma_k,\gamma_{k-m}) \geq 2\sinh^{-1}(1/\sinh(L_0/2)) b_1.\]
A similar argument applies for the estimate on $\hyp_{\bar a_k}(\gamma_k)$.
\end{proof}

We will also need good estimates on $w_t(\gamma_k)$, especially on the first half of the interval when $\gamma_k$ initially becomes short.
\begin{lem} \label{lem : linear growth of width}
For all sufficiently large $k$ and $t \in [\underline{a}_k,a_k]$, we have
\[ w_t(\gamma_k) \stackrel + \asymp 4( t-\underline{a}_k).\]
\end{lem}
The implicit constant depends on the constant from Lemma~\ref{lem : gammak bounded at under ak}.
\begin{remark}
There is a mistake in \cite[Lemma 8.3]{nonuniqueerg}, which claims that the width grows at most linearly with coefficient $1$ (instead of $4$).  This does not affect any of the proofs.  It is also worth noting that only an upper bound was proved there, whereas here there are both upper and lower bounds.
\end{remark}
\begin{proof}
We first prove the upper bound on $w_t(\gamma_k)$.  For this, we note that by Lemma~\ref{lem : 2 exponential growth}
\[ 1 \stackrel * \asymp \hyp_{\underline{a}_k}(\gamma_k) \leq \exp(2(t-\underline{a}_k)) \hyp_t(\gamma_k).\]
Dividing by $\hyp_t(\gamma_k)$ and taking logarithms, we get
\[ \log\Big(\frac{1}{\hyp_t(\gamma_k)}\Big) \stackrel + \prec 2(t-\underline{a}_k). \]
Multiplying by $2$ and applying (\ref{eq : w}) proves 
\[ w_t(\gamma_k) \stackrel + \prec 4(t-\underline{a}_k).\]

For the lower bound, we will appeal to (\ref{eq : modulus estimate}), which for $k$ sufficiently large implies
\[ mod(\cyl_t(\gamma_k)) \stackrel * \asymp \frac{e_k}{\cosh^2(t-a_k)} .\]
Lifting $\cyl_t(\gamma_k)$ to the annular cover $Y_{\gamma_k}$, the modulus of the former is bounded above by the modulus of the latter by monotonicity of modulus of annuli.  The latter on the other hand can be computed explicitly as $\pi/\hyp_t(\gamma_k)$ (see e.g.~\cite{maskit}).
Thus, taking logs and noting that 
$$\log(\cosh^2(t-a_k)) \stackrel + \asymp 2|t-a_k| = 2(a_k-t),$$
 we have
\[ \log(e_k) - 2 (a_k-t) \stackrel + \prec \log\Big(\frac{\pi}{\hyp_t(\gamma_k)} \Big). \]
Then by Proposition~\ref{prop : activeint} we have $\log(e_k) \stackrel + \asymp 2(a_k - \underline{a}_k)$ and hence
\[ 2(t-\underline{a}_k) \stackrel + \prec \log\Big(\frac{1}{\hyp_t(\gamma_k)} \Big). \]
Appealing to (\ref{eq : w}) again we have
\[ 4(t-\underline{a}_k) \stackrel + \prec w_t(\gamma_k).\]
\end{proof}

We will also want to estimate $\tw_{\gamma_k}(\delta,X_t)$, for an arbitrary curve $\delta$.  This is given by the following formula from \cite{rcombteich}.

\begin{thm}\label{thm : rafitw}
Given a curve $\delta\in\mathcal{C}_{0}(S)$ and large enough $k\in\mathbb{N}$ we have
$$\tw_{\gamma_{k}}(\delta,X_{t})=
\begin{cases}
0\pm O(\frac{1}{\hyp_{X_{t}}(\gamma_{k})}),\;\; t\leq a_{k}\\
e_{k}\pm O(\frac{1}{\hyp_{X_{t}}(\gamma_{k})}),\;\; t\geq a_{k}
\end{cases}$$
\end{thm}

This Theorem shows, in particular, that the twisting is independent of $\delta$ (up to an error). In fact, arguing as in Lemma~\ref{lem : horizontal like marking}, we can easily prove that this is the case in general.
\begin{lem} \label{lem : all twists similar}  For any two curves $\delta,\delta'$ and constant $L$, there exists $T > 0$ with the following property.  If $\alpha \in \mathcal C(S)$ is a curve and $t_0 \geq T$ with $\hyp_{t_0}(\alpha) \leq L$, then for all $t$,
\[ \tw_\alpha(\delta,X_t) \stackrel{+}{\asymp}_G \tw_\alpha(\delta',X_t)\]
where $G$ is the constant from Theorem~\ref{thm : bddgeod} (for geodesics).
\end{lem}
\begin{proof}  For sufficiently large $t_0$, a curve $\alpha$ with bounded length must have bounded distance from some $\gamma_k$ in $\mathcal C(S)$.  As in the proof of Lemma~\ref{lem : horizontal like marking}, this can be assumed to be very far from the geodesic in $\mathcal C(S)$ between $\delta$ and $\delta'$ (by assuming $t_0$, and hence $k$, is very large).  Appealing to Theorem~\ref{thm : bddgeod}, we see that $d_\alpha(\delta,\delta') \leq G$.  Since $\tw_\alpha(\delta,X_t)$ is defined in terms of distance in $\mathcal C(\alpha)$, the lemma follows from the triangle inequality in $\mathcal C(\alpha)$.
\end{proof}

\subsection{Bounded length pants decompositions} \label{subsec : bdd length pants}

When $m = \xi(S)$, then for all sufficiently large times $t$, there exists $k$ so that $\sigma_k$ is a bounded length pants decomposition for $X_t$.  In this case, the estimates from the previous section then provide many of the necessary ingredients to apply Proposition~\ref{prop : hXgP} to control $\hyp_t(\delta)$, for an arbitrary curve $\delta$.

If $m < \xi(S)$, then a bounded length pants decompositions will contain other curves not in the sequence $\{\gamma_k\}$, and in this section, we describe the necessary estimates to handle the contribution to length from these.  The reader only interested in the case $m = \xi(S)$ may skip this subsection.

\medskip

We begin by bounding from below the length of the other curves in a bounded length pants decomposition.

\begin{lem} \label{lem : lower bound on sequence curves} There exists $\epsilon > 0$ depending on $R(\mu)$ from Proposition~\ref{prop : anncoeff + coeffbd} such that for all sufficiently large $t$, if $\hyp_t(\alpha) \leq \epsilon$, then $ \alpha \in \{\gamma_k\}_{k=0}^\infty$.
\end{lem}
\begin{proof} 
Let $\alpha$ be a curve not in $\{ \gamma_k\}_{k=0}^\infty$. We will show that 
$K_\alpha$ is uniformly bounded. This requires us to bound $K(Z)$ for all essential subsurfaces $Z$ with $\alpha\subseteq\partial{Z}$; see the proof of Proposition~\ref{prop : activeint} for the definition of $K_{\alpha}$ and $K(Z)$. 

By Proposition~\ref{prop : anncoeff + coeffbd} and Lemma~\ref{lem : horizontal like marking}, $d_\alpha(\eta, \nu) \leq R(\mu) + G +1$.  Consider the set of curves in $\{ \gamma_k\}_{k=0}^\infty$ that are contained in and fill an essential subsurface $Z$ with the property that $\alpha\subseteq \partial{Z}$. Then, by $\mathcal P(iv)$, this set of curves is contained in a subsurface $Y \subset Z$ such that $\alpha$ is not a boundary component of $Y$.  

Let $W \subset S - (Y \cup \alpha)$ be the (possibly disconnected) union of components meeting $\alpha$ (so two components of $\partial W$ are isotopic to $\alpha$ in $S$).  Since $W$ contains no curves in $\{\gamma_k\}$,  Proposition~\ref{prop : anncoeff + coeffbd} and Lemma~\ref{lem : horizontal like marking} imply that for all connected subsurfaces $V \subset W$, $d_V(\eta,\nu) \leq R(\mu) + G+1$.  By Theorem~\ref{thm : i=dY} $i(\pi_W(\eta),\pi_W(\nu))$ is bounded above (depending only on $R(\mu)$ and $G$).    Consequently, there exists a simple closed curve $\omega$ in $S$ intersecting $\alpha$ at most twice with $i(\pi_W(\omega),\pi_W(\eta))$ and $i(\pi_W(\omega),\pi_W(\nu))$ uniformly bounded (again depending on $R(\mu)$ and $G$).  Therefore, $i(\pi_Z(\omega),\pi_Z(\eta))$ and $i(\pi_Z(\omega),\pi_Z(\nu))$ are uniformly bounded, hence so is $K(Z)$.

According to \cite[Theorem 6.1]{rshteich}, there is uniform lower bound for $\hyp_t(\alpha)$.   The lemma is completed by setting $\epsilon> 0$ to be any number less than this uniform lower bound.
\end{proof}

In what follows, we will assume $L \geq L_0 = \exp(2W_0)\epsilon_0$ as in \S\ref{subsec : hyp est 1}.

\begin{thm}\label{thm : abds}
Let $\delta \in \mathcal C(S)$ be any curve and $L \geq L_0$.  Then there exists $K,C,T > 0$, depending on $L$, $\delta$, and $R(\mu)$ from Proposition~\ref{prop : anncoeff + coeffbd}, with the following property. Suppose $t \geq T$ and that $P$ is an $L$--bounded length pants decomposition of $S$ containing $\sigma_k$, for some $k$.  Then for all $\alpha \in P \setminus \sigma_k$, we have
\[ i(\delta,\alpha) \stackrel * \prec_K A(0,k+m-1).\]
and
\[ \tw_\alpha(\delta,X_t) \leq C.\]
\end{thm}
\begin{proof}
We first prove the bound on intersection numbers.  For any $t$, suppose $\alpha$ is part of an $L$--bounded length pants decomposition.  Then \cite[Theorem 6.1]{rteichhyper} and the triangle inequality imply that for every subsurface $Z \neq Y_\alpha$, we have
\[ d_Z(\eta,\alpha) + d_Z(\alpha,\nu) \stackrel + \asymp d_Z(\eta,\nu)\]
where the additive error depends on $S$ and $L$.

We assume that $T_0 > 0$ is large enough so that for all $t \geq T_0$ there exists $k$ so that every curve in $\sigma_k$ has length at most $L$ at time $t$.  We write $k(t)$ for such a $k$.
As in the proof of Lemma~\ref{lem : horizontal like marking} and Lemma~\ref{lem : all twists similar}, we may take $T \geq T_0$ so that for all $t \geq T$, $d_Z(\delta,\nu) \stackrel + \asymp d_Z(\eta,\nu)$ for surfaces $Z$ with $d_S(\partial Z,\gamma_{k(t)}) \leq 2$.

Now let $t \geq T$ and $P$ be an $L$--bounded length pants decomposition containing $\sigma_{k(t)}$, and let $Y$ be the component of $S \setminus \sigma_{k(t)}$ containing $\alpha$ and $Z \subseteq Y$ any subsurface.  According to Proposition~\ref{prop : anncoeff + coeffbd} and Lemma~\ref{lem : horizontal like marking} we have $d_Z(\eta,\nu) \leq R(\mu) + G + 1$, and so combining the inequalities above, there exists $R'$ (depending on $R(\mu)$ and $L$) so that for all surfaces $Z \subseteq Y$, we have
\[ d_Z(\delta,\alpha) \leq R'.\]
Therefore, taking the threshold sufficiently large in Theorem~\ref{thm : i=dY} for the subsurface $Y$, there exists a constant $I$ (depending on $R'$ and Theorem~\ref{thm : i=dY}) so that
\[ i(\pi_Y(\delta),\alpha) \leq I.\]
Now, every arc of $\pi_Y(\delta)$ comes from a pair of intersection points with curves in $\sigma_{k(t)}$.  Consequently, taking $\kappa(\delta)$ to be the constant from Lemma~\ref{lem : int estimate general} we have
\[ i(\delta,\alpha) \leq I \sum_{d = k(t)}^{k(t)+m-1} i(\delta,\gamma_d) \stackrel * \asymp_{\kappa(\delta)} I \sum_{d = k(t)}^{k(t)+m-1}A(0,d) \leq m I A(0,k(t)+m-1).\]
Thus, setting $K = mI \kappa(\delta)$ proves the first statement. 

For the bound on twist number, we again appeal to \cite{rcombteich}---the same estimate in Theorem~\ref{thm : rafitw}.  Since $\alpha \not \in \{\gamma_k\}_{k=0}^\infty$ (and $\alpha$ has bounded length at time $t \geq T$), we have $d_\alpha(\eta,\nu) \leq R(\mu) + G + 1$, where $R(\mu)$ is from Proposition~\ref{prop : anncoeff + coeffbd} and $G$ the constant appearing in Lemma~\ref{lem : horizontal like marking} (from Theorem~\ref{thm : bddgeod}).  Since the length of $\alpha$ is bounded below by $\epsilon$, according to Lemma~\ref{lem : lower bound on sequence curves}, \cite{rcombteich} implies
\[ \tw_\alpha(\delta,X_t) \leq C\]
for some $C >0$ depending on $R(\mu),G,\epsilon$ and the surface $S$.
\end{proof}

\subsection{Second reduction and division into cases}

We now consider the setup as in Theorem~\ref{thm : reduction 1}.  As mentioned in Remark~\ref{rem : simplify to h=0}, to simplify the notation we assume $h =0$ and $h'=1$.  It is convenient to switch to the notation $\gamma_i^h = \gamma_{im+h}$, $a_i^h = a_{im+h}$, $\sigma_i^h = \sigma_{im+h}$, etc.  

We consider sequences $\{t_i\}$ with $t_i \in [\bar{a}_i^0,\bar{a}_i^1]$ for all sufficiently large $i$, falling into one of two possible cases:
\[ \begin{array}{l} \mbox{{\bf Case 1.} There exists $W > 0$ so that $t_i \in [\bar{a}_i^0,\bar{a}_i^0+W]$.}\\\\
\mbox{{\bf Case 2.}$\displaystyle{\lim_{i \to \infty} t_i-\bar{a}_i^0 = \infty}$.}
\end{array}\]
For any curve $\delta \in \mathcal C(S)$ define
\[ U_i^h(t,\delta) = w_t(\gamma_i^h) + \tw_{\gamma_i^h}(\delta,X_t)\hyp_t(\gamma_i^h).\]
We will also fix a curve $\delta_0$ for reference and write
\[ U_i^h(t) = U_i^h(t,\delta_0). \]
The next lemma is not needed for the reduction, but for later use we make note of it now.
\begin{lem} \label{lem : can ignore delta} For any curve $\delta \in \mathcal C(S)$ and $L >0$, there exists $T >0$ so that for all $t \geq T$ and $i,h$ with $\hyp_t(\gamma_i^h) \leq L$, we have
\[ U_i^h(t,\delta) \stackrel + \asymp_{GL} U_i^h(t). \]
\end{lem}
Here the constant $G$ is from Theorem~\ref{thm : bddgeod} appearing in Lemma~\ref{lem : all twists similar}.
\begin{proof} Given $L$, Lemma~\ref{lem : all twists similar} provides $T > 0$ so that for all $t \geq T$, if $\hyp_t(\gamma_i^h) \leq L$, then
\[ |\tw_{\gamma_i^h}(\delta,X_t) - \tw_{\gamma_i^h}(\delta_0,X_t)| \leq G.\]
Therefore, we have
\[ |U_i^h(t,\delta) - U_i^h(t)| = |\tw_{\gamma_i^h}(\delta,X_t) - \tw_{\gamma_i^h}(\delta_0,X_t)|\hyp_t(\gamma_i^h) \leq GL.\]
\end{proof}

We now turn to our second reduction.
\begin{thm} \label{thm : reduction 2}
Suppose $\{t_i\}$ is a sequence with $t_i \in [\bar{a}_i^0,\bar{a}_i^1]$ for all sufficiently large $i$ and $\delta$ is any curve (not necessarily $\delta_0$).

If $\{t_i\}$ falls into Case 1, then
\[ \lim_{i \to \infty} \frac{U_i^0(t_i)i(\delta,\gamma_i^0)}{\hyp_{t_i}(\delta)} = 1.\]

If $\{t_i\}$ falls into Case 2, then
\[ \lim_{i \to \infty} \frac{U_i^1(t_i)i(\delta,\gamma_i^1) + U_{i+1}^0(t_i)i(\delta,\gamma_{i+1}^0)}{\hyp_{t_i}(\delta)}=1.\]
\end{thm}
Note in this theorem, the terms $U_j^h(t_i)$ do not depend on $\delta$ (c.f.~Lemma~\ref{lem : can ignore delta}).
\begin{proof}[Proof of Theorem~\ref{thm : reduction 1} assuming Theorem~\ref{thm : reduction 2}.]  Suppose $\{t_j\}_{j=0}^\infty$ with $t_j \in [\bar{a}_{i_j}^0,\bar{a}_{i_j}^1]$ for all sufficiently large $j$ and some $i_j$, so that $X_{t_j}$ converges to some point in $\mathcal{PML}(S)$.  We may pass to a subsequence so that either $t_j - \bar a_{i_j}^0 \leq W$ for some $W$, or else $t_j - \bar a_{i_j}^0 \to \infty$ with $j$.  This subsequence can be viewed as a subsequence of a sequence falling into Case 1 or Case 2, respectively, and hence the conclusion of Theorem~\ref{thm : reduction 2} holds for $\{t_j\}$.

Now let $\delta,\delta' \in \mathcal C(S)$ be any two curves.  If we are in Case 1, then by Theorem~\ref{thm : reduction 2} and Theorem~\ref{thm : MLlimitgi} we have
\begin{eqnarray*} \lim_{j \to \infty} \frac{\hyp_{t_j}(\delta)}{\hyp_{t_j}(\delta')} & = & \lim_{j \to \infty} \frac{\hyp_{t_j}(\delta) \frac{U_{i_j}^0(t_j)i(\delta,\gamma_{i_j}^0)}{\hyp_{t_j}(\delta)}}{\hyp_{t_j}(\delta') \frac{U_{i_j}^0(t_j)i(\delta',\gamma_{i_j}^0)}{\hyp_{t_j}(\delta')}}\\\\  
& = & \lim_{j \to \infty} \frac{i(\delta,\gamma_{i_j}^0)}{i(\delta',\gamma_{i_j}^0)} = \frac{i(\delta,\bar \nu^0)}{i(\delta',\bar \nu^0)}.
\end{eqnarray*}
Since $\delta,\delta'$ were arbitrary, it follows that $X_{t_j} \to [\bar \nu^0]$.

Now suppose we are in the second case.  Compactness of $\mathcal{PML}(S)$ implies that by passing to a further subsequence (of the same name) the sequence
\[ \{[U_{i_j}^1(t_j)\gamma_{i_j}^1 + U_{i_j+1}^0(t_j)\gamma_{i_j+1}^0] \}_{j=0}^\infty \]
converges in $\PML(S)$.  Note that this limit is necessarily of the form
\[ [y_0\bar \nu^0 + y_1 \bar \nu^1] \in \Big[ [\bar \nu^0],[\bar \nu^1] \Big] \]
by Theorem~\ref{thm : MLlimitgi}.  Now observe that for all $j$, the numerator from Case 2 of Theorem~\ref{thm : reduction 2} is given by
\[ U_{i_j}^1(t_j)i(\delta,\gamma_{i_j}^1) + U_{i_j+1}^0(t_j)i(\delta,\gamma_{i_j+1}^0)  = i(\delta,U_{i_j}^1(t_j)\gamma_{i_j}^1 + U_{i_j+1}^0(t_j)\gamma_{i_j+1}^0).\]
Therefore, similar to the above calculation, appealing to Theorem~\ref{thm : reduction 2} we have
\begin{eqnarray*}
\lim_{j \to \infty} \frac{\hyp_{t_j}(\delta)}{\hyp_{t_j}(\delta')} & = &\lim_{j\to\infty} \frac{i(\delta,U_{i_j}^1(t_j)\gamma_{i_j}^1 + U_{i_j+1}^0(t_j)\gamma_{i_j+1}^0)}{i(\delta',U_{i_j}^1(t_j)\gamma_{i_j}^1 + U_{i_j+1}^0(t_j)\gamma_{i_j+1}^0)}\\\\
& = & \frac{i(\delta,y_0\bar \nu^0 + y_1 \bar \nu^1)}{i(\delta',y_0\bar \nu^0 + y_1 \bar \nu^1)}
\end{eqnarray*}
Again, because $\delta,\delta'$ were arbitrary we see that $X_{t_j}$ limits to $[y_0 \bar \nu^0+y_1\bar \nu^1]$.
This completes the proof.
\end{proof}

\subsection{Final estimates and proof of Theorem~\ref{thm : reduction 2}.} \label{subsec : final estimates}

Here we provide the final estimates necessary for the proof of Theorem~\ref{thm : reduction 2} (and hence the main theorem).  The proof for each of the two cases are similar, and many of the estimates can be made simultaneously.  

We assume for the remainder of the paper that $\{t_i\}$ is a sequence so that $t_i \in [\bar a_i^0,\bar a_i^1]$ for all sufficiently large $i$ and that $\delta$ is an arbitrary curve (not necessarily our reference curve $\delta_0$).

If we are in Case 1 with $t_i-\bar a_i^0 \leq W$, then by Lemma~\ref{lem : bounded length pants sigmak}, for all sufficiently large $i$ there exists $L \geq \exp(2W)\epsilon$ and an $L$--bounded length pants decomposition $P_i$ for $X_{t_i}$ containing $\sigma_i^0$.  Let
\[  P^c_i = \ P_i \setminus \sigma_i^0. \]

If we are in Case 2, then by Lemma~\ref{lem : af af+m}, for $i$ sufficiently large, we have $t_i \in [\underline{a}_{i+1}^0,a_{i+1}^0]$, and there exists $L \geq 0$ (depending only on $S$) and an  $L$--bounded pants decomposition $P_i$ for $X_{t_i}$ containing $\sigma_i^1$.  Similar to Case 1, we let
\[  P^c_i =  P_i \setminus \sigma_i^1.\]

We use Proposition~\ref{prop : hXgP} to estimate $\hyp_{t_i}(\delta)$.  Appealing to Theorem~\ref{thm : abds} together with Lemma~\ref{lem : int estimate general} and monotonicity of $\{A(0,k)\}_{k=0}^\infty$ (Lemma~\ref{lem : A ratio bound}) to group together all the intersection number errors in Proposition~\ref{prop : hXgP}, this takes a somewhat simpler form.  To write it, recall that for all $h \in \{0,\ldots,m-1\}$ and $i \geq 0$, we have
\[ c_i^h = A(0,im+h) = \prod_{j=0}^{i-1} be_j^h.\]
The estimates are then similar, but depend on the case:\\

\noindent {\bf Case 1.}
\begin{equation} \label{eq : Case 1a}
\hyp_{t_i}(\delta) = \sum_{h=0}^{m-1} \hyp_{t_i}(\delta,\gamma_i^h) + \sum_{\alpha \in  P^c_i} \hyp_{t_i}(\delta,\alpha) + O(c^{m-1}_i).
\end{equation}
The $O$-error term depends on $L$ (hence $W$) and $\delta$, but is independent of $i$.\\

\noindent {\bf Case 2.}
\begin{equation} \label{eq : Case 2a} \begin{array}{l} \displaystyle{ \hyp_{t_i}(\delta) = \sum_{h=1}^{m-1} \hyp_{t_i}(\delta,\gamma_i^h) + \hyp_{t_i}(\delta,\gamma_{i+1}^0)}\\
\displaystyle{ \hspace{4cm} + \sum_{\alpha \in P^c_i} \hyp_{t_i}(\delta,\alpha) + O(c^0_{i+1}).}
\end{array}
\end{equation}
In this case, the $O$--error term depends on $L$ (which depends only on $S$) and $\delta$, but is again independent of $i$.

We will appeal to the various estimates previously made, specifically those in \S\ref{sec : actinterval}, \S\ref{subsec : hyp est 1}, and \S\ref{subsec : bdd length pants}.  
The first estimate involves the contributions to (\ref{eq : Case 1a}) and (\ref{eq : Case 2a}) from the curves of $ P^c_i$.

\begin{lem} For all $i$ sufficiently large and $\alpha \in P^c_i$, we have 
\[ \hyp_{t_i}(\delta,\alpha) = \left\{ \begin{array}{ll}  O(c_i^{m-1}) & \mbox{ in Case 1}\\
O(c_{i+1}^m) & \mbox{ in Case 2.}  \end{array}  \right.\]
Here the implicit constant in the $O$--notation depends on $\delta$.
\end{lem}
\begin{proof}
From (\ref{eq : hypcollar}) we have
\[ \hyp_{t_i}(\delta,\alpha)  = \left(w_{t_i}(\alpha) + \tw_{\alpha}(\delta,X_{t_i})\hyp_{t_i}(\alpha)\right) i(\delta,\alpha).\]
By Lemma~\ref{lem : lower bound on sequence curves} and Theorem~\ref{thm : abds}, every term on the right except $i(\delta,\alpha)$ is bounded, depending on $\delta$ and $L$ (and the resulting constants from those statements).  The lemma follows.
\end{proof}
\begin{cor} \label{cor : reduced to sigmas}
For all $i$ sufficiently large we have
\begin{eqnarray}
\label{eq : final hyp est case 1} & & \hyp_{t_i}(\delta) = \sum_{h=0}^{m-1} \hyp_{t_i}(\delta,\gamma_i^h) +  O(c^{m-1}_i) \hspace{2.6cm} \mbox{ in Case 1}\\
\label{eq : final hyp est case 2} & & \hyp_{t_i}(\delta) = \sum_{h=1}^{m-1} \hyp_{t_i}(\delta,\gamma_i^h) + \hyp_{t_i}(\delta,\gamma_{i+1}^0) + O(c^0_{i+1}) \mbox{ in Case 2}.
\end{eqnarray}
\end{cor}

We write the remaining terms using the notation set in the previous section as
\[ \hyp_{t_i}(\delta,\gamma_j^h) = U_j^h(t_i, \delta) i(\delta,\gamma_j^h). \]
Estimates for these terms are given in the next four lemmas.

\begin{lem} \label{lem : Uih}
For all sufficiently large $i$ and all $1 < h \leq m-1$, we have
\[ U_i^h(t_i,\delta) \stackrel + \asymp 4\Big(\sum_{j=1}^i\log\big(\tfrac{e_j^0}{e_{j-1}^h}\big) + t_i - \bar a_i^0 \Big).  \]
In Case 1, this also holds for $h =1$.
\end{lem}
\begin{proof}
Note that for $1 < h \leq m-1$ (as well as $h =1$ in Case 1), we have $\underline{a}_i^h < t_i < a_i^h$, for all sufficiently large $i$.  Therefore, $\hyp_t(\gamma_i^h) \leq \epsilon_0 < L$ and so Theorem~\ref{thm : rafitw} implies
\[ \tw_{\gamma_i^h}(\delta,X_{t_i})\hyp_{t_i}(\gamma_i^h) \stackrel * \prec 1. \]
On the other hand by Lemma~\ref{lem : linear growth of width}
\[ w_{t_i}(\gamma_i^h) \stackrel + \asymp 4(t_i- \underline{a}_i^h) =4(\bar a_i^0 - \underline{a}_i^h + (t- \bar a_i^0)), \]
since $\underline{a}_i^h \ll \bar a_i^0 \leq t_i \leq a_i^h$ (for sufficiently large $i$ and all $1 <  h \leq m-1$ in both cases, and also $h=1$ in Case 1) by Lemma~\ref{lem : af in Jl}.  
The lemma now follows from this by substituting in from (\ref{eq : lJ}), (\ref{eq : rJ}), and (\ref{eq : mJ}) and dropping constants.
\end{proof}

\begin{lem} \label{lem : Ui0}
Suppose $\{t_i\}$ falls into Case 1 with constant $W$.  Then for all sufficiently large $i$, we have
\[ U_i^0(t_i,\delta) \stackrel * \asymp e_i^0\]
where the multiplicative error depends on $W$, $\delta$, (and all resulting constants), but not $i$.
\end{lem}
\begin{proof}
Because $t_i - \bar a_i^0 \leq W$, $\hyp_{t_i}(\gamma_i^0)$ is bounded above and below by 
Lemma~\ref{lem : gammak bounded at under ak} and Lemma~\ref{lem : 2 exponential growth},  the bound depending on $W$.  By Lemma~\ref{lem : collar lemma}, $w_{t_i}(\gamma_i^0)$ is also bounded.  To complete the proof we note that by Theorem~\ref{thm : rafitw} 
\[ \tw_{\gamma_i^0}(\delta,t_i) \stackrel * \asymp e_i^0.\]
\end{proof}

\begin{lem} \label{lem : Ui+10} Suppose $\{t_i\}$ falls into Case 2.  Then for all large $i$, we have
\[ U_{i+1}^0(t_i,\delta) \stackrel + \asymp 4(t - \bar a_i^0). \]
\end{lem}
\begin{proof}
This is almost identical to the proof of Lemma~\ref{lem : Uih}, so we omit it.\end{proof}

For the only remaining situation, a very coarse estimate will suffice.
\begin{lem} \label{lem : Ui1}
Suppose $\{t_i\}$ falls into Case 2.  Then
\[ U_i^1(t_i,\delta) \to \infty.\]
\end{lem}
\begin{proof} Since we are in Case 2,  $t_i - \underline{a}_i^1 \geq t_i-\bar a_i^0 \to \infty$.  Then either $t_i \leq a_i^1$ or $a_i^1 \leq t_i \leq \bar a_i^1$.  In the former case, Lemma~\ref{lem : linear growth of width} shows that $w_{t_i}(\gamma_i^1) \to \infty$.  In the latter case, either $w_{t_i}(\gamma_i^1) \to \infty$, and we are done, or else $w_{t_i}(\gamma_i^1)$ is bounded.  If $w_{t_i}(\gamma_i^1)$ is bounded, then (\ref{eq : w}) implies $\hyp_{t_i}(\gamma_i^1)$ is bounded below.  Since $e_i^1 \to \infty$, Theorem~\ref{thm : rafitw} implies that $\tw_{\gamma_i^1}(\delta,\gamma_i^1) \to \infty$, completing the proof.
\end{proof} 
From these, we deduce the following
\begin{cor}
If $\{t_i\}$ falls into Case 1 (and hence $t_i-\bar a_i^0 \leq W$), then for all $i$ sufficiently large and $1 \leq h \leq m-1$ we have
\begin{eqnarray}
\label{eq : case 1 mult est 1} \hyp_{t_i}(\delta,\gamma_i^h) & \stackrel * \asymp & \Big(\sum_{j=1}^i\log\big(\tfrac{e_j^0}{e_{j-1}^h} \big)\Big) \prod_{j=0}^{i-1} be_j^h, \mbox{ and }\\
\label{eq : case 1 mult est 2} \hyp_{t_i}(\delta,\gamma_i^0) & \stackrel * \asymp & \prod_{j=0}^i be_j^0
\end{eqnarray}
If $\{t_i\}$ falls into Case 2 (and hence $t_i - \bar a_i^0 \to \infty$), then for all $i$ sufficiently large and $2 \leq h \leq m-1$ we have
\begin{eqnarray}
\label{eq : case 2 mult est 1} \hyp_{t_i}(\delta,\gamma_i^h) & \stackrel * \asymp &\Big(\sum_{j=1}^i\log\big(\tfrac{e_j^0}{e_{j-1}^h}\big) + t_i - \bar a_i^0 \Big) \prod_{j=0}^{i-1} be_j^h, \mbox{ and }\\
\label{eq : case 2 mult est 2} \hyp_{t_i}(\delta,\gamma_{i+1}^0) & \stackrel * \asymp & (t_i - \bar a_i^0)\prod_{j=0}^i be_j^0
\end{eqnarray}
\end{cor}
The multiplicative constants depend on $W$ (in Case 1) and $\delta$, and all constants that depend on these.
\begin{proof}  By Lemma~\ref{lem : int estimate general}, there exists $\kappa(\delta) > 0$ so that
\[ i(\delta,\gamma_i^h) \stackrel * \asymp_{\kappa(\delta)} A(0,im+h) = c_i^h = \prod_{j=0}^{i-1} be_i^h.\]
Since
\[ \hyp_{t_i}(\delta,\gamma_j^h) = U_j^h(t_i,\delta) i(\delta,\gamma_j^h),\]
the corollary follows from Lemmas~\ref{lem : Uih}, \ref{lem : Ui0}, \ref{lem : Ui+10}, and \ref{lem : Ui1}.
\end{proof}

We are now ready for the
\begin{proof}[Proof of Theorem~\ref{thm : reduction 2}]  Observe that from Lemmas~\ref{lem : Uih}, \ref{lem : Ui0}, \ref{lem : Ui+10}, and \ref{lem : Ui1}, we see that for all $h$, as $i \to \infty$ we have
\[ U_i^h(\delta,t_i) \to \infty \mbox{ and } U_{i+1}^0(\delta,t_i) \to \infty.\]
where the second limit is only true in Case 2, and the first is only relevant for $h=0$ in Case 1.
By Lemma~\ref{lem : can ignore delta}, it suffices to prove Theorem~\ref{thm : reduction 2} replacing all terms of the form $U_j^h(t_i)$ with terms $U_j^h(t_i,\delta)$.

The proof will use the estimates (\ref{eq : final hyp est case 1}) and (\ref{eq : final hyp est case 2}) from Corollary~\ref{cor : reduced to sigmas} and we divide it into the two cases.\\

\noindent
{\bf Proof in Case 1.} We look at each term on the right-hand side of (\ref{eq : final hyp est case 1}) and divide by the term $\hyp_{t_i}(\delta,\gamma_i^0)$.  Doing this for the terms $\hyp_{t_i}(\delta,\gamma_i^h)$ for $1 \leq h \leq m-1$, Equations (\ref{eq : case 1 mult est 1}) and (\ref{eq : case 1 mult est 2}) imply
\[ \frac{\hyp_{t_i}(\delta,\gamma_i^h)}{\hyp_{t_i}(\delta,\gamma_i^0)} \stackrel * \asymp be_0^h \Big(\sum_{j=1}^i\log\big(\tfrac{e_j^0}{e_{j-1}^h}\big) \Big)  \prod_{j=1}^i \tfrac{e_{j-1}^h}{e_j^0} =  \log\Big(\prod_{j=1}^i \tfrac{e_j^0}{e_{j-1}^h}\Big)\prod_{j=1}^i \tfrac{e_{j-1}^h}{e_j^0}.\]
Since $jm > (j-1)m+h$ implies $e_j^0 \geq a e_{j-1}^h$, we have $\prod \tfrac{e_{j-1}^h}{e_j^0 }\leq a^{-i}$, and since $a >1$
\[ \lim_{i \to \infty} \frac{\hyp_{t_i}(\delta,\gamma_i^h)}{\hyp_{t_i}(\delta,\gamma_i^0)} = 0.\]
The only remaining term, other than $\hyp_{t_i}(\delta,\gamma_i^0)$, is $O(c_i^{m-1})$.  For this, we note that by definition
\[ c_i^h = \prod_{j=0}^{i-1} be_j^h,\]
and therefore, for the same reason as above, we have
\[ \frac{O(c_i^h)}{\hyp_{t_i}(\gamma_i^0)} \stackrel * \asymp be_0^h \prod_{j=1}^i \tfrac{e_{j-1}^h}{e_j^0} \to 0 \]
as $i \to \infty$.
Now combining all these estimates into (\ref{eq : final hyp est case 1}) we have
\[ \lim_{i \to \infty} \frac{\hyp_{t_i}(\delta)}{\hyp_{t_i}(\delta,\gamma_i^0)} = \lim_{i \to \infty} \sum_{h=0}^{m-1} \frac{\hyp_{t_i}(\delta,\gamma_i^h)}{\hyp_{t_i}(\delta,\gamma_i^0)} +  \tfrac{O(c^{m-1}_i)}{\hyp_{t_i}(\delta,\gamma_i^0)} = 1 \]
This completes the proof since
\[ \hyp_{t_i}(\delta,\gamma_i^0) = U_i^0(t_i,\delta)i(\delta,\gamma_i^0).\]

\medskip

\noindent
{\bf Proof in Case 2.}  We again look at each term on the right-hand side of (\ref{eq : final hyp est case 2}) and this time begin by dividing most of the terms by $\hyp_{t_i}(\delta,\gamma_{i+1}^0)$.  Doing this for the terms $\hyp_{t_i}(\delta,\gamma_i^h)$ for $2 \leq h \leq m-1$, Equations (\ref{eq : case 2 mult est 1}) and (\ref{eq : case 2 mult est 2}), together with the fact that $t_i - \bar a_i^0 \to \infty$, imply
\begin{eqnarray*}
\frac{\hyp_{t_i}(\delta,\gamma_i^h)}{\hyp_{t_i}(\delta,\gamma_{i+1}^0)} & \stackrel * \asymp &  \tfrac{be_0^h}{t_i-\bar a_i^0}\Big(\sum_{j=1}^i\log\big(\tfrac{e_j^0}{e_{j-1}^h} \big) + t_i-\bar a_i^0 \Big)  \prod_{j=1}^i \tfrac{e_{j-1}^h}{e_j^0}\\
& \leq & be_0^h\Big(1+ \log\Big( \prod_{j=1}^i \tfrac{e_j^0}{e_{j-1}^h}\Big) \Big)\prod_{j=1}^i \tfrac{e_{j-1}^h}{e_j^0}.
\end{eqnarray*}
Now as above, the right-hand side tends to $0$ as $i \to \infty$, and hence
\[ \lim_{i \to \infty} \frac{\hyp_{t_i}(\delta,\gamma_i^h)}{\hyp_{t_i}(\delta,\gamma_{i+1}^0)} = 0\]

Next we consider the $O(c_{i+1}^0)$ term of (\ref{eq : final hyp est case 2}).  By definition of $c_{i+1}^0$, together with (\ref{eq : case 2 mult est 2}) and the fact that $t_i - \bar a_i^0 \to \infty$, as $i \to \infty$ we have
\[  \frac{O(c_{i+1}^0)}{\hyp_{t_i}(\delta,\gamma_{i+1}^0)} \stackrel * \asymp \frac{\prod_{j=0}^i be_j^0}{(t_i-\bar a_i^0)\prod_{j=0}^i be_j^0} = \frac{1}{t_i-\bar a_i^0} \to 0.\]

Since $\hyp_{t_i}(\delta,\gamma_i^1) + \hyp_{t_i}(\delta,\gamma_{i+1}^0) > \hyp_{t_i}(\delta,\gamma_{i+1}^0)$, we could have divided by this larger quantity, and the above limits would still be zero.  Plugging into (\ref{eq : final hyp est case 2}) we deduce
\[ \lim_{i \to \infty} \frac{\hyp_{t_i}(\delta)}{\hyp_{t_i}(\delta,\gamma_i^1) + \hyp_{t_i}(\delta,\gamma_{i+1}^0)} = 1.\]
Since
\[ \hyp_{t_i}(\delta,\gamma_i^1) + \hyp_{t_i}(\delta,\gamma_{i+1}^0) = U_i^1(t_i,\delta)i(\delta,\gamma_i^1) + U_{i+1}^0(t_i,\delta)i(\delta,\gamma_{i+1}^0)\]
this completes the proof of Case 2, and hence of the theorem.
\end{proof}

\bibliographystyle{amsalpha}
\bibliography{reference}
\end{document}